%% file: main.tex
\documentclass[10pt, a4paper,
oneside, headinclude,footinclude]{scrartcl}
\usepackage[utf8]{inputenc}

\input{structure.tex} 

\title{\normalfont\spacedallcaps{On rectifiable measures in Carnot groups: representation}} 
\author{\spacedlowsmallcaps{Gioacchino Antonelli\textsuperscript{*} and Andrea Merlo\textsuperscript{**}}}

\date{}

\begin{document}

\renewcommand{\sectionmark}[1]{\markright{\spacedlowsmallcaps{#1}}} 
\lehead{\mbox{\llap{\small\thepage\kern1em\color{halfgray} \vline}\color{halfgray}\hspace{0.5em}\rightmark\hfil}} 
\pagestyle{scrheadings}
\maketitle 
\setcounter{tocdepth}{2}
\paragraph*{Abstract}
This paper deals with the theory of rectifiability in arbitrary Carnot groups, and in particular with the study of the notion of $\mathscr{P}$-rectifiable measure. 

First, we show that in arbitrary Carnot groups the natural \textit{infinitesimal} definition of rectifiabile measure, i.e., the definition given in terms of the existence of \textit{flat} tangent measures, is equivalent to the global definition given in terms of coverings with intrinsically differentiable graphs, i.e., graphs with \textit{flat} Hausdorff tangents. In general we do not have the latter equivalence if we ask the covering to be made of intrinsically Lipschitz graphs.

Second, we show a geometric area formula for the centered Hausdorff measure restricted to intrinsically differentiable graphs in arbitrary Carnot groups. The latter formula extends and strengthens other area formulae obtained in the literature in the context of Carnot groups.

As an application, our analysis allows us to prove the intrinsic $C^1$-rectifiability of almost all the preimages of a large class of Lipschitz functions between Carnot groups. In particular, from the latter result, we obtain that any geodesic sphere in a Carnot group equipped with an arbitrary left-invariant homogeneous distance is intrinsic $C^1$-rectifiable.
{\let\thefootnote\relax\footnotetext{* \textit{Scuola Normale Superiore, Piazza dei Cavalieri, 7, 56126 Pisa, Italy,}}}
{\let\thefootnote\relax\footnotetext{** \textit{Université Paris-Saclay, 307 Rue Michel Magat Bâtiment, 91400 Orsay, France.}}}




\section{Introduction}


In the Euclidean setting the notion of rectifiable set, and more in general that of rectifiable measure, can be given in two equivalent ways. Either one could prescribe the \textit{infinitesimal} behaviour of the measure by saying that it has \textit{flat} tangent measures almost everywhere, i.e., Hausdorff measures supported on vector subspaces of dimension $k$ or, following a \textit{global} approach, one could say that the measure is absolutely continuous with respect to the Hausdorff $k$-dimensional measure, and that it is supported on a countable union of $k$-dimensional Lipschitz graphs. In Euclidean spaces the latter two notions are equivalent, pretty well-understood and  thoroughly studied, see \cite{Federer1996GeometricTheory, Preiss1987GeometryDensities, Mattila1995GeometrySpaces, DeLellis2008RectifiableMeasures}. 

In the last two decades an increasing interest has grown towards the understanding of rectifiability in some specific non-smooth contexts, such as the context of Carnot groups, see \cref{sec:Prel} for details. A Carnot group $\mathbb{G}$ is a simply connected nilpotent Lie group, whose Lie algebra is stratified and generated by its first layer. Carnot groups are a generalization of Euclidean spaces, and we remark that (quotients of) Carnot groups arise as the infinitesimal models of sub-Riemannian manifolds. The geometry of a Carnot group, even at an infinitesimal scale, might be very different from the Euclidean one. On every Carnot group we have a natural family of anisotropic dilations $\{\delta_\lambda\}_{\lambda>0}$. We always endow $\mathbb G$ with an arbitrary left-invariant homogeneous (with respect to $\{\delta_\lambda\}_{\lambda>0}$) distance $d$, and we recall that any two of them are bi-Lipschitz equivalent. We denote $Q$ the Hausdorff dimension of $\mathbb G$ with respect to any of such distances.

As shown in the fundamental papers \cite{Serapioni2001RectifiabilityGroup, step2}, in step-2 Carnot groups the reduced boundary of a finite perimeter set can be covered up to $\mathcal{H}^{Q-1}$-negligible sets by countably many intrinsic $C^1$-regular hypersurfaces, $C^1_{\mathrm{H}}$ hypersurfaces from now on, see \cref{def:C1Hmanifold}.
The positive De Giorgi's rectifiability result in \cite{Serapioni2001RectifiabilityGroup} has started an effort to study Geometric Measure Theory in sub-Riemannian Carnot groups, and in particular to study various notion(s) of rectifiability, mostly given following the \textit{global} approach described at the beginning of this paragraph. 
 
One of the big efforts in this study is trying to understand what is the correct class of \textit{building blocks} to consider in order to give a satisfactory \textit{global} definition of rectifiable set, or measure, in the setting of Carnot groups. The first notion of rectifiability that has been proposed and studied is the one which considers as buliding blocks $C^1_{\mathrm H}$-surfaces, as explained above, see \cite{step2, FSSC03b, FSSC07, MagnaniTwoardsDiff, JNGV20}. Then a notion of intrinsic Lipschitz graph (see \cref{def:iLipfunctions}) has been proposed and studied in \cite{FSSC06, FranchiSerapioni16}, and relations between the notion of intrinsic Lipschitz rectifiability and the notion of $C^1_{\mathrm H}$-rectifiability have been investigated in \cite{FSSC11, FMS14}. The problem of linking the latter two definitions of rectifiability with the \textit{infinitesimal} viewpoint was raised in \cite{MatSerSC} in the setting of Heisenberg groups $\mathbb H^n$. From the results in \cite{MatSerSC} one deduces that in $\mathbb H^n$ the natural \textit{infinitesimal} notion of rectifiable measure - namely the one given in terms of the existence of flat tangent measures almost everywhere - agrees with the one given in terms of intrinsic Lipschitz graphs in low dimensions, and with the one given in terms of $C^1_{\mathrm H}$-surfaces in low codimensions. Eventually, it took about ten years to conclude that a Rademacher theorem for intrinsic Lipschitz functions in low codimensions holds in $\mathbb H^n$, see \cite{Vittone20}. As a consequence, at least in $\mathbb H^n$, the natural infinitesimal definition of rectifiability always agrees with the one given in terms of coverings with intrinsic Lipschitz graphs. 
An analysis similar to the one of \cite{MatSerSC} has been pursued in \cite{IMM20} in the setting of homogeneous groups and for measures with horizontal tangents. 

Other notions of rectifiability modelled on Lipschitz images of (homogeneous subgroups of) Carnot groups have been proposed by Pauls and Cole-Pauls in \cite{Pauls04, CP06}. An interesting open question asks whether in $\mathbb H^1$ the notion of rectifiability by means of $C^1_{\mathrm H}$-hypersurfaces is equivalent to the one of Cole-Pauls given in \cite{CP06}, see \cite{BigolinVittone10,DDFO20} for some partial results. Nevertheless, in arbitrary Carnot groups, the two notions could be very different, see the results in \cite{ALD}. A weak notion of rectifiability in terms of building blocks that satisfy some mild cone property has also been recently investigated in \cite{DLDMV19}.

On the other hand, from the \textit{infinitesimal} viewpoint, a notion that makes sense in arbitrary Carnot groups has been proposed in \cite{MarstrandMattila20} by the second-named author, namely the notion of $\mathscr{P}$-rectifiable measure, which we recall here. We recall that a subgroup $\mathbb V$ of $\mathbb G$ is said to be \textit{homogeneous} if it is closed under the action of the dilations $\{\delta_\lambda\}_{\lambda>0}$. Again we remark that $\mathbb G$ is endowed with a left-invariant homogeneous distance $d$.

\begin{definizione}[$\mathscr{P}$-rectifiable measures]\label{def:PhRectifiableMeasureINTRO}
Fix a natural number $1\leq h\leq Q$. A Radon measure $\phi$ on $\mathbb G$ is said to be {\em $\mathscr{P}_h$-rectifiable} if for $\phi$-almost every $x\in \mathbb{G}$ we have \begin{itemize}
    \item[(i)]$0<\Theta^h_*(\phi,x)\leq\Theta^{h,*}(\phi,x)<+\infty$,
    \item[(\hypertarget{due}{ii})]$\mathrm{Tan}_h(\phi,x) \subseteq \{\lambda\mathcal{S}^h\llcorner \mathbb V(x):\lambda\geq 0\}$, where $\mathbb V(x)$ is a homogeneous subgroup of $\mathbb G$ of Hausdorff dimension $h$,
\end{itemize}
where $\Theta^h_*(\phi,x)$ and $\Theta^{h,*}(\phi,x)$ are, respectively, the lower and the upper $h$-density of $\phi$ at $x$, see \cref{def:densities}, $\mathrm{Tan}_h(\phi,x)$ is the set of $h$-tangent measures to $\phi$ at $x$, see \cref{def:TangentMeasure}, and $\mathcal{S}^h$ is the spherical Hausdorff measure of dimension $h$, see \cref{def:HausdorffMEasure}.
\end{definizione}

In \cite{antonelli2020rectifiable} we started to study structure results for the class of $\mathscr{P}$-rectifiable measures, proving a Marstrand-Mattila type rectifiability criterion in the co-normal case \cite[Theorem 1.3]{antonelli2020rectifiable}. The latter theorem directly leads to the proof of Preiss's theorem in the first Heisenberg group $\mathbb H^1$ equipped with the Koranyi norm, see \cite[Theorem 1.4]{antonelli2020rectifiable}. In this paper we complete the study of the notion of $\mathscr{P}$-rectifiable measure, when the tangents are complemented, showing that the notion of $\mathscr{P}$-rectifiability, which is \textit{infinitesimal} in nature, is equivalent to the \textit{global} one given in terms of intrinsic differentiable graphs, see \cref{thm:INTRO1Equivalence}. We stress that our \cref{thm:INTRO1Equivalence} extends to arbitrary Carnot groups and to all the dimensions the results given in \cite{MatSerSC}.

All in all we conclude that, in Carnot groups, the correct \textit{building blocks} to consider in order to give a \textit{global} definition of rectifiability that agrees with the infinitesimal one seem to be intrinsic differentiable graphs. We also provide an area formula for such building blocks, see \cref{thm:AREAINTRO2}. We stress that, due to the existence of intrinsic Lipschitz graphs that are nowhere intrinsically differentiable, see \cite{JNGV20a}, one cannot give a geometric area formula in the spirit of \cref{thm:AREAINTRO2} for arbitrary intrinisc Lipschitz graphs. Nevertheless the area formulae in \cref{thm:AREAINTRO1} and \cref{thm:AREAINTRO2} extend the area formula given in \cite[Theorem 1.1]{JNGV20}, see the discussion after \cref{thm:AREAINTRO2}.

We stress that one of the main achievements of this paper is also the rectifiability criterion in \cref{prop:rett.1} which allows to prove the $\mathscr{P}$-rectifiability of almost all the preimages of a large number of Lipschitz functions, see \cref{structure:liplevelsetsINTRO}, and \cref{structure:liplevelsets3Intro}. 

{\em Remark:} Some of the results presented in this paper use results proven in \cite[Sections 2-3-4-6]{antonelli2020rectifiable}. We recall the most important ones in the preliminary section of this work, see \cref{sec:Prel}. During the proofs we give precise references to the results of \cite{antonelli2020rectifiable} when we need them.

\subsection{Main results}

We discuss the main contributions of the present paper. We fix $\mathbb G$ a Carnot group and we equip it with a left-invariant homogeneous distance. We recall that when we say that a homogeneous subgroup $\mathbb V$ {\em admits a complementary subgroup}, we mean that there exists a homogeneous subgroup $\mathbb L$ such that $\mathbb G=\mathbb V\cdot \mathbb L$ and $\mathbb V\cap\mathbb L=\{0\}$. The first result of this work is a complete characterization of $\mathscr{P}_h$-rectifiable sets with complemented tangents in arbitrary Carnot groups either in terms of the existence of flat $h$-dimensional complemented Preiss's tangents almost everywhere or in terms of a covering property with $h$-dimensional intrinsically differentiable graphs with complemented tangents. We recall that, while $\mathrm{Tan}_h(\phi,x)$ captures the behaviour of tangent measures obtained rescaling with the $h$-th power of the scale, see \cref{def:TangentMeasure}, the Preiss's tangent $\mathrm{Tan}(\phi,x)$, see \cref{def:TangentMeasure}, captures the behaviour of all the possible tangent measures, namely 
$$
\mathrm{Tan}(\phi,x):=\{\nu:\,\,\text{there exist $\{c_i\}$, with $c_i>0$, and $\{r_i\}$ with $r_i\to _{i}0$ such that}\,\, c_iT_{x,r_i}\phi\rightharpoonup_{i} \nu\},
$$
where the convergence of measures is meant in the duality with $C_c(\mathbb G)$, see \cref{def:WeakConvergence}. For the reader's convenience we recall here that an intrinsic graph with respect to a splitting $\mathbb G=\mathbb V\cdot\mathbb L$ of the group is said to be {\em intrinsically differentiable} at one of its points if the Hausdorff tangent at that point is a homogeneous subgroup, see \cref{defiintrinsicdiffgraph} for a precise definition. For the proof of the next statement, see the end of \cref{sec:Density}.
\begin{teorema}\label{thm:INTRO1Equivalence}
Let $\mathbb G$ be a Carnot group and fix a natural number $1\leq h\leq Q$. Let $\Gamma\subseteq\mathbb G$ be a Borel set such that $\mathcal{S}^h(\Gamma)<+\infty$, where $\mathcal{S}^h$ is the $h$-dimensional spherical Hausdorff measure. Then the following are equivalent
\begin{enumerate}
    \item $\mathcal{S}^h\llcorner\Gamma$ is a $\mathscr{P}_h$-rectifiable measure with \textbf{complemented} tangents, i.e., a $\mathscr{P}_h^c$-rectifiable measure, see \cref{def:PhRectifiableMeasure},
    \item For $\mathcal{S}^h\llcorner\Gamma$-almost every $x\in\mathbb G$ we have 
    $$
    \Tan(\mathcal{S}^h\llcorner\Gamma,x)=\{\lambda\mathcal{S}^h\llcorner\mathbb V(x):\lambda>0,\mathbb V(x)\,\text{is a \textbf{complemented} homogeneous subgroup of $\mathbb G$ with $\mathrm{dim}_H\mathbb V(x)=h$}\},
    $$
    \item There exist countably many compact intrinsic Lipschitz graphs $\Gamma_i$ that are $h$-dimensional intrinsically differentiable graphs at $\mathcal{S}^h$-almost every $x\in\Gamma_i$, that have \textbf{complemented} Hausdorff tangents at $\mathcal{S}^h$-almost every $x\in\Gamma_i$, and such that 
    $$
    \mathcal{S}^h(\Gamma\setminus\cup_{i=1}^{+\infty}\Gamma_i)=0.
    $$
\end{enumerate}
Moreover, denoting with $\mathcal{C}^h$ the centered Hausdorff measure of dimension $h$, see \cref{def:HausdorffMEasure}, if any of the previous holds, then $\Theta^{h}(\mathcal{C}^h\llcorner\Gamma,x)=1$ exists for $\mathcal{C}^h\llcorner\Gamma$-almost every $x\in\mathbb G$ and
$$
r^{-h}(T_{x,r})_*(\mathcal{C}^h\llcorner\Gamma)\rightharpoonup \mathcal{C}^h\llcorner\mathbb V(x), \qquad \text{for $\mathcal{C}^h\llcorner\Gamma$-almost every $x\in\mathbb G$,}
$$
where the convergence of measures is meant in the duality with $C_c(\mathbb G)$.
\end{teorema}


Let us observe that when a Rademacher theorem is available, we can equivalently consider as the building blocks in item 3. of \cref{thm:INTRO1Equivalence} the class of $h$-dimensional intrinsically Lipschitz graphs, without asking anything a priori on the differentiability. Let us recall that a Rademacher theorem is proved in \cite{FMS14, FSSC11} in the setting of Carnot groups $\mathbb G$ of type $\star$, i.e., a class strictly larger than Carnot groups of step 2, and for maps $\varphi:U\subseteq\mathbb W\to \mathbb L$, where $\mathbb W$ and $\mathbb L$ are complementary subgroups of $\mathbb G$, with $\mathbb L$ \textbf{horizontal and one-dimensional}. Moreover, with the recent results of \cite{LDM20}, the latter codimension-one Rademacher theorem can be extended to the groups of type diamond introduced in \cite{LDM20}.
Recently, by making use of the theory of currents, the author of \cite{Vittone20} has proved the Rademacher theorem for intrinsically Lipschitz maps between complementary subgroups of any dimension in the Heisenberg groups $\mathbb H^n$, while in \cite{AM20} we proved the validity of a Rademacher theorem for co-normal intrinsically Lipschitz graphs.

Nevertheless, Rademacher theorem is now known to be false in arbitrary Carnot groups in a very strict sense, i.e., there exists an $h$-dimensional intrinsically Lipschitz graph in a Carnot group such that at every point of it there exist infinitely many blow-ups and each of these blow-ups is not a homogeneous subgroup, see \cite[Theorem 1.1]{JNGV20a}. This latter result implies that in general in item 3. of \cref{thm:INTRO1Equivalence} one cannot equivalently consider as building blocks of a locally well-behaved definition of rectifiable sets the family of $h$-dimensional intrinsically Lipschitz graphs. So, in some sense, the result of \cref{thm:INTRO1Equivalence} is sharp also in view of the negative result of \cite{JNGV20a}.

Let us further notice that we do not consider in this work the relations between the three items in \cref{thm:INTRO1Equivalence} and the existence of an approximate tangent in the sense of \cite[Definition 15.17]{Mattila1995GeometrySpaces} (cf. \cite[Definition 3.7]{MatSerSC}). This relation will be target of future investigations. All in all, taking into account that $\mathcal{S}^h\llcorner\Gamma$ is $\mathscr{P}_h^c$-rectifiable with co-horizontal tangents if and only if $\Gamma$ is $C^1_H(\mathbb G,\mathbb R^{Q-h})$-rectifiable, see \cref{def:GG'Rect} and \cref{cor:PhCCoorizzontali}, our result in \cref{thm:INTRO1Equivalence} extends and strengthens \cite[(i)$\Leftrightarrow$(ii)$\Leftrightarrow$(iv)$\Leftrightarrow$(v) of Theorem 3.15]{MatSerSC}. Notice also that in the previous chain of equivalences, we can also drop the assumption on the lower density in \cite[(iv),(v)]{MatSerSC}. 
Moreover, taking into account the Rademacher theorem of \cite{AM20} in the co-normal case, our result in \cref{thm:INTRO1Equivalence} extends \cite[(i)$\Leftrightarrow$(ii)$\Leftrightarrow$(iv)$\Leftrightarrow$(v) of Theorem 3.14]{MatSerSC} as well. 
Let us recall, for the reader's convenience, that \cite[Theorem 3.15]{MatSerSC} deals with the characterization of co-horizontal rectifiability in the Heisenberg groups $\mathbb H^n$, while \cite[Theorem 3.14]{MatSerSC} 
deals with the characterization of horizontal rectifiability in the Heisenberg groups $\mathbb H^n$.

Let us final notice that we stated our result in \cref{thm:INTRO1Equivalence} for measures of the form $\mathcal{S}^h\llcorner \Gamma$, but we could also give a version of it for Radon measures with $\Theta^{h,*}(\phi,x)<+\infty$ for $\phi$-almost every $x\in\mathbb G$, after having proven the analogue of \cref{thm:ExistenceOfDensityPlus} for measures. 

%

The second result of the work is an area formula for intrinsic Lipschitz graphs that are in addition $\mathscr{P}_h^c$-rectifiable. The proof of the following statement is at the end of \cref{sec:AreaFormula}.
\begin{teorema}\label{thm:AREAINTRO1}
Let $\mathbb V,\mathbb L$ be homogeneous complementary subgroups of a Carnot group $\mathbb G$ such that $h:=\dim_H\mathbb V$. Let $\Gamma$ be the graph of an intrinsic Lipschitz map $\varphi:A\subseteq\mathbb V\to\mathbb L$ (see \cref{def:iLipfunctions}), with $A$ Borel, such that $\mathcal{S}^h\llcorner\Gamma$ is $\mathscr{P}_h^c$-rectifiable with tangent measures $\mathcal{S}^h\llcorner\Gamma$-almost everywhere supported on homogeneous subgroups \textbf{complemented by $\mathbb L$}. Then, for every Borel function $\psi:\Gamma\to[0,+\infty)$ the following area formula holds
\begin{equation}\label{eqn:AreaFormulaIntro}
\int_\Gamma \psi d\mathcal{C}^h\llcorner\Gamma = \int_A \psi(a\cdot\varphi(a))\mathcal{A}(\mathbb V(a\cdot\varphi(a)))\mathcal{C}^h\llcorner\mathbb V,
\end{equation}
where $\mathcal{C}^h$ is the centered Hausdorff measure, see \cref{def:HausdorffMEasure}, $\mathbb V(a\cdot\varphi(a))$ is the tangent on which it is supported the tangent measure of $\mathcal{S}^h\llcorner\Gamma$ at the point $a\cdot\varphi(a)\in\Gamma$, and $\mathcal{A}(\cdot)$ is the centered area factor defined with respect to the splitting $\mathbb G=\mathbb V\cdot\mathbb L$, see \cref{def:areafactorCENTR}.
\end{teorema}

A consequence of the previous result is the following one, which is an area formula for intrinsic Lipschitz graphs that are also intrinsically differentiable. The proof of the following statement can be found at the end of \cref{sec:AreaFormula}.

\begin{teorema}\label{thm:AREAINTRO2}
Let $\mathbb V,\mathbb L$ be homogeneous complementary subgroups of a Carnot group $\mathbb G$ such that $h=\dim_H\mathbb V$. Let $\Gamma$ be the graph of an intrinsic Lipschitz map $\varphi:A\subseteq\mathbb V\to\mathbb L$, with $A$ Borel. Let us assume $\Gamma$ is an intrinsically differentiable graph (see \cref{defiintrinsicdiffgraph}) at $\mathcal{S}^h$-almost every $x\in\Gamma$ and let us assume that the Hausdorff tangent $\mathbb V(x)$ of $\Gamma$ at $x$ is complemented by $\mathbb L$ at $\mathcal{S}^h$-almost every $x\in\Gamma$. Then, for every Borel function $\psi:\Gamma\to[0,+\infty)$, the area formula in \eqref{eqn:AreaFormulaIntro} holds.
\end{teorema}

Let us remark that, taking into account that $\Theta^{h}(\mathcal{C}^h\llcorner\Gamma,x)=1$ for $\mathcal{C}^h\llcorner\Gamma$-almost every $x\in\mathbb G$, see \cref{thm:INTRO1Equivalence}, and considering the result in \cite[Corollary 3.6]{JNGV20} one can show that \eqref{eqn:AreaFormulaIntro} extends and strengthens the area formula of  \cite[Theorem 1.1]{JNGV20}. Indeed, the graph of a $C^1_{\mathbb W,\mathbb V}$ function as in the statement of \cite[Theorem 1.1]{JNGV20} is a $\mathscr{P}_h^c$-rectifiable set, see \cite[Proposition 6.2]{antonelli2020rectifiable}. Moreover, we stress that \cref{thm:AREAINTRO1} strictly strengthens \cite[Theorem 1.1]{JNGV20} for two reasons: there exists natural examples of graphs that are $\mathscr{P}_h$-rectifiable but not $C^1_{\mathrm H}$-rectifiable, see \cite[Remark 6.3]{antonelli2020rectifiable}, and moreover in our result the map $\varphi$ does not need to be defined on an \textbf{open} set but just on a \textbf{Borel} set. 


Let us notice that the area formula in \cref{thm:AREAINTRO2} is \textit{geometric}. It roughly asserts that when an intrinsic Lipschitz graph over the split $\mathbb V\cdot\mathbb L$ has almost everywhere a flat Hausdorff tangent complemented by $\mathbb L$, then the area of this graph can be obtained integrating on $\mathbb V$ a \textit{geometric} area factor. With \textit{geometric} we mean that the factor only depends on the tangent. Let us stress that when a Rademacher theorem is available, one can remove the hypothesis about the intrinsic differentiability in \cref{thm:AREAINTRO2}. Nevertheless, as discussed above, a Rademacher theorem might not hold in arbitrary Carnot groups, see \cite{JNGV20a}. 

Let us point out that in the literature one can find many more \textit{analytic} area formulae in Carnot groups, i.e., in which the area element is expressed in terms of properly defined intrinsic derivatives of the map $\varphi$. This is the case of \cite[Theorem 1.1 and Theorem 1.2]{CM20} for low-codimensional $C^1_{\mathrm H}$-surfaces in Heisenberg groups (cf. also \cite[Theorem 2]{FSSC06}), which has been extended to intrinsic Lipschitz low-codimensional surfaces  in \cite[Theorem 1.3]{Vittone20} (cf. also \cite[Theorem 1.6]{CMPSC14}); 
and of \cite[Proposition 1.8]{ADDDLD20} for one-codimensional $C^1_{\mathrm H}$-graphs in arbitrary Carnot groups. These formulas could be derived from \cref{thm:AREAINTRO2} explicitly writing the area element in terms of the intrinsic derivatives of the parametrisation map $\varphi$. Other geometric area formulae for Euclidean $C^1$ or $C^{1,1}$-submanifolds in Carnot groups have been investigated in \cite{MV08, MTV15, MagnaniTowardArea}. Let us remark that our point of view is \textit{intrinsic} while on the contrary the works \cite{MV08, MTV15, MagnaniTowardArea} investigate Euclidean-regular manifolds. The results in \cite[Theorem 1.1 and Theorem 1.2]{MV08}, \cite[Theorem 1.1]{MTV15}, and \cite[Theorem 1.1 and Theorem 1.2]{MagnaniTowardArea} roughly assert that whenever a point of a Euclidean-regular submanifold is sufficiently \textit{nice}, then the intrinsic blow-up at that point exists and it is a homogeneous subgroup; and as a consequence also the density - of the correct dimension - of the (Euclidean) volume measure of the submanifold exists at that point. Then what one notices is that in a lot of cases the \textit{nice} points are almost all - with respect to the intrinsic Hausdorff measure of the correct dimension - the points of the submanifold, cf. \cite[Theorem 1.2]{MTV15}. These latter results have to be compared with our \cref{prop:rett.1} in which we prove that having almost everywhere an intrinsic complemented blow-up implies the existence of the density of the Hausdorff measure. Let us notice that when a negligibility theorem, a blow-up theorem, and the existence of the density hold in the sense of \cite{MV08, MTV15, MagnaniTowardArea} discussed above for a Euclidean-regular submanifold $\Sigma$, then one gets that $\mathrm{Tan}_h(\mathcal{S}^h\llcorner\Sigma,x)\subseteq \{\lambda\mathcal{S}^h\llcorner\mathbb V(x):\lambda>0\}$ holds for $\mathcal{S}^h$-almost every $x\in\Sigma$, where $h$ is the Hausdorff dimension of $\Sigma$, and where $\mathbb V(x)$ is a homogeneous subgroup. This last observation easily follows arguing as in the last lines of the proof of \cref{prop:rett.1}. As a result, when a negligibility theorem, a blow-up theorem, and the existence of the density hold in the sense of \cite{MV08, MTV15,MagnaniTowardArea} for a Euclidean-regular submanifold $\Sigma$, one gets that $\mathcal{S}^h\llcorner\Sigma$ is a $\mathscr{P}_h$-rectifiable measure.

Let us finally notice that the area formula in \cref{thm:AREAINTRO2} is a formula for the building blocks for the global notion of rectifiability in item 3. of \cref{thm:INTRO1Equivalence}. Thus, by localization, one obtains from \cref{thm:AREAINTRO2} and \cref{thm:INTRO1Equivalence}, an area formula for arbitrary $\mathscr{P}_h^c$-rectifiable measures, see \cref{cor:AreaFormulaMeasure}.

The third result of the work is a rectifiability result for the level sets of Lipschitz functions between Carnot groups. For the proof of the following statement, we refer the reader to \cref{structure:liplevelsets}.

\begin{proposizione}\label{structure:liplevelsetsINTRO}
Let $B$ be a Borel set in $\mathbb{G}$ and suppose $\mathbb{H}$ is a Carnot group of homogeneous dimension $Q'$ with $Q\geq Q'$. Let $f:B\subseteq \mathbb{G}\to \mathbb{H}$ be a Lipschitz map such that
\begin{equation}
    \text{$\mathrm{Ker}(df(x))$ is a \textbf{complemented} homogeneous sbugroup of $\mathbb G$ for $\mathcal{S}^{Q}$-almost every $x\in \{z\in B:df(z)\, \text{exists surjective}\}$},
    \nonumber
\end{equation}
where $df(x)$ is the Pansu differential that exists for $\mathcal{S}^Q$-almost every $x\in\mathbb G$, see \cref{def:C1h}. Then, for $\mathcal{S}^{Q'}$-almost every $y\in f(B)$, the measure $\mathcal{S}^{Q-Q'}\llcorner f^{-1}(y)$ is $\mathscr{P}^c_{Q-Q'}$-rectifiable in $\mathbb{G}$ and 
$$
\Tan_{Q-Q'}(\mathcal{S}^{Q-Q'}\llcorner f^{-1}(y),x)\subseteq\{\lambda\mathcal{S}^{Q-Q'}\llcorner\mathrm{Ker}(df(x)):\lambda> 0\},\quad\text{for $\mathcal{S}^{Q-Q'}$-almost every $x\in f^{-1}(y)$.}
$$
\end{proposizione}

As an immediate consequence of the previous result we have the following corollary, which is worth pointing out explicitly, since, even if simple to state, it seems to be not present in the literature up to the authors' knowledge. For a proof, see \cref{structure:liplevelsets3}. We remark that a more general statement, in the co-horizontal case, is still true and can be found in \cref{structure:liplevelsets2}.

\begin{corollario}\label{structure:liplevelsets3Intro}
Suppose $f:B\subseteq \mathbb{G}\to \R$ is a Lipschitz map, where $B$ is a Borel set. Then, for $\mathcal{S}^1$-almost every $y\in f(B)$ the set $f^{-1}(y)$ is $C^1_{\mathrm H}$-rectifiable. In particular, for every $r>0$ and $x\in\mathbb G$, every geodesic sphere $\partial B_r(x)$ is $C^1_{\mathrm H}$-rectifiable.
\end{corollario}

Notice that the last part of the previous result comes from the first applied to the distance function from $x$, and the observation that, by dilating, once $\partial B_r(x)$ is $C^1_{\mathrm H}$-rectifiable for \emph{one} radius $r>0$, then it is $C^1_{\mathrm H}$-rectifiable for \emph{every} radius $r>0$. The previous corollary should be compared with \cite[Theorem 3.2]{Vit12}, where an intrinisc Lipschitz rectifiability for Lipschitz surfaces is proved in CC-spaces. Notice, however, that nowadayas it is not known wheter, in codimension-one, intrinisc Lipschitz rectifiability and $C^1_{\mathrm H}$-rectifiability are equivalent in arbitrary Carnot groups; while being intrinisc Lipschitz rectifiable is weaker than being $C^1_{\mathrm H}$-rectifiable. Nevertheless, our previous corollary provides the $C^1_{\mathrm H}$-rectifiability of all the spheres in every Carnot group. It is however interesting to point out how the {\em Euclidean} rectifiability of the geodesic spheres in Carnot groups is still an intriguing open question in general, and it is related to asymptotic volume expansion in nilpotent Lie groups \cite{BrLD, LeDonneNGTAMS}.

Let us remark that the previous results in \cref{structure:liplevelsetsINTRO} and \cref{structure:liplevelsets3Intro} follow from the rectifiability criterion in \cref{prop:rett.1}. It is worth pointing out that, given a Lipschitz function $f:B\subseteq\mathbb G\to\mathbb R$, for every $y\in f(B)$, the set $\{f\leq y\}$ is of locally finite perimeter in $\mathbb G$, see, e.g., \cite[Theorem 2.40]{DonVittone19}. Hence, taking into account \cref{structure:liplevelsets3Intro}, we deduce the following consequence: $\mathcal{S}^1$-almost all the sublevel sets of real-valued Lipschitz functions defined on Borel subsets of Carnot groups are examples of sets of locally finite perimeter whose boundary is $C^1_{\mathrm H}$-rectifiable - namely De Giorgi's rectifiability Theorem holds for such sets.

Let us finally stress that the previous results in \cref{structure:liplevelsetsINTRO} and \cref{structure:liplevelsets3Intro} open the way to proving slicing theorems and coarea formulae for $\mathscr{P}$-rectifiable and Lipschitz slicing functions. This will be target of future investigations.

\vspace{0.3cm}

Let us briefly comment on the proof of \cref{thm:INTRO1Equivalence}. For what concerns the implications 1. $\Rightarrow$ 2., and 1. $\Rightarrow$ 3., the first is just a matter of routine argument, see \cite[Remark 14.4(3)]{Mattila1995GeometrySpaces}, and the second is a consequence of \cite[Theorem 1.8]{antonelli2020rectifiable}. The main new contributions of this paper, which lead to the equivalence in \cref{thm:INTRO1Equivalence}, are the implications 2. $\Rightarrow$ 1., and 3. $\Rightarrow$ 1., both of them non-trivial. 

For what concerns the implication 2. $\Rightarrow$ 1., we first use that the hypothesis of flat Preiss's tangents allows to conclude that $\Gamma$ is $\mathcal{S}^h$-almost everywhere covered by countably many $h$-dimensional graphs $\Gamma_i$ of  intrinsically Lipschitz functions, namely $\mathcal{S}^h(\Gamma\setminus \cup_{i=1}^{+\infty}\Gamma_i)=0$, see \cref{prop:TangentCompemented}. Hence we exploit the general fact, that dates back to Preiss's paper (cf. \cite[Corollary 2.7]{Preiss1987GeometryDensities}), that a measure with a compact-based tangent at a point is  asymptotically doubling at that point. Joining the latter two observations, we deduce that, for every $i$, the measure $\mathcal{S}^h\llcorner\Gamma_i$ is asymptotically doubling, and then this enables us to prove that $\Gamma_i$ has big projections on the plane over which $\Gamma_i$ is a graph, see \cref{prop:proj}. The latter proposition is just a straightforward empowerment of our result already proved in \cite[Proposition 4.6]{antonelli2020rectifiable}. Finally, the big projections property of \cref{prop:proj} allows us to conclude that the $h$-lower density $\Theta^h_*(\mathcal{S}^h\llcorner\Gamma_i,\cdot)$ is positive $\mathcal{S}^h\llcorner\Gamma_i$ almost everywhere, see \cref{prop:DensitaInfPositiva}. Hence, the proof of the implication 2. $\Rightarrow$ 1. is concluded since we can argue, by exploiting Lebesgue differentiation theorem, that $\Theta^{h}_*(\mathcal{S}^h\llcorner\Gamma,\cdot)$ is positive $\mathcal{S}^h\llcorner\Gamma$-almost everywhere, which was the non-trivial missing information to prove 1. Let us notice that in 2. we are not requiring anything a priori on the positivity of the $h$-lower density of $\mathcal{S}^h\llcorner\Gamma$, otherwise the implication 2. $\Rightarrow$ 1. would have been trivial. Nevertheless, we deduce the positivity of the $h$-lower density from the fact that the tangents are flat and complemented as we discussed above. 

The proof of the implication 3. $\Rightarrow$ 1. relies on the fact that an arbitrary $h$-dimensional (almost everywhere) intrinsically differentiable graph $\Gamma$ with complemented Hausdorff tangents has the property that $\mathcal{S}^h\llcorner\Gamma$ is $\mathscr{P}_h^c$-rectifiable. This is exactly the content of \cref{prop:rett.1}. In order to prove the latter proposition, we show that when we have an arbitrary $h$-dimensional (almost everywhere) intrinsically differentiable graph $\Gamma$ with complemented Hausdorff tangents, at ($\mathcal{S}^h\llcorner\Gamma$-almost) every point we have that the graph $\Gamma$ is, at arbitrarily small scales, contained in a cone with arbitrarily small opening and with basis the Hausdorff tangent at that point. This observation enables us to perform a covering argument and to show directly that $\Theta^h(\mathcal{C}^h\llcorner\Gamma,\cdot)=1$ at $\mathcal{C}^h\llcorner\Gamma$-almost every point. Then the fact that $\mathcal{C}^h\llcorner\Gamma$, and hence $\mathcal{S}^h\llcorner\Gamma$, is $\mathscr{P}_h^c$-rectifiable is reached by using a classical argument, see \cref{prop::E} and \cref{regularity:density}. Let us notice that in \cref{prop:rett.1} it is essential to work with the centered Hausdorff measure $\mathcal{C}^h\llcorner\Gamma$, since we consider coverings with balls centered on $\Gamma$. It is also worth noticing that $\mathcal{S}^h\llcorner\Gamma$ and $\mathcal{C}^h\llcorner\Gamma$ are mutually absolutely continuous so any rectifiability property for one measure is transferred to the other by means of Lebesgue differentiation theorem and the locality of tangents.

The final part of the statement in \cref{thm:INTRO1Equivalence} is a consequence of the fact that the $h$-density of $\mathcal{C}^h\llcorner\Gamma$ is 1 as a consequence of the previous reasoning, and the fact that $\mathcal{C}^h\llcorner\mathbb V(B(0,1))=1$ for every homogeneous subgroup $\mathbb V$ of Hausdorff dimension $h$, see \cref{lemma:PALLAUNITARIAVOLUMEUNO}.

Let us briefly comment on the proof of \cref{thm:AREAINTRO1}. The strategy of the proof is similar to the one in \cite{JNGV20} and it is based on a continuity property of the volumes through a blow-up procedure, see \cref{prop:CrucialLimit}. Nevertheless, in order to prove \cref{prop:CrucialLimit}, one needs to face some delicate technical problems due to the fact that the map $\varphi$ is not necessarily defined on an open set, but just on a \textbf{Borel} set. Hence, one needs to argue directly on the graph by using a Vitali-type differentiation theorem, see \cref{prop:vitali2PDAdapted}, and a new delicate estimate on the volumes of the projections of balls in $\Gamma$, see \cref{prop:EstimateOnProjection}.

\section{Preliminaries}\label{sec:Prel}
\subsection{Carnot Groups}\label{sub:Carnot}
In this subsection we briefly introduce some notations on Carnot groups that we will extensively use throughout the paper. For a detailed account on Carnot groups we refer to \cite{LD17}.

A Carnot group $\mathbb{G}$ of step $\kappa$  is a connected and simply connected Lie group whose Lie algebra $\mathfrak g$ admits a stratification $\mathfrak g=V_1\, \oplus \, V_2 \, \oplus \dots \oplus \, V_\kappa$. We say that $V_1\, \oplus \, V_2 \, \oplus \dots \oplus \, V_\kappa$ is a {\em stratification} of $\mathfrak g$ if $\mathfrak g = V_1\, \oplus \, V_2 \, \oplus \dots \oplus \, V_\kappa$,
$$
[V_1,V_i]=V_{i+1}, \quad \text{for any $i=1,\dots,\kappa-1$}, \quad \text{and} \quad [V_1,V_\kappa]=\{0\},
$$ 
where $[A,B]:=\mathrm{span}\{[a,b]:a\in A,b\in B\}$. We call $V_1$ the \emph{horizontal layer} of $\mathbb G$. We denote by $n$ the topological dimension of $\mathfrak g$, by $n_j$ the dimension of $V_j$ for every $j=1,\dots,\kappa$. 
Furthermore, we define $\pi_i:\mathbb{G}\to V_i$ to be the projection maps on the $i$-th strata. 
We will often shorten the notation to $v_i:=\pi_iv$.

The exponential map $\exp :\mathfrak g \to \mathbb{G}$ is a global diffeomorphism from $\mathfrak g$ to $\mathbb{G}$.
Hence, if we choose a basis $\{X_1,\dots , X_n\}$ of $\mathfrak g$,  any $p\in \mathbb{G}$ can be written in a unique way as $p=\exp (p_1X_1+\dots +p_nX_n)$. This means that we can identify $p\in \mathbb{G}$ with the $n$-tuple $(p_1,\dots , p_n)\in \R^n$ and the group $\mathbb{G}$ itself with $\R^n$ endowed with $\cdot$ the group operation determined by the Baker-Campbell-Hausdorff formula. From now on, we will always assume that $\mathbb{G}=(\R^n,\cdot)$ and, as a consequence, that the exponential map $\exp$ acts as the identity.

The stratificaton of $\mathfrak{g}$ carries with it a natural family of dilations $\delta_\lambda :\mathfrak{g}\to \mathfrak{g}$, that are Lie algebra automorphisms of $\mathfrak{g}$ and are defined by
\begin{equation}
     \delta_\lambda (v_1,\dots , v_\kappa):=(\lambda v_1,\lambda^2 v_2,\dots , \lambda^\kappa v_\kappa), \quad \text{for any $\lambda>0$},
     \nonumber
\end{equation}
where $v_i\in V_i$. The stratification of the Lie algebra $\mathfrak{g}$  naturally induces a gradation on each of its homogeneous Lie sub-algebras $\mathfrak{h}$, i.e., a sub-algebra that is $\delta_{\lambda}$-invariant for any $\lambda>0$, that is
\begin{equation}
    \mathfrak{h}=V_1\cap \mathfrak{h}\oplus\ldots\oplus V_\kappa\cap \mathfrak{h}.
    \label{eq:intr1}
\end{equation}
We say that $\mathfrak h=W_1\oplus\dots\oplus W_{\kappa}$ is a {\em gradation} of $\mathfrak h$ if $[W_i,W_j]\subseteq W_{i+j}$ for every $1\leq i,j\leq \kappa$, where we mean that $W_\ell:=\{0\}$ for every $\ell > \kappa$.
Since the exponential map acts as the identity, the Lie algebra automorphisms $\delta_\lambda$ are also group automorphisms of $\mathbb{G}$.

\begin{definizione}[Homogeneous subgroups]\label{homsub}
A subgroup $\mathbb V$ of $\mathbb{G}$ is said to be \emph{homogeneous} if it is a Lie subgroup of $\mathbb{G}$ that is invariant under the dilations $\delta_\lambda$.
\end{definizione}

We recall the following basic terminology: a {\em horizontal subgroup} of a Carnot group $\mathbb G$ is a homogeneous subgroup of it that is contained in $\exp(V_1)$; a {\em Carnot subgroup} $\mathbb W=\exp(\mathfrak h)$ of a Carnot group $\mathbb G$ is a homogeneous subgroup of it such that the first layer $V_1\cap\mathfrak h$ of the grading of $\mathfrak h$ inherited from the stratification of $\mathfrak g$ is the first layer of a stratification of $\mathfrak h$.

Homogeneous Lie subgroups of $\mathbb{G}$ are in bijective correspondence through $\exp$ with the Lie sub-algebras of $\mathfrak{g}$ that are invariant under the dilations $\delta_\lambda$. 
For any Lie algebra $\mathfrak{h}$ with gradation $\mathfrak h= W_1\oplus\ldots\oplus W_{\kappa}$, we define its \emph{homogeneous dimension} as
$$\text{dim}_{\mathrm{hom}}(\mathfrak{h}):=\sum_{i=1}^{\kappa} i\cdot\text{dim}(W_i).$$
Thanks to \eqref{eq:intr1} we infer that, if $\mathfrak{h}$ is a homogeneous Lie sub-algebra of $\mathfrak{g}$, we have $\text{dim}_{\mathrm{hom}}(\mathfrak{h}):=\sum_{i=1}^{\kappa} i\cdot\text{dim}(\mathfrak{h}\cap V_i)$. It is well-known that the Hausdorff dimension (for a definition of Hausdorff dimension see for instance \cite[Definition 4.8]{Mattila1995GeometrySpaces}) of a graded Lie group $\mathbb G$ with respect to a left-invariant homogeneous distance coincides with the homogeneous dimension of its Lie algebra. For a reference for the latter statement, see \cite[Theorem 4.4]{LDNG19}. \textbf{From now on, if not otherwise stated, $\mathbb G$ will be a fixed Carnot group}.


For any $p\in \mathbb{G}$, we define the {\em left translation} $\tau _p:\mathbb{G} \to \mathbb{G}$ as\label{tran}
\begin{equation*}
q \mapsto \tau _p q := p\cdot q.
\end{equation*}
As already remarked above, the group operation $\cdot$ is determined by the Campbell-Hausdorff formula, and it has the form (see \cite[Proposition 2.1]{step2})
\begin{equation*}
p\cdot q= p+q+\mathscr{Q}(p,q), \quad \mbox{for all }\, p,q \in  \R^n,
\end{equation*} 
where $\mathscr{Q}=(\mathscr{Q}_1,\dots , \mathscr{Q}_\kappa):\R^n\times \R^n \to V_1\oplus\ldots\oplus V_\kappa$, and the $\mathscr{Q}_i$'s have the following properties. For any $i=1,\ldots \kappa$ and any $p,q\in \mathbb{G}$ we have
\begin{itemize}
    \item[(i)]$\mathscr{Q}_i(\delta_\lambda p,\delta_\lambda q)=\lambda^i\mathscr{Q}_i(p,q)$,
    \item[(ii)] $\mathscr{Q}_i(p,q)=-\mathscr{Q}_i(-q,-p)$,
    \item[(iii)] $\mathscr{Q}_1=0$ and $\mathscr{Q}_i(p,q)=\mathscr{Q}_i(p_1,\ldots,p_{i-1},q_1,\ldots,q_{i-1})$.
\end{itemize}
%
Thus, we can represent the product $\cdot$ as
\begin{equation}\label{opgr}
p\cdot q= (p_1+q_1,p_2+q_2+\mathscr{Q}_2(p_1,q_1),\dots ,p_\kappa +q_\kappa+\mathscr{Q}_\kappa (p_1,\dots , p_{\kappa-1} ,q_1,\dots ,q_{\kappa-1})). 
\end{equation}
%

\begin{definizione}[Homogeneous left-invariant distance and norm]
A metric $d:\mathbb{G}\times \mathbb{G}\to \R$ is said to be {\em homogeneous} and {\em left-invariant} if for any $x,y\in \mathbb{G}$ we have, respectively
\begin{itemize}
    \item[(i)] $d(\delta_\lambda x,\delta_\lambda y)=\lambda d(x,y)$ for any $\lambda>0$,
    \item[(ii)] $d(\tau_z x,\tau_z y)=d(x,y)$ for any $z\in \mathbb{G}$.
\end{itemize}
Given a homogeneous left-invariant distance, its associated norm is defined by $\|g\|_{d}:=d(g,0)$, for every $g\in\mathbb G$, where $0$ is the identity element of $\mathbb G$.

Given a homogeneous left-invariant distance $d$ on $\mathbb G$, for every $x\in \mathbb G$ and every $E\subseteq \mathbb G$ we define $\mathrm{dist}(x,E):=\inf\{d(x,y):y\in E\}$.
\end{definizione}

Throughout the paper we will always endow, if not otherwise stated, the group $\mathbb{G}$ with an arbitrary homogeneous and left-invariant metric. We will denote such a distance with $d$. We remark that two homogeneous left-invariant distances on a Carnot group are always bi-Lipschitz equivalent.


\begin{definizione}[Metric balls and tubular neighbourhoods]
Suppose a homogeneous and left-invariant metric  $d$ has been fixed on $\mathbb{G}$. Then, we define $U_d(x,r):=\{z\in \mathbb{G}:d(x,z)<r\}$ to be the open metric ball relative to the distance $d$ centred at $x$ at radius $r>0$. The closed ball will be denoted with $B_d(x,r):=\{z\in \mathbb{G}:d(x,z)\leq r\}$. Moreover, for a subset $E\subseteq \mathbb G$ and $r>0$, we denote with $B_d(E,r):=\{z\in\mathbb G:\dist(z,E)\leq r\}$ the {\em closed $r$-tubular neighborhood of $E$}  and with $U_d(E,r):=\{z\in\mathbb G:\dist(z,E)< r\}$ the {\em open $r$-tubular neighborhood of $E$}. When the metric $d$ is understood, we will tacitly drop the dependence on the metric in the notation.
\end{definizione}

\begin{definizione}[Hausdorff Measures]\label{def:HausdorffMEasure}
Let $d$ be a homogeneous and left-invariant metric on $\mathbb{G}$ and $A\subseteq \mathbb{G}$ be a Borel set. For any $0\leq h\leq \mathcal{Q}$ and  $\delta>0$, define
\begin{equation}
\begin{split}
    &\mathscr{C}^{h}_{d,\delta}(A):=\inf\Bigg\{\sum_{j=1}^\infty r_j^h:A\subseteq \bigcup_{j=1}^\infty B_{d}(x_j,r_j),\text{ }r_j\leq\delta\text{ and }x_j\in A\Bigg\},\\
    &\mathscr{S}^{h}_{d,\delta}(A):=\inf\Bigg\{\sum_{j=1}^\infty r_j^h:A\subseteq \bigcup_{j=1}^\infty B_d(x_j,r_j),\text{ }r_j\leq\delta\Bigg\},
    \nonumber
\end{split}    
\end{equation}
and $\mathscr{S}^{h}_{\delta,E}(\emptyset):=0=:\mathscr{C}^{h}_{\delta}(\emptyset)$. Eventually, we let
\begin{equation}
    \begin{split}
  \mathcal{C}^{h}_{d}(A):=\sup_{B\subseteq A}\sup_{\delta>0}\mathscr{C}^{h}_{d,\delta}(B)=\sup_{B\subseteq A}\mathcal{C}_{d,0}^h(B)&\qquad\text{be the {\em centred spherical Hausdorff measure}},\\
   \mathcal{S}^{h}_{d}(A):=\sup_{\delta>0}\mathscr{S}^{h}_{d,\delta}(A)&\qquad\text{be the {\em spherical Hausdorff measure}}.
   \nonumber
    \end{split}
\end{equation}
We stress that $\mathcal{C}^h_{d}$ is an outer measure, and thus it defines a Borel regular measure, see \cite[Proposition 4.1]{EdgarCentered}, and that the measures $\mathcal{S}^h_{d},\mathcal{H}^h_{d},\mathcal{C}^h_{d}$ are all equivalent measures, see \cite[Section 2.10.2]{Federer1996GeometricTheory} and \cite[Proposition 4.2]{EdgarCentered}. When the metric $d$ is understood, we will tacitly drop the dependence on the metric in the notation.
\end{definizione}

We recall here the following result that has been proved in \cite[item (iii) of Proposition 2.11]{antonelli2020rectifiable}, and that will often be used in the paper.
\begin{lemma}\label{lemma:PALLAUNITARIAVOLUMEUNO}
Let $\mathbb V$ be a homogeneous subgroup of a Carnot group $\mathbb G$ endowed with a left-invariant homogeneous distance $d$. Let us call $h$ the homogeneous dimension of $\mathbb V$. Hence 
$$
\mathcal{C}^h_d(B_d(x,r)\cap\mathbb V)=r^h,
$$
for every $x\in\mathbb V$ and any $r>0$.
\end{lemma}

\subsection{Densities and tangents of Radon measures}

Throughout the rest of  the paper we will always assume that $\mathbb{G}$ is a fixed Carnot group endowed with a left-invariant homogeneous distance $d$. The homogeneous, and thus Hausdorff, dimension with respect to $d$ will be denoted with $Q$. Furthermore, as discussed in the previous subsection, we will assume without loss of generality that $\mathbb{G}$ coincides with $\R^n$ endowed with the product induced by the Baker-Campbell-Hausdorff formula.

\begin{definizione}[Weak convergence of measures]\label{def:WeakConvergence}
Given a family $\{\phi_i\}_{i\in\N}$ of Radon measures on $\mathbb{G}$ we say that the sequence $\phi_i$ {\em weakly* converges} to a Radon measure $\phi$, and we write $\phi_i\rightharpoonup \phi$, if
$$
\int fd \phi_i \to \int fd\phi \qquad\text{for any } f\in C_c(\mathbb G),
$$
where $C_c(\mathbb G)$ is the space of compactly supported functions on $\mathbb G$.
\end{definizione}

\begin{definizione}[Tangent measures]\label{def:TangentMeasure}
Let $\phi$ be a Radon measure on $\mathbb G$. For any $x\in\mathbb G$ and any $r>0$ we define the measure
$$
T_{x,r}\phi(E):=\phi(x\cdot\delta_r(E)), \qquad\text{for any Borel set }E.
$$
Furthermore, we define $\mathrm{Tan}(\phi,x)$, the {\em tangent measures to $\phi$ at $x$}, to be the collection of the \emph{non-null} Radon measures $\nu$ for which there is a sequence $\{r_i\}_{i\in\N}$, with $r_i\to 0$, and a sequence $\{c_i\}_{i\in\N}$, with $c_i>0$, such that
$$c_iT_{x,r_i}\phi\rightharpoonup \nu.$$

Moreover, we define $\Tan_h(\phi,x)$, the {\em $h$-tangent measures to $\phi$ at $x$}, to be the collection of Radon measures $\nu$ for which there is a sequence $\{r_i\}_{i\in\mathbb N}$, with $r_i\to 0$, such that 
$$
r_i^{-h}T_{x,r_i}\phi\rightharpoonup \nu.
$$
\end{definizione}

\begin{definizione}[Lower and upper densities]\label{def:densities}
Suppose $d$ is a fixed homogeneous left-invariant metric on $\mathbb{G}$. If $\phi$ is a Radon measure on $\mathbb{G}$, and $h>0$, we define
$$
\Theta_{d,*}^{h}(\phi,x):=\liminf_{r\to 0} \frac{\phi(B_{d}(x,r))}{r^{h}},\qquad \text{and}\qquad \Theta^{h,*}_{d}(\phi,x):=\limsup_{r\to 0} \frac{\phi(B_d(x,r))}{r^{h}},
$$
and we say that $\Theta_{d,*}^{h}(\phi,x)$ and $\Theta^{h,*}_{d}(\phi,x)$ are the {\em lower and upper $h$-density of $\phi$ at the point $x\in\mathbb{G}$}, respectively. Furthermore, we say that measure $\phi$ {\em has $h$-density} if
$$
0<\Theta_{d,*}^{h}(\phi,x)=\Theta^{h,*}_{d}(\phi,x)<\infty,\qquad \text{for }\phi\text{-almost any }x\in\mathbb{G}.
$$
When the metric $d$ is understood, we will tacitly drop the dependence on the metric in the notation.
\end{definizione}

\begin{proposizione}\label{prop:convergence}
Assume $\phi$ is a Radon measure on $\mathbb{G}$ and suppose that $r_i^{-h}T_{x,r_i}\phi\rightharpoonup \nu$. Then, for any $z\in\supp(\nu)$ there exists a sequence $y_i\in \supp(\phi)$ such that $\delta_{1/r_i}(x^{-1}y_i)\to z$.
\end{proposizione}

\begin{proof}
A simple argument by contradiction yields the claim. The proof follows verbatim its Euclidean analogue, see for instance the proof of \cite[Proposition 3.4]{DeLellis2008RectifiableMeasures}.
\end{proof}

\begin{definizione}[Definition of $E(\vartheta,\gamma)$]\label{def:EThetaGamma}
Let $\phi$ be a Radon measure on $\mathbb G$ that is supported on the compact set $K$, i.e., such that $\phi(\mathbb G\setminus K)=0$. For any $\vartheta,\gamma\in\N$ we define
\begin{equation}
    E(\vartheta,\gamma):=\big\{x\in K:\vartheta^{-1}r^h\leq \phi(B(x,r))\leq \vartheta r^h\text{ for any }0<r<1/\gamma\big\}.
    \label{eq:A1}
\end{equation}
\end{definizione}

The following two propositions can be found in \cite[Proposition 2.4 and Proposition 2.5]{antonelli2020rectifiable}.

\begin{proposizione}[{\cite[Proposition 2.4 and Proposition 2.5]{antonelli2020rectifiable}}] \label{prop::E}
Assume $\phi$ is a Radon measure supported on the compact set $K$ such that $
0<\Theta^h_*(\phi,x)\leq \Theta^{h,*}(\phi,x)<\infty
$ for $\phi$-almost every $x\in\mathbb G$. Then, for every $\vartheta,\gamma\in\mathbb N$ the set $E(\vartheta,\gamma)$ is compact and $\phi(\mathbb{G}\setminus \bigcup_{\vartheta,\gamma\in\N} E(\vartheta,\gamma))=0$.
\end{proposizione}

\begin{proposizione}\label{regularity:density}
Suppose $\phi$ is a Radon measure with $h$-density. Then for $\phi$-almost every $x\in\mathbb{G}$ we have that $\Tan_h(\phi,x)$ is not empty and for any $\nu\in \Tan_h(\phi,x)$ we have $0\in\supp(\nu)$, and
$\nu(B(y,s))=\Theta^h(\phi,x)s^h$
for any $y\in \supp(\nu)$ and any $s>0$. 
\end{proposizione}

\begin{proof}
The proof follows verbatim its Euclidean counterpart, see for instance \cite[Proposition 3.4]{DeLellis2008RectifiableMeasures}.
\end{proof}

\begin{proposizione}\label{regularity:density2}
Suppose that $\mu$ is a Borel regular measure on $\mathbb{G}$ supported on a homogeneous subgroup $\mathbb{V}\in\G(h)$, such that $0\in\supp(\mu)$ and assume that there exists a constant $C>0$ such that for any $z\in \supp(\mu)$ and any $s>0$ we have
$$\mu(B(z,s))=Cs^h.$$
Then $\mu$ is a Haar measure of the subgroup $\mathbb V$.
\end{proposizione}

\begin{proof}
Without loss of generality we can assume $C=1$. Thanks to \cite[Theorem 3.1]{FSSC15} we thus infer that $\mu=\mathcal{C}^h\llcorner \text{supp}(\mu)$.
Moreover, for any $x\in \supp(\mu)$, thanks to \cref{lemma:PALLAUNITARIAVOLUMEUNO},
we have that $\mu(B(x,r))=\mathcal{C}^h\llcorner \mathbb V (B(x,r))$ for every $r>0$. If by contradiction $\text{supp}(\mu)\neq \mathbb V$, then there would exist a $p\in \mathbb V$ and a $r_0>0$ such that $B(p,r_0)\cap \text{supp} (\mu)=\emptyset$. This however is impossible since we would have
$$\mathcal{C}^h\llcorner \mathbb V(B(0,2(\|p\|+r_0)))\geq \mathcal{C}^h(B(0,2(\|p\|+r_0))\cap \text{supp}(\mu))+\mathcal{C}^h(B(p,r_0)\cap \mathbb V)>\mu(B(0,2(\|p\|+r_0)),$$
and this is a contradiction with what we found above, since by assumption $0\in\text{supp}(\mu)$.
\end{proof}

A very useful property of locally asymptotically doubling measures is that Lebesgue theorem holds and thus local properties are stable under restriction to Borel subsets. The forhcoming result is a direct consequence of \cite[Theorem 3.4.3]{HeinonenKoskelaShanmugalingam} and the Lebesgue differentiation Theorem in \cite[page 77]{HeinonenKoskelaShanmugalingam}.

\begin{proposizione}\label{prop:Lebesuge}
Suppose $d$ is a fixed homogeneous left-invariant metric on $\mathbb{G}$ and that $\phi$ is a Radon measure on $\mathbb{G}$ such that for $\phi$-almost every $x\in \mathbb{G}$ we have
\begin{equation}\label{eqn:DEFASYMPDOUBLING}
\limsup_{r\to 0}\frac{\phi(B_d(x,2r))}{\phi(B_d(x,r))}<\infty.
\end{equation}
Then
\begin{itemize}
    \item[(i)]for any Borel set $B\subseteq \mathbb{G}$ the measure $\phi\llcorner B$ is a locally asymptotically doubling measure, and we have that the following equalities hold for $\phi$-almost every $x\in B$
$$\Theta^h_{d,*}(\phi\llcorner B,x)=\Theta^h_{d,*}(\phi,x),\qquad \text{and}\qquad\Theta^{h,*}_d(\phi\llcorner B,x)=\Theta^{h,*}_d(\phi,x),$$
\item[(ii)] for every non-negative $\rho\in L^1(\phi)$, and for $\phi$-almost every $x\in\mathbb{G}$ we have
$\mathrm{Tan}(\rho\phi,x)=\rho(x)\mathrm{Tan}(\phi,x)$. More precisely, for $\phi$-almost every $x\in\mathbb G$ the following holds
\begin{equation}
    \begin{split}
        \text{if $r_i\to 0$ is such that}\quad c_iT_{x,r_i}\phi \rightharpoonup \nu\quad \text{then}\quad c_iT_{x,r_i}(\rho\phi)\rightharpoonup \rho(x)\nu.  \\
    \end{split}
\end{equation}
\end{itemize}
\end{proposizione}
We recall that any Radon measure $\phi$ on $(\mathbb G,d)$ that satisfies \eqref{eqn:DEFASYMPDOUBLING} for $\phi$-almost every $x\in\mathbb G$ is said to be {\em locally asymptotically doubling}, or simply {\em asymptotically doubling}.

\subsection{Intrinsic Grassmannian in Carnot groups}

We recall in this section some useful properties about homogeneous subgroups in Carnot groups. We equip $\mathbb G$ with a fixed left-invariant homogeneous distance $d$ that will sometimes be understood.

\begin{definizione}[Intrinsic Grassmanian on Carnot groups]\label{def:Grassmannian}
For any $1\leq h\leq Q$, we define $\G(h)$ to be the family of homogeneous subgroups $\mathbb V$ of $\mathbb{G}$ that have Hausdorff dimension $h$. 
Let us recall that if $\mathbb{V}$ is a homogeneous subgroup of $\mathbb{G}$, any other homogeneous subgroup $\mathbb L$ such that
$$
\mathbb{V}\cdot\mathbb{L}=\mathbb{G},\qquad \text{and}\qquad \mathbb{V}\cap \mathbb{L}=\{0\},
$$
is said to be a \emph{complement} of $\mathbb{G}$. We let $\G_c(h)$ to be the subfamily of those $\mathbb V\in\G(h)$ that have a complement and we will refer to $\G_c(h)$ as the $h$-dimensional \emph{complemented} Grassmanian. 
Finally, for any $1\leq h\leq Q$ we endow $\G(h)$ with the metric
$$
d_{\mathbb G}(\mathbb W_1,\mathbb W_2):=d_{H,\mathbb G}(\mathbb W_1\cap B(0,1),\mathbb W_2\cap B(0,1)),
$$
where $d_{H,\mathbb G}$ is the Hausdorff distance of sets induced by the distance $d$. For more details we refer to \cite{antonelli2020rectifiable}.
\end{definizione}

We recall here the following proposition from \cite[Proposition 2.7]{antonelli2020rectifiable} that will be used several times.
\begin{proposizione}[Compactness of the Grassmannian]\label{prop:CompGrassmannian}
For any $1\leq h\leq Q$ $(\G(h),d_{\mathbb G})$ is a compact metric space. 
\end{proposizione}

\begin{definizione}[Projections related to a splitting]\label{def:Projections}
For any $\mathbb V\in \G_c(h)$, if we choose a complement $\mathbb L$ of $\mathbb V$, we can find two unique elements $g_{\mathbb V}:=P_\mathbb V g\in \mathbb V$ and $g_{\mathbb L}:=P_{\mathbb L}g\in \mathbb L$ such that
$$
g=P_\mathbb V (g)\cdot P_{\mathbb L}(g)=g_{\mathbb V}\cdot g_{\mathbb L}.
$$
We will refer to $P_{\mathbb V}(g)$ and $P_{\mathbb L}(g)$ as the \emph{splitting projections}, or simply {\em projections}, of $g$ onto $\mathbb V$ and $\mathbb L$, respectively.
\end{definizione}

\begin{proposizione}[{\cite[Proposition 2.12 and Corollary 2.15]{FranchiSerapioni16}}, {\cite[Proposition 2.14]{antonelli2020rectifiable}}]\label{cor:2.2.19}
Suppose $d$ is a fixed homogeneous left-invariant metric on $\mathbb{G}$ and let $\lVert\cdot\rVert_{d}$ be the associated homogeneous norm. Then, for any $\mathbb{V}\in \G_c(h)$ with complement $\mathbb L$ there is a constant $ \newC\label{ProjC}(d,\mathbb{V},\mathbb L)>0$ such that for any $p\in\mathbb{G}$ we have
\begin{equation}\label{eqn:EstimateC1}
\begin{split}
        \oldC{ProjC}(d,\mathbb{V},\mathbb{L})\lVert P_\mathbb{L}(p)\rVert_{d}\leq \mathrm{dist}&(p,\mathbb{V})\leq\lVert P_\mathbb{L}(p)\rVert_{d},\\
        \oldC{ProjC}(d,\mathbb{V},\mathbb{L})(\lVert P_\mathbb{L}(p)\rVert_{d}+\lVert P_\mathbb{V}(p)\rVert_{d})\leq& \lVert p\rVert_{d}\leq\lVert P_\mathbb{L}(p)\rVert_{d}+\lVert P_\mathbb{V}(p)\rVert_{d}.
\end{split}
\end{equation}
Furthermore, for any $r>0$, there exists a constant $\newC\label{BallC}$ we have
$\mathcal{S}^{h}_d\llcorner \mathbb V\big(P_\mathbb{V}(B_d(p,r))\big)=\oldC{BallC}(d,\mathbb{V},\mathbb L)r^{h}$ and, for any Borel set $A\subseteq \mathbb{G}$ for which $\mathcal{S}^{h}_{d}(A)<\infty$, we have
\begin{equation}
\begin{split}
     \mathcal{S}^{h}_d\llcorner \mathbb{V}(P_\mathbb{V}(A))\leq 2\oldC{BallC}(d,\mathbb{V},\mathbb L)\mathcal{S}^{h}_d(A).
\end{split}
    \label{eq:n520}
\end{equation}
 When the metric $d$ is understood, we will tacitly drop the dependence on the metric in the notation.
\end{proposizione}
\begin{osservazione}[About the definition of \oldC{ProjC}]
We stress that we fix \oldC{ProjC} to be the supremum of all the constants for which the inequalities in \eqref{eqn:EstimateC1} hold.
\end{osservazione}
We recall here the following proposition that will be useful later on.
\begin{proposizione}[{\cite[Proposition 2.10]{antonelli2020rectifiable}, \cite[Proof of Lemma 2.20]{FranchiSerapioni16}}]\label{prop:InvarianceOfProj}
Suppose $d$ is a fixed homogeneous left-invariant metric on $\mathbb{G}$. Let us fix $\mathbb V\in\G_c(h)$ and $\mathbb L$ two complementary homogeneous subgroups of $\mathbb G$. Then, for any $x\in\mathbb{G}$ the map $\Psi:\mathbb{V}\to\mathbb{V}$ defined as $\Psi(z):=P_\mathbb{V}(xz)$ is invertible and it has unitary Jacobian. As a consequence $\mathcal{S}^h_{d}(P_{\mathbb V}(\mathcal{E}))=\mathcal{S}^h_{d}(P_{\mathbb V}(xP_{\mathbb V}(\mathcal{E})))=\mathcal{S}^h_{d}(P_{\mathbb V}(x\mathcal{E}))$ for every $x\in \mathbb G$ and $\mathcal{E}\subseteq \mathbb G$ Borel.
\end{proposizione}

The following proposition will be useful in the proof of \cref{structure:liplevelsets2}. We omit the proof. The first part of the statement, i.e., the one about the homeomorphism, can be reached by using elementary Linear Algebra; while the second part of the statement follows from the first. We stress that in the following statement we are endowing $\mathfrak g$, that is identified with $\mathbb R^n$ through a choice of a basis of left-invariant vector fields $X_1,\dots,X_n$, with a scalar product $\langle\cdot,\cdot\rangle$ that makes $X_1,\dots,X_n$ orthonormal. 
\begin{proposizione}\label{prop:projmap}
Let $\mathcal{L}(\mathfrak g,\mathfrak g)$ be the set of linear maps from the Lie algebra $\mathfrak g$ of $\mathbb G$ into itself, endowed with the operator norm $\rho$. Then, being $\G(\mathfrak g)$ the Grassmannian of the vector space $\mathfrak g$, the map $\mathfrak P:\G(\mathfrak g)\to \mathcal{L}(\mathfrak g,\mathfrak g)$ defined as $\mathfrak P(V):=\Pi_{V^\perp}$, where $\Pi_{V^\perp}$ is the orthogonal projection onto $V^\perp$, is a homeomorphism onto its image.

Then, a map $V:\mathfrak g\to \exp^{-1}(\G(h))$ is Borel measurable if and only if the projection map $\pi_{V^\perp}:\mathfrak g\to\mathcal{L}(\mathfrak g,\mathfrak g)$  defined as
$$\pi_{V^\perp}(x):=\Pi_{V(x)^\perp},$$
where $\Pi_{V(x)^\perp}$ is the orthogonal projection onto $V(x)^\perp$, is Borel measurable.
In addition to this, the Borelianity of $\pi_{V^\perp}$ is also equivalent to saying that  for any fixed $v,w\in \mathfrak g$, the map $x\mapsto \langle v,\Pi_{V(x)^\perp}[w]\rangle$ is Borel.
\end{proposizione}

\subsection{Cones over homogeneous subgroups and intrinisc Lipschitz functions}\label{sub:Cones}

In this section we recall some basic definitions about intrinsic cones in Carnot groups.
\begin{definizione}[Intrinsic cone]\label{def:Cone}
Suppose $d$ is an homogeneous left-invariant distance on $\mathbb{G}$. For any  $\alpha>0$ and $\mathbb V\in \G(h)$, we define the cone $C_{\mathbb V,d}(\alpha)$ as:
$$C_{\mathbb V,d}(\alpha):=\{w\in\mathbb{G}:\mathrm{dist}_{d}(w,\mathbb V)\leq \alpha\|w\|_{d}\}.$$
Furthermore, given $\mathbb V\in \G(h)$ and an $\alpha>0$, we say that a Borel set $E\subseteq \mathbb G$ is a {\em $C_{\mathbb V,d}(\alpha)$-set} if
$$
E\subseteq pC_{\mathbb V,d}(\alpha) \qquad  \text{for any } p\in E. 
$$
When the metric $d$ is understood, we will tacitly drop the dependence on the metric in the notation.
\end{definizione}

\begin{osservazione}[Equivalent intrinsic cones]\label{oss:ConiEquivalenti}Suppose $d$ is a homogeneous left-invariant distance on $\mathbb{G}$.
For some benefits toward the rest of the paper, let us prove that if $\mathbb V\in\G_c(h)$, $\mathbb L$ is a complementary subgroup of $\mathbb V$, and $\alpha<\oldC{ProjC}(d,\mathbb V,\mathbb L)$, then
\begin{equation}
C_{\mathbb V,d}(\alpha)\subseteq \left\{w\in\mathbb G:\|P_\mathbb Lw\|_{d}\leq \frac{\alpha}{\oldC{ProjC}(d,\mathbb V,\mathbb L)-\alpha}\|P_\mathbb Vw\|_{d}\right\}.
    \label{numberoo25}
\end{equation}
Indeed, let us take an element $w$ in the complement of the set in the right-hand-side above. Thanks to the fact that  $\|w\|_{d}\leq\|P_\mathbb Lw\|_{d}+\|P_\mathbb Vw\|_{d}< \oldC{ProjC}(d,\mathbb V,\mathbb L)\alpha^{-1}\|P_\mathbb Lw\|_{d}$ and to \cref{cor:2.2.19} we have
\begin{equation}
    \begin{split}
        \mathrm{dist}_{d}(w,\mathbb V)\geq\oldC{ProjC}(d,\mathbb V,\mathbb L)\|P_\mathbb L(w)\|_{d}>\alpha\|w\|_{d}.
    \end{split}
\end{equation}
Therefore,  any such $w$ is contained in the complement of the left-hand-side of \eqref{numberoo25}, and thus we get the sought conclusion.
Moreover, for any $\mathbb V\in\G_c(h)$ and any of its complementary subgroup $\mathbb{L}$, let us show that for any $\alpha>0$
$$
C_{\mathbb V,\mathbb L,d}(\alpha):=\{w\in\mathbb G:\|P_\mathbb Lw\|_d\leq \alpha\|P_\mathbb Vw\|_d\}\subseteq C_{\mathbb V,d}(\oldC{ProjC}(d,\mathbb V,\mathbb L)^{-1}\alpha).
$$
Indeed, if $w$ is an element in the left-hand-side above, we can readily see thanks to \cref{cor:2.2.19} that
$$
\mathrm{dist}_{d}(w,\mathbb V)\leq \|P_\mathbb Lw\|_{d}\leq \alpha\|P_\mathbb Vw\|_{d}\leq \alpha\oldC{ProjC}({d},\mathbb V,\mathbb L)^{-1}\|w\|_{d}.
$$

All in all we have proved that if $\mathbb V\in \G_c(h)$, $\mathbb L$ is one of its complementary subgroups and $\alpha<\oldC{ProjC}(d,\mathbb V,\mathbb L)$ we have
$$
C_{\mathbb V,\mathbb L, d}(\oldC{ProjC}(d,\mathbb V,\mathbb L)\alpha)\subseteq C_{\mathbb V,d}(\alpha)\subseteq C_{\mathbb V,\mathbb L,d}(\alpha/(\oldC{ProjC}(d,\mathbb V,\mathbb L)-\alpha)),
$$
thus showing that, below some threshold on the opening, the cones $C_{\mathbb V,d}$ and $C_{\mathbb V,\mathbb 
L,d}$ are equivalent.
\end{osservazione}

\begin{osservazione}[Equivalent distances and cones]\label{equiv.lip.cones}
Let $d_1,d_2$ be two homogeneous and left-invariant metrics on $\mathbb{G}$. Since they are bi-Lipschitz equivalent, we can find a constant $0<\mathfrak{c}<1$ such that $\mathfrak{c}d_1(x,y)\leq d_2(x,y)\leq \mathfrak{c}^{-1}d_1(x,y)$
for any $x,y\in \mathbb{G}$. Then for any $\alpha>0$ and any $\mathbb{V}\in\G(h)$ we have
$$C_{\mathbb{V},d_1}(\mathfrak{c}^{2}\alpha)\subseteq C_{\mathbb{V},d_2}(\alpha)\subseteq C_{\mathbb{V},d_1}(\mathfrak{c}^{-2}\alpha).$$
\end{osservazione}

We now recall two results that were already proven in \cite{antonelli2020rectifiable}. We refer the reader to the reference \cite{antonelli2020rectifiable} for the simple proofs.
\begin{lemma}[{\cite[Lemma 2.16]{antonelli2020rectifiable}}]\label{lemma:LCapCw=e}
Suppose $d$ is a homogeneous left-invariant distance on $\mathbb{G}$. For any $\mathbb V\in \G_c(h)$, given $\mathbb L$ to be a complementary subgroup of $\mathbb V$, there exists $0<\newep\label{ep:Cool}:=\oldep{ep:Cool}(d,\mathbb V,\mathbb L)<1$ such that 
$$
\mathbb L\cap C_{\mathbb V,d}(\oldep{ep:Cool})=\{0\}.
$$
and for the aims of this paper, we can and will fix $\oldep{ep:Cool}(d,\mathbb V,\mathbb L):=\oldC{ProjC}(d,\mathbb V,\mathbb L)/2$.
\end{lemma}

\begin{proposizione}\label{prop:grasscompiffC3}
Suppose $d$ is a homogeneous left-invariant distance on $\mathbb{G}$. The function $\mathfrak{e}:\G_c(h)\to\R$ defined as
\begin{equation}\label{eqn:mathfrake}
\mathfrak{e}(\mathbb{V}):=\sup\{\oldep{ep:Cool}(d,\mathbb{V},\mathbb{L}):\mathbb{L}\text{ is a complement of }\mathbb{V}\},
\end{equation}
is lower semicontinuous. In particular the following conclusion holds
\begin{itemize}
    \item[] if $\mathscr{G}\subseteq \G_c(h)$ is compact with respect to $d_{\mathbb G}$, then there exists a $\mathfrak{e}_{\mathscr G}>0$ such that $\mathfrak{e}(\mathbb{V})\geq \mathfrak{e}_{\mathscr G}$ for any $\mathbb{V}\in{\mathscr{G}}$.
\end{itemize}
\end{proposizione}
\begin{proof}
It follows verbatim from the proof of \cite[Proposition 2.22]{antonelli2020rectifiable}, taking \cite[Remark 2.5]{antonelli2020rectifiable} into account. 
\end{proof}

Let us recall the classical definition of intrinisc Lipschits function, see \cite[Definition 11]{FranchiSerapioni16}. 

\begin{definizione}[Intrinsic Lipschitz function]\label{def:iLipfunctions}
Let $\mathbb{W}\in \G_c(h)$, assume $\mathbb{L}$ is a complement of $\mathbb{W}$, and let $E\subseteq \mathbb{W}$. A function $f:E\to \mathbb{L}$ is said to be an \emph{intrinsic Lipschitz function} if there exists an $\alpha>0$ such that for every $p\in\text{graph}(f):=\{v\cdot f(v):v\in E\}$ we have 
$$
\text{graph}(f)\cap pC_{\mathbb W,\mathbb L}(\alpha)=\text{graph}(f),
$$
where the cones $C_{\mathbb W,\mathbb L}(\alpha)$ have been defined in \cref{oss:ConiEquivalenti}.
\end{definizione}

We finally state two properties of intrinisc Lipschitz graphs that will be useful later on, and whose simple proofs can be found in \cite{antonelli2020rectifiable}.
\begin{proposizione}[{\cite[Proposition 2.19]{antonelli2020rectifiable}}]\label{prop:ConeAndGraph}
Suppose $d$ is a homogeneous left-invariant distance on $\mathbb{G}$ and let us fix $\mathbb V\in\G_c(h)$ with complement $\mathbb L$. If $\Gamma\subset\mathbb G$ is a $C_{\mathbb V}(\alpha)$-set for some $\alpha\leq \oldep{ep:Cool}(d,\mathbb V,\mathbb L)$, then the map $P_{\mathbb V}:\Gamma\to\mathbb V$ is injective. As a consequence $\Gamma$ is the intrinsic graph of an intrinsically Lipschitz map defined on $P_{\mathbb V}(\Gamma)$.
\end{proposizione}

\begin{lemma}[{\cite[Lemma 2.21]{antonelli2020rectifiable}}]\label{lemma:lowerbd}
Let $\mathbb{L}$ and $\mathbb{V}$ be homogeneous complementary subgroups of $\mathbb G$ endowed with a left-invariant homogeneous distance $d$. Suppose $\Gamma$ is a $C_{\mathbb V,d}(\alpha)$-set with $\alpha\leq\oldep{ep:Cool}(d,\mathbb V,\mathbb L)$. Then there exists a constant $\mathfrak{C}(\alpha)=\mathfrak{C}(\alpha,d,\mathbb V,\mathbb L)>0$ such that
$$\mathcal{S}^h_d(P_\mathbb{V}(B_d(x,r)\cap \Gamma))\geq \mathcal{S}^h_d\Big(P_\mathbb{V}\big(B_d(x,\mathfrak{C}(\alpha)r)\cap xC_{\mathbb{V},d}(\alpha)\big)\cap P_\mathbb{V}(\Gamma)\Big),$$
for any $x\in\Gamma$, and any $r>0$.
\end{lemma}

\subsection{Rectifiable measures in Carnot groups}
In what follows we are going to define the class of $h$-flat measures on a Carnot group and then we will give proper definitions of rectifiable measures on Carnot groups. 

\begin{definizione}[Flat measures]
For any $h\in\{1,\ldots,Q\}$ we let $\mathfrak{M}(h)$ to be the {\em family of flat $h$-dimensional measures} in $\mathbb{G}$, i.e.,
$$\mathfrak{M}(h):=\{\lambda\mathcal{S}^h\llcorner \mathbb W:\text{ for some }\lambda> 0 \text{ and }\mathbb W\in\G(h)\}.$$ 
Furthermore, if $G$ is a subset of the $h$-dimensional Grassmanian $\G(h)$, we let $\mathfrak{M}(h,G)$ to be the set \begin{equation}
    \mathfrak{M}(h,G):=\{\lambda\mathcal{S}^h\llcorner \mathbb{W}:\text{ for some }\lambda> 0\text{ and }\mathbb{W}\in G\}.
    \label{Gotico(h)}
\end{equation}
\end{definizione}

\begin{definizione}[$\mathscr{P}_h$ and $\mathscr{P}_h^c$-rectifiable measures]\label{def:PhRectifiableMeasure}
Let $h\in\{1,\ldots,Q\}$. A Radon measure $\phi$ on $\mathbb G$ is said to be a $\mathscr{P}_h$-rectifiable measure if for $\phi$-almost every $x\in \mathbb{G}$ we have
\begin{itemize}
    \item[(i)]$0<\Theta^h_*(\phi,x)\leq\Theta^{h,*}(\phi,x)<+\infty$,
    \item[(\hypertarget{due}{ii})]there exists a $\mathbb{V}(x)\in\G(h)$ such that $\mathrm{Tan}_h(\phi,x) \subseteq \{\lambda\mathcal{S}^h\llcorner \mathbb V(x):\lambda\geq 0\}$.
\end{itemize}
Furthermore, we say that $\phi$ is $\mathscr{P}_h^c$-rectifiable if (\hyperlink{due}{ii})  is replaced with the weaker
\begin{itemize}
    \item[(ii)*] there exists a $\mathbb{V}(x)\in\G_c(h)$ such that $\mathrm{Tan}_h(\phi,x) \subseteq \{\lambda\mathcal{S}^h\llcorner \mathbb V(x):\lambda\geq 0\}$.
\end{itemize}
\end{definizione}
\begin{osservazione}(About $\lambda=0$ in \cref{def:PhRectifiableMeasure})\label{rem:AboutLambda=0}
It is readily noticed that, since in \cref{def:PhRectifiableMeasure} we are asking $\Theta^h_*(\phi,x)>0$ for $\phi$-almost every $x$, we can not have the zero measure as an element of $\Tan_h(\phi,x)$ thanks to \cite[Proposition 2.26]{antonelli2020rectifiable}. As a consequence, a posteriori, we have that in item (ii) and item (ii)* we can restrict to $\lambda>0$. We will tacitly work in this restriction from now on, see \cite[Remark 2.7]{antonelli2020rectifiable}.

\end{osservazione}
\begin{osservazione}[About the rectifiability of Hausdorff measures]\label{rem:ShPRectifiableIfAndOnlyIf}
We observe that if $\Gamma$ is a Borel set in $\mathbb G$, $\mathcal{S}^h\llcorner\Gamma$ is $\mathscr{P}_h$-rectifiable if and only if $\mathcal{C}^h\llcorner\Gamma$ (or also $\mathcal{H}^h\llcorner\Gamma)$ is $\mathscr{P}_h$-rectifiable. This is because $\mathcal{S}^h,\mathcal{H}^h,\mathcal{C}^h$ are equivalent measures (see \cref{def:HausdorffMEasure}), the $\mathscr{P}_h$-rectifiability implies being locally asymptotically doubling, and then we can transfer the property of being $\mathscr{P}_h$-rectifiable from one measure to the other by using Lebesgue--Radon--Nikodym theorem (see \cite[page 82]{HeinonenKoskelaShanmugalingam}) and the locality of tangents in \cref{prop:Lebesuge}.
\end{osservazione}

We introduce now a way to estimate how far two measures are.

\begin{definizione}[Definition of $F_K$]\label{def:Fk}
Given $\phi$ and $\psi$ two Radon measures on $\mathbb G$, and given $K\subseteq \mathbb G$ a compact set, we define 
\begin{equation}
    F_K(\phi,\psi):= \sup\left\{\left|\int fd\phi - \int fd\psi\right|:f\in \mathrm{Lip}_1^+(K)\right\},
    \label{eq:F}
\end{equation}
where $\mathrm{Lip}_1^+(K)$ denotes the class of $1$-Lipschitz nonnegative function with compact support contained in $K$. We also write $F_{x,r}$ for $F_{B(x,r)}$.
\end{definizione}

\begin{osservazione}[Properties of $F_K$]\label{rem:ScalinfFxr}
With few computations that we omit, it is easy to see that $F_{x,r}(\phi,\psi)=rF_{0,1}(T_{x,r}\phi,T_{x,r}\psi)$. Furthermore, $F_K$ enjoys the triangular inequality. Indeed, if $\phi_1,\phi_2,\phi_3$ are Radon measures and $f\in\lip(K)$, then
$$\Big\lvert\int fd\phi_1-\int fd\phi_2\Big\rvert\leq \Big\lvert\int fd\phi_1-\int fd\phi_3\Big\rvert+\Big\lvert\int fd\phi_3-\int fd\phi_2\Big\rvert\leq F_K(\phi_1,\phi_2)+F_K(\phi_2,\phi_3).$$
The arbitrariness of $f$ concludes that $F_K(\phi_1,\phi_2)\leq F_K(\phi_1,\phi_3)+F_K(\phi_3,\phi_2)$.
\end{osservazione}

\begin{definizione}[Definition of $F_r$]\label{def:Fr}
 For a given Radon measure $\phi$ on $\mathbb G$ and for $r>0$, let us define $F_r(\phi):=\int \text{dist}(z,U(0,r)^c)d\phi(z)$. 
\end{definizione} 

\begin{lemma}\label{dist:lemma}
For any Radon measure $\phi$ on $\mathbb{G}$ and any $r>0$ we have that  $F_r(\phi)=F_{0,r}(\phi,0)$.
\end{lemma}

\begin{proof}
It is immediate to see that $F_{0,r}(\phi,0)\geq F_r(\phi)$ for any $r>0$. In order to prove the viceversa, note that for any $f\in \lip(B(0,r))$ we have that $f\rvert_{\partial B(0,r)}=0$. Thanks to this observation, for any $y\in B(0,r)$ if we let $x\in \partial B(0,r)$ be a point of minimal distance of $y$ from $U(0,r)^c$ we have
$$f(y)=\lvert f(y)-f(x) \rvert\leq d(y,x)=\dist(y,U(0,r)^c),$$
and this finally shows that $F_{0,r}(\phi,0)= F_r(\phi)$, concluding the proof of the lemma.
\end{proof}

\begin{proposizione}\label{prop1.12preiss}
The function defined on $\mathcal{M}\times\mathcal{M}$ as
$$d(\phi,\psi):=\sum_{p=0}^\infty 2^{-p}\min\{1,F_{0,p}(\phi,\psi)\},$$
is a distance and $(\mathcal{M},d)$ is a separable metric space. The topology induced by $d$ on $\mathcal{M}$ coincides with the one given by the weak* topology.

Moreover, let us assume $\{\phi_i\}_{i\in\N}$ is a sequence of Radon measures such that $\limsup_{i\to\infty}\phi_i(B(0,r))<\infty$ for every $r>0$. Then $\{\phi_i\}_{i\in\N}$ has a converging subsequence with respect to the weak* topology.
\end{proposizione}

\begin{proof}
The result is stated in \cite[Proposition 1.12]{Preiss1987GeometryDensities} in the Euclidean case, but the proof works verbatim for Radon measures on Carnot groups.
\end{proof}

\begin{proposizione}\label{F1:metric}
The function $F_{0,1}(\cdot,\cdot)$ is a metric on $\mathfrak{B}(h):=\{\psi\in\mathfrak{M}(h):F_1(\psi)=1\}$ and $(\mathfrak{B}(h),F_{0,1})$ is a compact metric space.
\end{proposizione}

\begin{proof}
First of all, we note that for any $\mu,\nu\in\mathfrak{B}(h)$ we have that $F_{0,1}(\mu,\nu)=0$ if and only if $\mu=\nu$ and this is an immediate consequence of the fact that $\mu$ and $\nu$ are cones. Symmetry follows directly form the definition and the triangular inequality follows from\cref{rem:ScalinfFxr}.

We are left to show that $(\mathfrak{B}(h),F_{0,1})$ is a compact metric space. Let $\Psi_i$ be a sequence in $\mathfrak{B}(h)$ 
and note that since $\mathcal{C}^h\llcorner\mathbb V(B(0,1))=1$ for every $\mathbb V\in \G(h)$, because of \cref{lemma:PALLAUNITARIAVOLUMEUNO}, we deduce that $\Psi_i=(h+1)\mathcal{C}^h\llcorner \mathbb{V}_i$ for some $\mathbb{V}_i\in \G(h)$. Thus, we can find a (non-relabeled) subsequence of the planes $\mathbb{V}_i$ that converges to some $\mathbb{V}\in\G(h)$ in the Hausdorff metric thanks to the compactness of the Grassmanian $\G(h)$, see \cref{prop:CompGrassmannian}. Hence, by \cite[Proposition 2.29]{antonelli2020rectifiable} we infer that $\Psi_i\rightharpoonup (h+1)\mathcal{C}^{h}\llcorner \mathbb{V} \in \mathfrak{B}(h)$ and therefore the compactness follows.
\end{proof}

Now we are going to define some functionals that quantifies how far is a measure from being flat around a point $x\in\mathbb G$ and at a certain scale $r>0$.

\begin{definizione}[Definition of $d_{x,r}$]\label{def:metr}
For any $x\in\mathbb{G}$, any $h\in\{1,\ldots,Q\}$ and any $r>0$ we define the  functional
\begin{equation}
        d_{x,r}(\phi,\mathfrak{M}(h)):=\inf_{\substack{\mathbb V\in \G(h)}} F_{0,1}(T_{x,r}\phi/F_1(T_{x,r}\phi), (h+1)\mathcal{C}^{h}\llcorner \mathbb V).
\end{equation}
Furthermore, if $G$ is a subset of the $h$-dimensional Grassmanian $\G(h)$, we also define
\begin{equation}\label{eqn:Definitiondxr}
d_{x,r}(\phi,\mathfrak{M}(h,G)):=\inf_{\substack{\mathbb V\in G}} F_{0,1}(T_{x,r}\phi/F_1(T_{x,r}\phi), (h+1)\mathcal{C}^{h}\llcorner \mathbb V).
\end{equation}
\end{definizione}

\begin{osservazione}[About the definition of $d_{x,r}$]\label{rem:ScalingDandF}
For any Radon measure $\phi$ on $\mathbb{G}$ and any $r>0$ it is immediate to see that $F_1(T_{0,r}\phi)=r^{-1}F_r(\phi)$. Moreover, thanks to the first part of \cref{rem:ScalinfFxr}, we get by few simple computations 
\begin{equation}\label{eqn:Homogeneitydr}
F_{0,1}\left(\frac{T_{x,r}\phi}{F_1(T_{x,r}\phi)},(h+1)\mathcal{C}^h\llcorner\mathbb V\right)=r^{-(h+1)}F_{0,r}\left(\frac{T_{x,1}\phi}{r^{-(h+1)}F_r(T_{x,1}\phi)},(h+1)\mathcal{C}^h\llcorner\mathbb V\right),
\end{equation}
for all $r>0$ and $\mathbb V\in \G(h)$. Hence, since $F_1((h+1)\mathcal{C}^h\llcorner\mathbb V)=1$ as a consequence of \cite[Proposition 2.12]{antonelli2020rectifiable} and \cref{lemma:PALLAUNITARIAVOLUMEUNO}, we notice that the definition in \eqref{eqn:Definitiondxr} agrees with the definition given in \cite[\S 2.1(3)]{Preiss1987GeometryDensities}. Namely, $d_{x,r}(\phi,\mathfrak M(h,G))=d_r(T_{x,1}\phi,\mathfrak M(h,G))=d_1(T_{x,r}\phi,\mathfrak M(h,G))$, where $d_r$ is the one defined in \cite[\S 2.1(3)]{Preiss1987GeometryDensities}. 

 For the sake of completeness, and for some benefits toward subsequent calculations, let us recall here the precise definition of the function $d$ Preiss gave in his setting. Let $\mathscr{C}$ be an arbitrary cone of measures without the origin, that means $0\not\in\mathscr{C}$ and $\mu\in\mathscr{C}$ implies $\lambda T_{0,\nu}\mu\in\mathscr{C}$ for every $\lambda,\nu>0$. Then, for every $r>0$ and $\phi$ a Radon measure we define
 \begin{equation}\label{eqn:drPreiss}
 d_r(\phi,\mathscr{C}):=\inf \left\{F_{r}\left(\frac{\phi}{F_r(\phi)},\psi\right):\psi\in\mathscr{C},\,\,F_r(\psi)=1\right\}.
 \end{equation}
 By the explicit expression and the continuity of $F_r(\cdot)$ with respect to the weak* convergence, one easily verifies that, for every $r>0$
 \begin{equation}\label{eqn:ContinuitydrPreiss}
 \phi_k\rightharpoonup_{k}\phi,\,\,F_r(\phi)>0\,\, \Rightarrow d_r(\phi_k,\mathscr{C})\to_{k} d_r(\phi,\mathscr{C}),
 \end{equation}
 compare \cite[2.1(6)]{Preiss1987GeometryDensities}. Moreover, due to a slight modification of \eqref{eqn:Homogeneitydr}, we have, for every $r>0$ and every Radon measure $\phi$, 
 \begin{equation}\label{eqn:HomogeneityPreiss}
   d_r(\phi,\mathscr{C})=d_1(T_{0,r}\phi,\mathscr{C}).  
 \end{equation}
\end{osservazione}

We now adapt some classical results contained in \cite{Preiss1987GeometryDensities} to our context. The aim will be to prove that when a Radon measure on $\mathbb G$ has a tangent at a point that is a cone (of measures) with compact basis, then the measure is locally asymptotically doubling. The following proposition is the analogue of \cite[Propostion 2.2]{Preiss1987GeometryDensities}.

\begin{proposizione}\label{prop:asdoubl}
Assume that $\mathcal{T}$ is a non-empty cone of Radon measures, i.e.,
for any $\nu\in\mathcal{T}$ and any $\lambda,\eta>0$ we have $\eta T_{0,\lambda}\nu\in\mathcal{T}$, and moreover $0\not\in \mathcal{T}$. Then, the following are equivalent
\begin{enumerate}
    \item[(i)] the set $\mathfrak{B}(\mathcal{T}):=\{\nu\in\mathcal{T}:F_{1}(\nu)=1\}$ is weak* compact,
    \item[(ii)] for any sequence $\{\nu_i\}_{i\in\N}\subseteq \mathcal{T}$ such that $\lim_{i\to\infty}F_1(\nu_i)=0$, we have $\nu_i\rightharpoonup 0$,
    \item[(iii)]there is a $q\in(0,\infty)$ such that $\nu(B(0,2r))\leq q\nu(B(0,r))$ for every $r>0$ and any $\nu\in\mathcal{T}$.
\end{enumerate}
\end{proposizione}

\begin{proof}
Let us first prove that (i)$\Rightarrow$(ii). Let $\nu_i$ be a sequence in $\mathcal{T}$ and let us assume that $\lim_{i\to\infty}F_{1}(\nu_i)=0$. 
We note that $\nu_i\rightharpoonup 0$ if and only if $F_{0,t}(\nu_i,0)=F_t(\nu_i)\to_i 0$ for any $t>0$. This means that if $\nu_i$ does not converge to $0$, we infer that there are  a $t>0$ and an $\varepsilon>0$ such that, up to passing to subsequences, we have $F_t(\nu_i)>\varepsilon$ for any $i\in\N$. 
Let us define
$$r_i:=\sup\{r\in[1,t]:F_r(\nu_i)\leq F_1(\nu_i)+1/i\}.$$
It is immediate to see that up to further subsequences
 $F_1(T_{0,r_i}\nu_i)=r_i^{-1}F_{r_i}(\nu_i)>0$ and that
  $$\lim_{i\to \infty} \frac{F_{t/r_i}(T_{0,r_i}\nu_i)}{F_1(T_{0,r_i}\nu_i)}=\lim_{i\to \infty}\frac{F_{t}(\nu_i)}{F_{r_i}(\nu_i)}>\varepsilon\lim_{i\to\infty}(F_1(\nu_i)+1/i)^{-1}=\infty.$$
Thanks to the fact that $\mathcal{T}$ is a cone, we know that $F_1(T_{0,r_i}\nu_i)^{-1}T_{0,r_i}\nu_i\in \mathfrak{B}(\mathcal{T})$ and thus there must exists a converging (non-relabeled) subsequence of $r_i$ and a $\nu\in\mathfrak{B}(\mathcal{T})$ such that $F_1(T_{0,r_i}\nu_i)^{-1}T_{0,r_i}\nu_i\rightharpoonup\nu$. This however implies that
 $$\infty=\lim_{i\to\infty}\frac{F_{t/r_i}(T_{0,r_i}\nu_i)}{F_1(T_{0,r_i}\nu_i)}\leq \lim_{i\to\infty}\frac{F_{t}(T_{0,r_i}\nu_i)}{F_1(T_{0,r_i}\nu_i)}=\lim_{i\to\infty}F_{t}(F_1(T_{0,r_i}\nu_i)^{-1}T_{0,r_i}\nu_i)=F_t(\nu),$$
 that is a contradiction with the fact that $\nu$ is a Radon measure.
 
 Secondly, let us show that (ii)$\Rightarrow$(iii). Since $\mathcal{T}$ is a cone, it suffices to prove that there exists $q\in(0,+\infty)$ such that $\nu(B(0,2))\leq q\nu (B(0,1))$ for every $\nu\in\mathcal{T}$. Indeed, we thus would get that for every $\nu\in\mathcal{T}$ and $r>0$ we have $\nu(B(0,2r))=T_{0,r}\nu(B(0,2))\leq qT_{0,r}\nu(B(0,1))=q\nu(B(0,r))$. Suppose by contradiction that there exists a sequence of measures $\nu_i\in \mathcal{T}$ such that $\nu_i(B(0,2))>i\nu_i(B(0,1))$. Note now that since $\mathcal{T}$ is a cone, the measures $\nu_i(B(0,2))^{-1}\nu_i$ are still in $\mathcal{T}$ and $\lim_{i\to \infty} F_1(\nu_i(B(0,2))^{-1}\nu_i)=0$. Thanks to (ii) this shows in particular that 
 \begin{equation}
 \nu_i(B(0,2))^{-1}\nu_i\rightharpoonup0     
     \label{eq:conv0}
 \end{equation}
 However, since $ F_3(\nu_i(B(0,2))^{-1}\nu_i)\geq 1$ for any $i\in \N$, this is a contradiction with \eqref{eq:conv0}, according to which one should have
 $$\lim_{i\to\infty} F_3(\nu_i(B(0,2))^{-1}\nu_i)=F_3(0)=0,$$
 as clearly $F_3$ is a weak* continuous operator on Radon measures.
 
 Finally, let us prove the implication (iii)$\Rightarrow$(i). Let $\{\nu_i\}_{i\in\N}$ be a sequence in $\mathfrak{B}(\mathcal{T})$ and note that for any $i\in\N$ we have
 $$\nu_i(B(0,1/2))\leq 2F_1(\nu_i)=2,$$
 and thus thanks to (iii) we infer that for any $r>0$ we have $\nu_i(B(0,r))\leq 4\max\{1,q^{\log_2(r)+1}\}$ for any $i\in\N$. \cref{prop1.12preiss} and the weak* continuity of $F_1$ conclude the proof.
\end{proof}
\begin{osservazione}\label{rem:LambdaCones}
For some benefit towards the remaining part of this section, let us notice that if $\mathcal{T}$ is a non-empty cone of Radon measure such that $\mathfrak B(\mathcal{T})$ is weak* compact, for every $\lambda>1$ there is $\tau>1$ such that $F_{\tau r}(\psi)\leq \lambda F_r(\psi)$ for every $r>0$ and $\psi\in\mathcal{T}$. The proof follows verbatim from the five lines in \cite[(1)$\Rightarrow$(5) of Proposition 2.2]{Preiss1987GeometryDensities}.
\end{osservazione}

\begin{proposizione}\label{prop:PhiAsDoubling}
For any Radon measure $\phi$ on $\mathbb{G}$ and $\phi$-almost every $x\in\mathbb{G}$ the set $\Tan(\phi,x)$ is either empty or a cone.
Suppose $\phi$ is a Radon measure on $\mathbb{G}$ such that the set $\mathfrak{B}(\phi,x):=\{\nu\in\Tan(\phi,x):F_1(\nu)=1\}$
is a non-empty weak* compact for $\phi$-almost every $x\in\mathbb{G}$. Then $\phi$ is locally asymptotically doubling.
\end{proposizione}

\begin{proof}
In order to prove the first part of the statement, let $x\in\supp(\phi)$ be a point where $\Tan(\phi,x)$ is non-empty, choose a $\nu\in \Tan(\phi,x)$ and assume that $r_i$ and $c_i$ are two sequences such that
$$c_i T_{x,r_i}\phi\rightharpoonup \nu.$$
To conclude the proof of the claim we need to show that for any $\eta,\lambda>0$ we have $\eta T_{0,\lambda}\nu\in\Tan(\phi,x)$ and to do this, we just note that
$$\eta c_i T_{x,\lambda r_i} \phi=\eta T_{0,\lambda}(c_iT_{x,r_i}\phi)\rightharpoonup\eta T_{0,\lambda}\nu.$$
This shows that $\eta T_{0,\lambda}\nu\in \Tan(\phi,x)$ and thus $\Tan(\phi,x)$ is a cone.

Fix a point in $\mathbb{G}$ where the set $\mathfrak{B}(\phi,x)$ is a compact cone and thanks to \cref{prop:asdoubl}(iii) we infer there exists a $q>0$ such that $\nu(B(0,2r))\leq q\nu(B(0,r))$ for any $\nu\in\Tan(\phi,x)$ and every $r>0$. Let $\mathfrak{d}:=\inf\{\dist(z,U(0,1/2)^c):z\in B(0,1/4)\}>0$. We now prove that
\begin{equation}\label{eqn:AgognatoLimsup}
\limsup_{r\to 0}F_1(T_{x,2r}\phi)/F_1(T_{x,r}\phi)\leq 2d^{-1}q^2.
\end{equation}
Indeed, if by contradiction $r_i$ is an infinitesimal sequence such that $F_{1}(T_{x,2r_i}\phi)> 2\mathfrak{d}^{-1}q^2 F_{1}(T_{x,r_i}\phi)$, then for any $\nu\in\mathfrak{B}(\phi,x)$ we have
\begin{equation}
    F_1(T_{x,2r_i}\phi/F_1(T_{x,2r_i}\phi),\nu)\geq F_{1/2}(T_{x,2r_i}\phi/F_1(T_{x,2r_i}\phi) ,\nu)\geq F_{1/2}(\nu)-F_{1/2}(T_{x,2r_i}\phi)/F_1(T_{x,2r_i}\phi),
    \label{eq:ineq1}
\end{equation}
where the last inequality comes from \cref{rem:ScalinfFxr} and \cref{dist:lemma}. Furthermore, we also have for any $\nu\in\mathfrak{B}(\phi,x)$ that
\begin{equation}
    F_{1/2}(\nu)=\frac{F_{1/2}(\nu)}{F_{1}(\nu)}\geq \frac{\mathfrak{d}\nu(B(0,1/4))}{2\nu(B(0,1))}\geq \frac{\mathfrak{d}}{2q^2}.
    \label{eq:ineq2}
\end{equation}
Thanks to the absurd hypothesis and the fact that for any $s>0$ we have $F_s(T_{x,r}\phi)=sF_1(T_{x,rs}\phi)$, we infer that
\begin{equation}
    F_{1/2}(T_{x,2r_i}\phi)/F_1(T_{x,2r_i}\phi)=F_{1}(T_{x,r_i}\phi)/2F_1(T_{x,2r_i}\phi)\leq \mathfrak{d}/4q^2.
    \label{eq:ineq3}
\end{equation}
Putting \eqref{eq:ineq1}, \eqref{eq:ineq2} and \eqref{eq:ineq3} together, we conclude that
\begin{equation}
    F_1(T_{x,2r_i}\phi/F_1(T_{x,2r_i}\phi),\nu)\geq \mathfrak{d}/4q^2\geq\min\{\mathfrak{d}/4q^2,1/2\}=:\varepsilon,
    \label{eq:ineq4}
\end{equation}
for any $\nu\in\mathfrak{B}(\phi,x)$. Let us now denote, for simplicity, $\mathscr{T}:=\mathrm{Tan}(\phi,x)$. By taking into account the definition of $d_1$ in \eqref{eqn:drPreiss}, we get from the previous computations that $d_{1}(T_{x,2r_i}\phi,\mathscr{T})\geq \varepsilon$ for every $i$. Let us fix $\nu\in\Tan(\phi,x)$ such that $c_iT_{x,s_i}\phi\rightharpoonup \nu$ and let us note that \eqref{eqn:drPreiss} and \eqref{eqn:ContinuitydrPreiss} imply that $$\lim_{i\to 0}d_1(T_{x,s_i}\phi,\mathscr{T})=\lim_{i\to\infty}d_1(c_iT_{x,s_i}\phi,\mathscr{T})=d_1(\nu,\mathscr{T})=0.$$
Thanks to the above chain of identities, for $i$ sufficiently large, we denote by $\ell_i$ the smallest number among those $\ell\in[0,s_i]$ with the property that $d_1(T_{x,\eta},\mathscr{T})<\varepsilon$ for every $\ell<\eta\leq s_i$.  Since $d_1(T_{x,2r_i}\phi,\mathscr{T})\geq \varepsilon$ we conclude that $\ell_i>0$ for $i$ sufficiently large and $d_1(T_{x,\ell_i}\phi,\mathscr{T})= \varepsilon$ by the minimality of $\ell_i$ and the continuity of the map $\eta\mapsto d_1(T_{x,\eta}\phi,\mathscr{T})$. 

If, up to subsequences, $\ell_i/s_i\to_{i}t>0$, we conclude that, thanks to \eqref{eqn:ContinuitydrPreiss}, 
$$
d_1(T_{0,t}\nu,\mathscr{T})=\lim_{i\to+\infty}d_1(T_{x,ts_i}\phi,\mathscr{T})\geq \varepsilon,
$$
where the last inequality is true since $ts_i$ is arbitrarily near to $\ell_i$ for $i$ large enough, and $d_1(T_{x,\ell_i}\phi,\mathscr{T})\geq \varepsilon$. The previous inequality gives a contradiction since $T_{0,t}\nu\in\mathscr{T}$ and hence we should have $d_1(T_{0,t}\nu,\mathscr{T})=0$. Thus, $\ell_i/s_i\to 0$. This means that for every $r\geq 1$, taking into account \eqref{eqn:HomogeneityPreiss}, we have 
\begin{equation}\label{eqn:PreissLim}
\limsup_{i\to+\infty} d_r(T_{x,\ell_i}\phi,\mathscr{T})=\limsup_{i\to+\infty}d_1(T_{x,r\ell_i}\phi,\mathscr{T})\leq \varepsilon,
\end{equation}
since $\ell_i\leq r\ell_i\leq s_i$ for $i$ sufficiently large.  Since $\varepsilon< 1$, we have that $\lambda:=2/(1+\varepsilon)>1$, and hence, by \cref{rem:LambdaCones}, there exists $\tau>1$ such that $F_{\tau r}(\psi)\leq \lambda F_r(\psi)$ for every $\psi\in\mathscr{C}$ and for every $r>0$, since $\mathscr{C}$ has a compact basis.  Hence, taking \eqref{eqn:PreissLim} into account with $\tau r$ instead of $r$, we get that, whenever $r\geq 1$ and $i$ is sufficiently big, there exists $\psi\in\mathscr{T}$ with $F_{\tau r}(\psi)=1$ and $$
F_{\tau r}\left(\frac{T_{x,\ell_i}\phi}{F_{\tau r}(T_{x,\ell_i}\phi)},\psi\right)\leq\varepsilon/2.
$$
As a consequence, whenever $r\geq 1$ and $i$ is sufficiently big, by the triangle inequality for $F$ (cf. \cref{F1:metric}) and by the fact that $F_{\tau r}(\cdot)\geq F_{r}(\cdot)$, we get that 
$$
\frac{F_{r}(T_{x,\ell_i}\phi)}{F_{\tau r}(T_{x,\ell_i}\phi)}\geq F_r(\psi)-\varepsilon/2\geq \lambda^{-1}F_{\tau r}(\psi)-\varepsilon\geq 1/2.
$$
Hence, iterating, we have shown that there exists $\tau>1$ such that that for every $r\geq 1$ and every $p\in\mathbb N$, 
$$
\limsup_{i\to +\infty}\frac{F_{\tau^p r}(T_{x,\ell_i}\phi)}{F_{r}(T_{x,\ell_i}\phi)}<+\infty.
$$
By the arbitrariness of $p\in\mathbb N$ and $r\geq 1$, this implies that we are in a position to apply \cref{prop1.12preiss} to the sequence $\frac{T_{x,\ell_i}\phi}{F_1(T_{x,\ell_i}\phi)}$, which then converges, up to subsequences, to $\widetilde\nu\in\mathscr{T}$ with $F_1(\widetilde \nu)=1$. But then, by \eqref{eqn:ContinuitydrPreiss},
$$
d_1(\widetilde\nu,\mathscr{T})=\lim_{i\to+\infty}d_1(T_{x,\ell_i}\phi,\mathscr{T})\geq \varepsilon,
$$
that is a contradiction since $d_1(\widetilde\nu,\mathscr{T})=0$. Hence we finally have proven \eqref{eqn:AgognatoLimsup}.

Hence, from \eqref{eqn:AgognatoLimsup}, we deduce
$$
\limsup_{r\to 0}\frac{\phi(B(x,2r))}{\phi(B(x,r))} \leq\limsup_{r\to 0} \frac{2F_1(T_{x,4r}\phi)}{2^{-1}F_1(T_{x,r}\phi)}\leq 16\mathfrak{d}^{-2}q^4,
$$
whence the conclusion.
\end{proof}

Let us now prove a simple consequence of the previous Proposition.
\begin{proposizione}\label{prop:ConvergencePreiss}
Let $\phi$ be a Radon measure on $\mathbb G$ such that for $\phi$-almost every $x\in\mathbb G$ we have $\mathrm{Tan}(\phi,x)=\{\lambda\mathcal{S}^h\llcorner\mathbb V(x),\lambda>0\}$ for some homogeneous subgroup $\mathbb V(x)$ of Hausdorff dimension $h\in\mathbb N$. Then, for $\phi$-almost every $x\in\mathbb G$, the measure $T_{x,r}\phi/F_1(T_{x,r}\phi)$ weak* converges to $(h+1)\mathcal{C}^h\llcorner\mathbb V(x)$.
\end{proposizione}

\begin{proof}
For $\phi$-almost every $x\in\mathbb G$ we have that $\mathfrak{B}(\phi,x)=\{(h+1)\mathcal{C}^h\llcorner\mathbb V(x)\}$, taking into account \cite[Proposition 2.12]{antonelli2020rectifiable} and \cref{lemma:PALLAUNITARIAVOLUMEUNO}. Hence $\mathfrak{B}(\phi,x)$ is clearly compact for $\phi$-almost every $x\in\mathbb G$, and then $\phi$ is locally asymptotically doubling, due to \cref{prop:PhiAsDoubling}. Hence for every sequence $r_i\to 0$ we can extract a subsequence in $i$ such that $T_{x,r_i}\phi/F_1(T_{x,r_i}\phi)$ weak* converges to some $\nu\in\Tan(\phi,x)$, due to the fact that $\phi$ is locally asymptotically doubling and thus the hypothesis of \cref{prop1.12preiss} is verified. Since $F_1(\nu)=1$ by continuity of $F_1$, we conclude that $\nu=(h+1)\mathcal{C}^h\llcorner\mathbb V(x)$. Thus, being the sequence $r_i$ arbitrary, we obtain the thesis.
\end{proof}

The following proposition, which is inspired by \cite[4.4(4)]{Preiss1987GeometryDensities}, will be of crucial importance in the proof of the two fundamental results of this section, namely \cref{prop:TangentLocalizedCone}, and \cref{prop:TangentCompemented}.

\begin{proposizione}\label{prop:bdonplanes}
Let $0<\sigma<1/5$, $\phi$ be a Radon measure on $\mathbb{G}$, $h\in\{1,\ldots, Q\}$, and
$d_{z,t}(\phi,\mathfrak{M}(h,\{\mathbb{V}\}))\leq \sigma^{h+4}$,
then
$$\phi(B(y,s)\cap B(y\mathbb{V},\sigma^2t/(h+1)))\geq (1-5\sigma)(s/r)^h\phi(B(x,r)),$$
whenever $x,y\in z\mathbb{V}\cap B(z,(1-\sigma)t)$, $\sigma t\leq r\leq (1-\sigma)t-\lVert z^{-1}x\rVert$, and $\sigma t\leq s\leq (1-\sigma)t-\lVert z^{-1}y\rVert$.
\end{proposizione}

\begin{proof}
The definition of $d_{z,t}(\cdot,\mathfrak{M}(h,\{\mathbb{V}\}))$ implies that
$$F_{0,1}(T_{z,t}\phi/F_1(T_{z,t}\phi),(h+1)\mathcal{C}^h\llcorner \mathbb{V})\leq \sigma^{h+4},$$
and thus up to redefining $\phi$ we can assume without loss of generality that $z=0$, $t=1$ and that $F_1(\phi)=1$.
Thus, let $q:=\sigma^2/(h+1)$, $x\in \mathbb{V}$ and $r>0$ as in the hypothesis of the proposition. Define $$g(w):=\min\{1,\dist(w,\mathbb{G}\setminus B(x,r+q))/q\}.$$
Notice that $B(x,r)\Subset B(0,1)$, and thanks to the assumptions on $\phi$ we infer that, calling $\mathrm{Lip}(g)$ the Lipschitz constant of the function $g$,
\begin{equation}
\begin{split}
     \phi(B(x,r))&\leq \int g(w)d\phi(w)
     \leq(h+1)\int g(w) d\mathcal{C}^h\llcorner \mathbb{V}(w)+\text{Lip}(g)F_{0,1}(\phi,(h+1)\mathcal{C}^h\llcorner \mathbb{V}(w))\\
     &\leq(h+1)\mathcal{C}^h\llcorner \mathbb{V}(B(x,r+q))+\sigma^{h+4}/q=(h+1)(r+q)^h+\sigma^{h+4}/q.
     \label{eq:ineq10}
\end{split}
\end{equation}
With the same argument used above,see \cite[Equation (37)]{antonelli2020rectifiable}, for any $y$ and $s>0$ as in the hypothesis of the proposition one can also show that
\begin{equation}
    (h+1)(s-q)^h=(h+1)\mathcal{C}^h\llcorner \mathbb{V}(B(y,s-q))\leq \phi(B(y,s)\cap B(\mathbb{V},q))+\sigma^{h+4}/q.
    \label{eq:ineq11}
\end{equation}
Thus, putting together \eqref{eq:ineq10} and \eqref{eq:ineq11} we infer that
\begin{equation}
    \begin{split}
        \frac{\phi(B(y,s)\cap B(\mathbb{V},q))}{\phi(B(x,r))}&\geq \frac{(h+1)(s-q)^h-\sigma^{h+4}/q}{(h+1)(r+q)^h+\sigma^{h+4}/q}=\left(\frac{s}{r}\right)^h\frac{\left(1-\frac{\sigma^2}{s(h+1)}\right)^h-\frac{\sigma^{h+2}}{s^h}}{\left(1+\frac{\sigma^2}{r(h+1)}\right)^h+\frac{\sigma^{h+2}}{r^h}} \\
        &\geq \left(\frac{s}{r}\right)^h
       \frac{\left(1-\frac{\sigma}{h+1}\right)^h-\sigma^2}{\left(1+\frac{\sigma}{h+1}\right)^h+\sigma^2}\geq \left(\frac{s}{r}\right)^h\frac{1-h/(h+1)\sigma-\sigma^2}{1+2h/(h+1)\sigma+\sigma^2} \geq \left(\frac{s}{r}\right)^h\frac{1-2\sigma}{1+3\sigma}\geq (1-5\sigma)\left(\frac{s}{r}\right)^h
       ,
        \nonumber
    \end{split}
\end{equation}
where in the third inequality above we are using that $\sigma\leq r$ and $\sigma\leq s$; in the fourth inequality we are using that $(1-\sigma/(h+1))^h\geq 1-h/(h+1)\sigma$ by Bernoulli inequality, and $(1+\sigma/(h+1))^h\leq 1+2h/(h+1)\sigma$, which can be easily verified by induction since $2h\sigma/(h+1)\leq 1$.
\end{proof}

Before proving the main results of this section, namely \cref{prop:TangentLocalizedCone}, and \cref{prop:TangentCompemented}, we now state and prove three measurability results that will play a crucial role in the rest of the paper. Roughly speaking, we prove that when a measure has unique tangents (or unique approximate tangents), the map that associates a point $x\in\mathbb G$ to its tangent (or approximate tangent) is measurable.

\begin{lemma}\label{lem:BorelTangents3}
Let $\phi$ be a Radon measure such that, for $\phi$-almost every $x\in\mathbb G$, there exists $\tau(\phi,x)\in \G(h)$ such that
$$
\mathrm{Tan}(\phi,x)=\{\lambda\mathcal{C}^h\llcorner \tau(\phi,x):\lambda>0\}.
$$
Then the map $x\mapsto \tau(\phi,x)$ is $\phi$-measurable as a map from $\mathbb G$ to $\G(h)$.
\end{lemma}
\begin{proof}
First of all, from \cref{prop:PhiAsDoubling} we get that $\phi$ is locally asymptotically doubling.  We let $\{\mathbb{V}_\ell\}_{\ell\in \N}$ be a countable dense set in $\G(h)$ that exists thanks to the compactness of the Grassmanian, see \cref{prop:CompGrassmannian}. Furthermore, for any $r\in (0,1)\cap \Q$ any $\varepsilon>0$, and any $\ell\in\N$ we define the function
$$
f_{r,\ell,\varepsilon}(x):=\phi(B(x,r))^{-1}\phi(\{w\in B(x,r):\dist(x^{-1}w,\mathbb{V}_\ell)\geq\varepsilon\lVert x^{-1}w\rVert\})=:\phi(B(x,r))^{-1}\phi(I(x,r,\ell,\varepsilon)),
$$
when $\phi(B(x,r))>0$ and we set it to be $+\infty$ if $\phi(B(x,r))=0$.
We claim that the functions $f_{r,\ell,\varepsilon}$ are upper semicontinuous. Let $\{x_i\}_{i\in\N}$ be a sequence of points converging to some $x\in\mathbb{G}$. If $\phi(B(x,r))=0$ the upper semicontinuity on the sequence $\{x_i\}_{i\in\mathbb N}$ is trivially verified by definition of $f_{r,\ell,\varepsilon}$. So let us assume without loss of generality that $\phi(B(x,r))>0$. Since $x_i\to x$ and $\phi$ is a Radon measure we have $\phi(B(x,r))= \lim_i \phi(B(x_i,r))$, and then we can assume withouot loss of generality that $\phi(B(x_i,r))>0$ for every $i$.

Since the sets $I(x_i,r,\ell,\varepsilon)$ are contained in $B(x,2)$ provided $i$ is sufficiently big, we infer thanks to Fatou's Lemma that
\begin{equation}
\limsup_{i\to+\infty}f_{r,\ell,\varepsilon}(x_i)=\limsup_{i\to+\infty}\phi(B(x_i,r))^{-1}\int \chi_{I(x_i,r,\ell,\varepsilon)}(z)d\phi(z)\leq \phi(B(x,r))^{-1}\int\limsup_{i\to+\infty} \chi_{I(x_i,r,\ell,\varepsilon)}(z)d\phi(z).
\label{eq:::num20}
\end{equation}
Furthermore, since $x_i\to x$ and the sets $I(x_i,r,\ell,\varepsilon)$ and $I(x,r,\ell,\varepsilon)$ are closed, we have
$$\limsup_{i\to+\infty} \chi_{I(x_i,r,\ell,\varepsilon)}=\chi_{\limsup_{i\to+\infty} I (x_i,r,\ell,\varepsilon)}\leq \chi_{I(x,r,\ell,\varepsilon)},
$$ 
where the first equality is true in general. Then, from \eqref{eq:::num20}, we infer that
\begin{equation}
\limsup_{i\to\infty}f_{r,\ell,\varepsilon}(x_i)\leq \phi(B(x,r))^{-1}\int\limsup_{i\to+\infty} \chi_{I(x_i,r,\ell,\varepsilon)}(z)d\phi(z)\leq \phi(B(x,r))^{-1}\int \chi_{I(x,r,\ell,\varepsilon)}(z)d\phi(z)=f_{r,\ell,\varepsilon}(x),
\nonumber
\end{equation}
and this concludes the proof that $f_{r,\ell,\varepsilon}$ is upper semicontinuous. This shows that for every $\ell\in\mathbb N$ and $\varepsilon>0$, the function
$$
f_{\ell,\varepsilon}:=\liminf_{r\in\Q\cap(0,1),r\to 0} f_{r,\ell,\varepsilon},
$$ 
is $\phi$-measurable. Hence also $\widetilde f_{\ell,\varepsilon}:=\sup_{\widetilde{\varepsilon}\in\mathbb Q,\widetilde{\varepsilon}>\varepsilon}f_{\ell,\varepsilon}$ is $\phi$-measurable. As a consequence, since $\mathrm{Tan}(\phi,x)= \{\lambda\mathcal{C}^h\llcorner \tau(\phi,x):\lambda>0\}$ for $\phi$-almost every $x\in\mathbb{G}$, we infer that the set
\begin{equation}\label{eqn:FIRSTINCLUSION}
\begin{split}
     B_{\ell,\varepsilon}&:=\big\{x\in\mathbb{G}:\widetilde f_{\ell,\varepsilon}(x)=0\}\cap \{x\in \mathbb{G}: \text{there exists $\tau(\phi,x)$}\big\}=\{x\in \mathbb{G}:\tau(\phi,x) \subseteq C_{\mathbb V_\ell}(\varepsilon)\},
\end{split}
\end{equation}
is $\phi$-measurable as well. Let us justify the last equality in the previous line. If $\widetilde f_{\ell,\varepsilon}(x)=0$, then $f_{\ell,\widetilde{\varepsilon}}(x)=0$ for every $\widetilde{\varepsilon}>\varepsilon$, $\widetilde\varepsilon\in\mathbb Q$. Hence, arguing as in \cite[Proof of Proposition 5.5]{antonelli2020rectifiable}, in particular as in the lines slightly above \cite[Equation (109)]{antonelli2020rectifiable}, we get $\tau(\phi,x) \subseteq C_{\mathbb V_\ell}(\widetilde{\varepsilon})$ for every $\widetilde\varepsilon\in\mathbb Q$ and $\widetilde\varepsilon>\varepsilon$. Let us explain this with further details. We first get that there exist $r_i\to 0$ such that $f_{r_i,\ell,\widetilde\varepsilon}(x)\to_{i}0$. Since $\phi$ is locally asymptotically doubling, thanks to \cref{prop1.12preiss} we deduce that $\phi(B(x,r_i))^{-1}T_{x,r_i}\phi$ converges, up to subsequences, to some tangent measure $\nu\in \Tan(\phi,x)$, and then from the hypothesis we have $\nu=\lambda\mathcal{C}^h\llcorner\tau(\phi,x)$, for some $\lambda>0$. Then the same computations in the two displayed equations before \cite[Equation (109)]{antonelli2020rectifiable} give the sought conclusion. Taking $\widetilde\varepsilon\to\varepsilon$ we get the first inclusion of \eqref{eqn:FIRSTINCLUSION}. On the other hand, if $\tau(\phi,x) \subseteq C_{\mathbb V_\ell}(\varepsilon)$, we get that $\tau(\phi,x)\subseteq \{w\in\mathbb G:\dist(w,\mathbb V_\ell)<\widetilde\varepsilon\|w\|\}$ for every $\widetilde\varepsilon>\varepsilon$. Hence we can argue as in \cite[Equation (111)]{antonelli2020rectifiable} in order to obtain that $f_{\ell,\widetilde\varepsilon}(x)=0$ for every $\widetilde\varepsilon>\varepsilon$ and then passing to the limit as $\widetilde\varepsilon\to\varepsilon$ we get the sought conclusion.

In order to prove that the map $x\mapsto \tau(\phi,x)$ is $\phi$-measurable, it suffices to check that the for any open $\Omega\subseteq \G(h)$ we have that $\tau^{-1}(\Omega)$ is $\phi$-measurable. To show this we note that, thanks to \cite[Lemma 2.15]{antonelli2020rectifiable}, there is a sequence of radii $r_k>0$ such that
$$
\Omega=\bigcup_{\substack{k\in\N\\ \mathbb{V}_k\in\Omega}} \{\mathbb{W}\in \G(h):\mathbb W\subseteq C_{\mathbb V_k}(r_k)\}.
$$
This implies that, up to $\phi$-null sets, $\tau^{-1}(\Omega)= \bigcup_{k\in\N} B_{k,r_k}$,
which thanks to the above discussion is a $\phi$-measurable set.
\end{proof}

\begin{lemma}\label{lem:BorelTangents}
Let $\phi$ be a $\mathscr{P}_h$-rectifiable measure. Denote $\tau(\phi,x)$ to be the unique element of $\G(h)$, that exists $\phi$-almost everywhere by definition, for which
$$
\mathrm{Tan}_h(\phi,x)\subseteq\{\lambda\mathcal{C}^h\llcorner \tau(\phi,x):\lambda>0\}.
$$
Then the map $x\mapsto \tau(\phi,x)$ is $\phi$-measurable as a map from $\mathbb{G}$ to $\G(h)$.
\end{lemma}

\begin{proof}
From a routine argument (cf. \cite[Remark 14.4(3)]{Mattila1995GeometrySpaces}), we get that $\mathrm{Tan}(\phi,x)=\{\lambda\mathcal{C}^h\llcorner\tau(\phi,x)\}$ for $\phi$-almost every $x\in\mathbb G$. Hence we can apply \cref{lem:BorelTangents3}
\end{proof}

The proof of the following lemma follows as the ones above. We omit the details.

\begin{lemma}\label{prop:approxtangentismeasurable}
Suppose $d$ is a homogeneous left-invariant metric on $\mathbb{G}$, let $E$ be a Borel set of finite $\mathcal{S}^h_{d}$-measure, and suppose that for $\mathcal{S}^h_{d}$-almost every $x\in E$ there exists $\mathbb V(x)\in\G(h)$ for which for any $0<\varepsilon<1$ and any $0<\beta<1$ there exist a $\rho(x,\varepsilon,\beta)>0$ such that
\begin{equation}\label{eqn:ConditionApprox}
\mathcal{S}^h_{d}\llcorner E(B_{d}(x,r)\setminus xC_{\mathbb{V}(x),{d}}(\beta))\leq \varepsilon \mathcal{S}^h_{d}\llcorner E(B_d(x,r)),
\end{equation}
for any $0<r<\rho(x,\varepsilon,\beta)$. Then the map $x\mapsto \mathbb{V}(x)$ from $E$ to $\G(h)$ is $\mathcal{S}^h_{d}\llcorner E$-measurable.
\end{lemma}

Notice that the previous statement could be also obtained arguing as in \cite[Proposition 3.9]{MatSerSC}, after having noticed that, since $2^{-h}\leq \Theta^{h,*}_{d}(\mathcal{S}^h_{d}\llcorner E,x)\leq 1$ for $\mathcal{S}^h_{d}\llcorner E$-almost every $x\in\mathbb G$ due to \cite[2.10.19(1) and 2.10.19(5)]{Federer1996GeometricTheory}, the condition \eqref{eqn:ConditionApprox} is equivalent to asking that $\mathbb V(x)$ is an approximate tangent plane to $E$ at $x$ in the sense of \cite[Equation (3.2)]{MatSerSC}.

\begin{osservazione}
The results in \cref{lem:BorelTangents3}, \cref{lem:BorelTangents}, and \cref{prop:approxtangentismeasurable} are readily true also when we allow $\tau(\phi,x)$ (or $\mathbb V(x)$) to be in some Borel subset of $\G(h)$.
\end{osservazione}

\begin{proposizione}\label{prop:TangentLocalizedCone}
Suppose $\phi$ is a Radon measure on $\mathbb{G}$ such that, for $\phi$-almost every $x\in\mathbb G$, we have $\Tan(\phi,x)=\{\lambda\mathcal{S}^h\llcorner \mathbb{V}(x):\lambda>0\}$ for some $\mathbb{V}(x)\in\G(h)$. Then, for every $\alpha\in(0,1)$ there exist $\{\mathbb V_i\}_{i\in\mathbb N}\subseteq \G(h)$, and a family of compact $C_{\mathbb V_i}(\alpha)$-sets $\{\Gamma_i\}_{i\in\mathbb N}$ such that 
$$
\phi(\mathbb G\setminus \cup_{i\in\mathbb N}\Gamma_i)=0.
$$
\end{proposizione}

\begin{proof}
First of all, by \cref{prop:PhiAsDoubling}, the measure $\phi$ is locally asymptotically doubling. Up to restricting $\phi$ to closed balls and by using the locality of tangents in \cref{prop:Lebesuge} and Lebesgue Theorem in \cref{prop:Lebesuge}, we may assume that $\phi$ is supported on a compact set $K$ and that it is still locally asymptotically doubling. Let $S$ be dense countable subset of $(\G(h),d_\mathbb{G})$, that exists thanks to \cref{prop:CompGrassmannian}. Thanks to \cite[Proposition 2.29]{antonelli2020rectifiable}, we infer that also the countable set $\{(h+1)\mathcal{C}^{h}\llcorner \mathbb{W}:\mathbb{W}\in S\}$ is dense in the metric space $(\{(h+1)\mathcal{C}^{h}\llcorner \mathbb{V}:\mathbb{V}\in \G(h)\},F_{0,1})$.


Let us now fix $\mathfrak{e}<1/10$, $\sigma<1/100(\mathfrak{e}/(3(1+\mathfrak{e})))^h$,$\mathbb V\in S$ and let us denote 
$$
K_{\mathbb V}:=\{x\in K:F_{0,1}((h+1)\mathcal{C}^h\llcorner \mathbb{V},(h+1)\mathcal{C}^h\llcorner \mathbb{V}(x))<\sigma^{h+4}\},
$$
where $\mathbb V(x)\in\G(h)$ is such that $\mathrm{Tan}(\phi,x)=\{\lambda\mathcal{S}^h\llcorner\mathbb V(x),\lambda>0\}$. Since $\{(h+1)\mathcal{C}^{h}\llcorner \mathbb{W}:\mathbb{W}\in S\}$ is dense in the metric space $(\{(h+1)\mathcal{C}^{h}\llcorner \mathbb{V}:\mathbb{V}\in \G(h)\},F_{0,1})$ we conclude that $K=\cup_{\mathbb V\in S}K_{\mathbb V}$. By \cref{lem:BorelTangents3}, one gets that $K_{\mathbb V}$ is $\phi$-measurable for every $\mathbb V\in S$. Thus by \cref{prop:Lebesuge},  we can assume without loss of generality that $\phi$ is locally asymptotically doubling and supported on $K_{\mathbb V}$ for some $\mathbb V\in S$, which from now on we fix.

We now claim that for $\phi$-almost every $x\in\mathbb G$ the following holds
 \begin{equation}\label{eqn:CLAIMCLAIM}
 \lim_{r\to 0^+}d_{x,r}(\phi,\mathfrak{M}(h,\{\mathbb V\}))= F_{0,1}((h+1)\mathcal{C}^h\llcorner \mathbb{V},(h+1)\mathcal{C}^h\llcorner \mathbb{V}(x)).
 \end{equation}
 Indeed, for $\phi$-almost every $x\in\mathbb G$ the measure $T_{x,r}\phi/F_1(T_{x,r}\phi)$ converges to $(h+1)\mathcal{C}^h\llcorner \mathbb V(x)$ as $r\to 0^+$, see \cref{prop:ConvergencePreiss} and thus, from the definition of $d_{x,r}$, we get that 
 \begin{equation}\label{eqn:CiFaVedereContinuita}
 d_{x,r}(\phi,\mathfrak{M}(h,\{\mathbb V\}))=F_{0,1}(T_{x,r}\phi/F_1(T_{x,r}\phi),(h+1)\mathcal{C}^h\llcorner\mathbb V),
 \end{equation}
 from which we deduce the claim \eqref{eqn:CLAIMCLAIM} by using the previous convergence and the continuity of $F_{0,1}$, see, e.g., \cite[Proposition 1.10]{MarstrandMattila20} or \cref{prop1.12preiss}. Moreover, the function $x\to d_{x,r}(\phi,\mathfrak{M}(h,\{\mathbb V\}))$ is continuous in $x$ for every $r>0$. Indeed, by \eqref{eqn:CiFaVedereContinuita} and the continuity of $F_{0,1}$, it is sufficent to see that, for every $r>0$, the map $x\to T_{x,r}\phi/F_1(T_{x,r}\phi)$ is continuous from $\mathbb G$ to the space of Radon measures equipped with the weak* convergence, which is clear again by the continuity of $F_1(\cdot)$ and by the continuity of the map $x\to T_{x,r}\phi$, which is readily verified (see, e.g., the computations at the end of \cite[page 22]{MarstrandMattila20}). Hence, by using Severini-Egoroff Theorem, we can assume without loss of generality that $\phi$ is supported on a compact set $E$ such that $\diam(E)<s$ and such that  $d_{x,r}(\phi,\mathfrak{M}(h,\{\mathbb{V}\}))<\sigma^{h+4}$ whenever $ x\in E$ and $r\in(0,400(h+1)s)$. Let us now fix $\widetilde x,\widetilde y\in E$ and denote $a:=d(\widetilde x,\widetilde y)$, $\widetilde t:=2a(1+\mathfrak e)$, $\widetilde r:=a(1+\mathfrak e)$ and $\widetilde s:=a\mathfrak e$.
 
 Let us apply \cref{prop:bdonplanes} first with the choices $x=y=z=\widetilde y$, $s=r=\widetilde r$, $t=\widetilde t$ and $\sigma$ as above, that yields
\begin{equation}
    \phi(B(\widetilde y,\widetilde r)\cap B(\widetilde y\mathbb{V},\sigma^2\widetilde t/(h+1)))\geq (1-5\sigma)\phi(B(\widetilde y,\widetilde r)),
    \label{eq:num:num1}
\end{equation}
and secondly with $x=y=z=\widetilde x$, $r=\widetilde r+a$, $s=\widetilde s$, $t=3a(1+\mathfrak e)$ and $\sigma$, we get
\begin{equation}
    \phi(B(\widetilde x,\widetilde s)\cap B(\widetilde x\mathbb{V},\sigma^2\cdot 3a(1+\mathfrak e)/(h+1)))\geq (1-5\sigma)(\widetilde s/(\widetilde r+a))^h\phi(B(\widetilde x,\widetilde r+a)).
      \label{eq:num:num2}
\end{equation}
Putting together \eqref{eq:num:num1} and \eqref{eq:num:num2}, we conclude that
 \begin{equation}
     \begin{split}
         \phi(B(\widetilde y,\widetilde r)\setminus B(\widetilde y\mathbb{V},2a\sigma^2(1+\mathfrak{e})/(h+1)))&=\phi(B(\widetilde y,\widetilde r))-\phi(B(\widetilde y,\widetilde r)\cap B(\widetilde y\mathbb{V},2a\sigma^2(1+\mathfrak{e})/(h+1))) \\
         \leq 5\sigma \phi(B(\widetilde y,\widetilde r))& \leq 5\sigma \phi(B(\widetilde x,\widetilde r+a))\leq \frac{5\sigma}{1-5\sigma}\left(\frac{2(1+\mathfrak e)}{\mathfrak e}\right)^h\phi(B(\widetilde x,\widetilde s))
         \label{eq:num:num:3}
     \end{split}
 \end{equation}
 If by contradiction $B(\widetilde x,\mathfrak{e}a)\cap B(\widetilde y\mathbb{V},2a\sigma^2(1+\mathfrak{e})/(h+1)))= \emptyset$ then from \eqref{eq:num:num:3} and the fact that $B(\widetilde x,\widetilde s)\subseteq B(\widetilde y,\widetilde r)$, we infer
 \begin{equation}
     \begin{split}
\phi(B(\widetilde x,\widetilde s))\leq \frac{5\sigma}{1-5\sigma}\left(\frac{2(1+\mathfrak e)}{\mathfrak e}\right)^h\phi(B(\widetilde x,\widetilde s)),
\nonumber
     \end{split}
 \end{equation}
 that is in contradiction thanks with the choice of $\sigma$. Hence, for every $\widetilde x,\widetilde y\in E$ we have that $B(\widetilde x,\mathfrak{e}a)\cap B(\widetilde y\mathbb{V},2a\sigma^2(1+\mathfrak{e})/(h+1)))\neq \emptyset$ and thus $d(\widetilde x,\widetilde y\mathbb V)\leq a(\mathfrak e+2\sigma^2(1+\mathfrak e)/(h+1))=d(\widetilde x,\widetilde y)(\mathfrak e+2\sigma^2(1+\mathfrak e)/(h+1))$. Hence, the compact set $E$ is a $C_{\mathbb V}(\mathfrak e+2\sigma^2(1+\mathfrak e)/(h+1))$-set. Since it is clear that, for any given $\alpha>0$, $\sigma$ and $\mathfrak e$ can be chosen small enough in order to have $\mathfrak e+2\sigma^2(1+\mathfrak e)/(h+1)<\alpha$, the proof is thus concluded.
 \end{proof}

In the case the tangents are complemented we can give the following improvement of the latter Proposition.
\begin{proposizione}\label{prop:TangentCompemented}
Let $1\leq h\leq Q$ be a natural number. There exist $\{\mathbb V_i\}_{i\in\mathbb N}\subseteq \G_c(h)$ and $\mathbb L_i$ complementary subgroups of $\mathbb V_i$ such that the following holds.

Suppose $\phi$ is a Radon measure on $\mathbb{G}$ such that, for $\phi$-almost every $x\in\mathbb G$, we have $\Tan(\phi,x)=\{\lambda\mathcal{S}^h\llcorner \mathbb{V}(x):\lambda>0\}$ for some $\mathbb{V}(x)\in\G_c(h)$. Then, for every $\alpha\in(0,1)$ there exists a family of compact sets $\{\Gamma_i\}_{i\in\mathbb N}$ such that 
$$
\phi(\mathbb G\setminus \cup_{i\in\mathbb N}\Gamma_i)=0,
$$
and, for every $i\in\mathbb N$, $\Gamma_i$ is a compact intrinsic Lipschitz graph, which is also a $C_{\mathbb V_i}(\alpha)$-set, of a map $\varphi_i:A_i\subseteq \mathbb V_i\to\mathbb L_i$, where $A_i$ is compact.
\end{proposizione}

\begin{proof}
The proof follows exactly the same lines as the proof of \cref{prop:TangentLocalizedCone}, so we just sketch it underlying the main differences. For every 
$\ell\in\N$, with $\ell\geq 2$, let us define
$$
\G_c(h,\ell):=\{\mathbb V\in \G_c(h):\exists\, \mathbb L\,\, \text{complement of}\,\, \mathbb V
 \,\,\text{s.t.}\,\, 1/\ell <\oldep{ep:Cool}(\mathbb V,\mathbb L)\leq 1/(\ell-1)\}.
$$ 
Observe that \cref{prop:CompGrassmannian} implies that $\G_c(h,\ell)$ is separable for any $\ell\in\N$, since $\G_c(h,\ell)\subseteq \G(h)$ and $(\G(h),d_{\mathbb G})$ is a compact metric space. Let \begin{equation}\label{eqn:Dl}
 \mathscr{D}_\ell:=\{\mathbb V_{i,\ell}\}_{i\in\N},
 \end{equation} 
 be a countable dense subset of $\G_c(h,\ell)$ and
 \begin{equation}\nonumber
     \text{for all $i\in\N$, choose a complement $\mathbb L_{i,\ell}$ of $\mathbb V_{i,\ell}$ s.t. $1/\ell <\oldep{ep:Cool}(\mathbb V_{i,\ell},\mathbb L_{i,\ell})\leq 1/(\ell-1)$}.
 \end{equation}
 Now, let $S:=\{\mathbb V_{i,\ell}\}_{i,\ell\in\mathbb N}$, which is a dense countable subset of $(\G_c(h),d_\mathbb{G})$ thanks to the definition given above. As in the above proposition we infer that also the countable set $\{(h+1)\mathcal{C}^{h}\llcorner \mathbb{W}:\mathbb{W}\in S\}$ is dense in the metric space $(\{(h+1)\mathcal{C}^{h}\llcorner \mathbb{V}:\mathbb{V}\in \G_c(h)\},F_{0,1})$. Let us now fix, for every $\ell\in\mathbb N$,  $\mathfrak{e}_\ell<\min\{1/10,1/(2\ell),\alpha/2\}$, where $\alpha$ is as in the statement, and $\sigma_\ell<\min\{1/100(\mathfrak{e}_\ell/(3(1+\mathfrak{e}_\ell)))^h,\sigma'_\ell\}$, where $\sigma'_\ell$ is chosen small enough such that $\mathfrak e_\ell + 2(\sigma_\ell')^2(1+\mathfrak e_\ell)/(h+1)<\min\{\alpha,1/\ell\}$. Moreover, for every $\mathbb V_{i,\ell}\in \mathscr{D}_\ell$, let us denote 
$$
K_{\mathbb V_{i,\ell}}:=\{x\in K:F_{0,1}((h+1)\mathcal{C}^h\llcorner \mathbb{V}_{i,\ell},(h+1)\mathcal{C}^h\llcorner \mathbb{V}(x))<\sigma_\ell^{h+4}\},
$$
where $\mathbb V(x)$ is the element of $\G(h)$ for which $\mathrm{Tan}(\phi,x)=\{\lambda\mathcal{S}^h\llcorner\mathbb V(x),\lambda>0\}$. Arguing as in the above proposition, being $K$ the compact set on which we can assume $\phi$ is supported without loss of generality, we have $K=\cup_{\ell\in\mathbb N}\cup_{\mathbb V_{i,\ell}\in\mathscr{D}_\ell}K_{\mathbb V_{i,\ell}}$. Hence, we can assume without loss of generality that $\phi$ is supported on $K_{\mathbb V_{i,\ell}}$ for some $\mathbb V_{i,\ell}$. The computations in the proof of the above proposition can be repeated substituting $\sigma_\ell$ with $\sigma$ accordingly, allowing us to conclude that $\phi$-almost every $K_{\mathbb V_{i,\ell}}$ can be covered by compact sets that are $C_{\mathbb V_{i,\ell}}(\mathfrak e_\ell + 2(\sigma_\ell)^2(1+\mathfrak e_\ell)/(h+1))$. By the very choice of $\sigma'_\ell$ this implies that the latter compact sets are $C_{\mathbb V_{i,\ell}}(\min\{\alpha,1/\ell\})$-sets, and since $1/\ell<\oldep{ep:Cool}(\mathbb V_{i,\ell},\mathbb L_{i,\ell})$, we also conclude that they are graphs according to the splitting $\mathbb G=\mathbb V_{i,\ell}\cdot\mathbb L_{i,\ell}$, see \cref{prop:ConeAndGraph}.
\end{proof}

\section{From flat tangents to $\mathscr{P}$-rectifiability}\label{sec:Density}
In this section we first prove that, in an arbitrary Carnot group, having flat (complemented) tangent measures à la Preiss implies being $\mathscr{P}$-rectifiable, see \cref{thm:ExistenceOfDensityPlus}. Then we will prove a rectifiable criterion, see \cref{prop:rett.1}, which will allow us to complete the proof of \cref{thm:INTRO1Equivalence}. 

In this section a Carnot group $\mathbb G$ will be fixed, along with a left-invariant homogeneous distance on it, that sometimes will be understood.
Throughout this section we assume that $\mathbb{V}\in\G_c(h)$ and that $\mathbb{V}\mathbb{L}=\mathbb{G}$. In this chapter whenever we deal with $C_{\mathbb V}(\alpha)$-sets we are tacitly assuming that $\alpha\leq\oldep{ep:Cool}(\mathbb V,\mathbb L)$.
\medskip

\subsection{From flat Preiss's tangents to $\mathscr{P}$-rectifiability}

We are going to prove that whenever a measure has flat (complemented) tangents \'a la Preiss, then it is $\mathscr{P}$-rectifiable. Throughout this section we assume that $\mathbb{V}\in\G_c(h)$ and that $\mathbb L$ is a complementary subgroup of $\mathbb V$. Let us begin with a proposition that roughly tells us the following. If $\Gamma$ is a compact $C_{\mathbb V}(\alpha)$-set with $\alpha\leq \oldep{ep:Cool}(\mathbb V,\mathbb L)$, and moreover we know that the measure $\mathcal{S}^h\llcorner\Gamma$ is locally asymptotically doubling, hence $\Gamma$ has big projections on $\mathbb V$. This will allow us to prove that the lower $h$-density of $\mathcal{S}^h\llcorner\Gamma$ is positive almost everywhere, see \cref{prop:DensitaInfPositiva}. The latter conclusion eventually leads to the following result: if a set has flat (complemented) Preiss's tangents, then it is $\mathscr{P}^c$-rectifiable, see \cref{thm:ExistenceOfDensityPlus}. 

Let us start by recalling two results from \cite{antonelli2020rectifiable} and proving an adaption of \cite[Proposition 4.6]{antonelli2020rectifiable} to our context. 

\begin{lemma}[{\cite[Lemma 4.2]{antonelli2020rectifiable}}]\label{lemma:tec.cono}
Suppose $d$ is a homogeneous left-invariant metric on $\mathbb{G}$. Then, there exists an $A:= A(d,\mathbb V,\mathbb L)>1$ such that for any $w\in B_d(0,1/5A)$, any $y\in \partial B_d(0,1)\cap C_{\mathbb{V},d}(\oldep{ep:Cool}(d,\mathbb{V},\mathbb{L}))$ and any $z\in B_d(y,1/5A)$, we have
$w^{-1}z\not\in \mathbb{L}$.
\end{lemma}

\begin{proposizione}[{\cite[Proposition 4.3]{antonelli2020rectifiable}}]\label{prop:palleseparate}
Suppose $d$ is a homogeneous left-invariant metric on $\mathbb{G}$. Let us fix $\alpha<\oldep{ep:Cool}(d,\mathbb V,\mathbb L)$ and suppose $\Gamma$ is a $C_{\mathbb{V},d}(\alpha)$-set. For any $x\in \Gamma$ let $\rho(x)$ to be the biggest number satisfying the following condition. For any $0<r<\rho(x)$ and any $y\in B(x,r)\cap \Gamma$ we have
    $$P_\mathbb{V}(B_d(x,r)) \cap P_\mathbb{V}(B_d(y,s))=\emptyset\,\, \text{ for any }r,s<d(x,y)/5A,$$
where $A=A(d,\mathbb V,\mathbb L)$ is the constant yielded by \cref{lemma:tec.cono}. Then, the function $x\mapsto \rho(x)$ is positive everywhere on $\Gamma$ and upper semicontinuous.  
\end{proposizione}

\begin{proposizione}\label{prop:proj}
Let $\alpha\leq\oldep{ep:Cool}(\mathbb V,\mathbb L)$ and suppose $\Gamma$ is a compact $C_\mathbb{V}(\alpha)$-set of $\mathcal{S}^h$-finite measure such that for  $\mathcal{S}^h\llcorner \Gamma$-almost every $x\in\mathbb{G}$ we have
$$\limsup_{r\to 0}\frac{\mathcal{S}^h\llcorner \Gamma(B(x,2r))}{\mathcal{S}^h\llcorner \Gamma(B(x,r))}<\infty.$$
Then, there exists a constant $\newC\label{C:proj}:=\oldC{C:proj}(h,A)>0$ such that for $\mathcal{S}^h$-almost every $x\in \Gamma$ there exists an infinitesimal sequence $\{\ell_i(x)\}_{i\in\N}$ such that for any $i\in\N$ we have
\begin{equation}
    \mathcal{S}^h(P_{\mathbb{V}}(\Gamma\cap B(x,\ell_i(x))))>\oldC{C:proj}\ell_i(x)^h
    \label{eq:bigproj}
\end{equation}
\end{proposizione}

\begin{proof}
We will just sketch the proof, that is an adaptation of \cite[Proposition 4.6]{antonelli2020rectifiable} and we refer the reader to \cite[Proposition 4.6]{antonelli2020rectifiable} for the missing details. Let $N\in \N$ be the unique natural number for which $5^{N-2}\leq A<5^{N-1}$, where $A$ is as in \cref{lemma:tec.cono}. Notice that, since $2^{-h}\leq \Theta^{h,*}(\mathcal{S}^h\llcorner\Gamma,x)\leq 1$ for $\mathcal{S}^h\llcorner\Gamma$-almost every $x\in\mathbb G$ (cf. \cite[2.10.19(1) and 2.10.19(5)]{Federer1996GeometricTheory}), hence, for $\mathcal{S}^h\llcorner \Gamma$-almost every $x\in\mathbb G$ there exists an infinitesimal sequence $\{\ell_i(x)\}_{i\in\N}$ such that
\begin{equation}
    \frac{1}{2^{h+1}}\leq\frac{\mathcal{S}^h(\Gamma\cap B(x,\ell_i(x)))}{\ell_i(x)^h}\leq 2.
    \label{eq:bdbddens}
\end{equation}
Thus, for any $k\in\N$ and $0<\delta<1/2$ we define the following sets
\begin{equation}
    \begin{split}
        A(k):=&\{x\in \Gamma:\rho (x)>1/k\},\\
        \mathscr{D}(k):=&\bigg\{x\in A(k):\lim_{r\to 0}\frac{\mathcal{S}^{h}(B(x,r)\cap A(k))}{\mathcal{S}^{h}(B(x,r)\cap \Gamma)}=1\bigg\},\\
        \mathscr{F}_{\delta}(k):=&\bigg\{B(x,r):x\in \mathscr{D}(k) \text{, }r\leq \frac{\min\{k^{-1},\delta\}}{1000 A}\text{ and }\frac{1}{2^{h+1}}\leq\frac{\mathcal{S}^h\llcorner\Gamma(B(x,5^{N+1}r))}{(5^{N+1}r)^h}\leq 2\bigg\},
    \end{split}
\end{equation}
where $\rho(x)$ is the number defined in \cref{prop:palleseparate}.
First of all notice that, thanks to \eqref{eq:bdbddens}, $\mathscr{F}_{\delta}(k)$ is a fine covering of $\mathcal{S}^h\llcorner\Gamma$-almost all $\mathscr{D}(k)$. Furthermore, for any $k$ the sets $A(k)$ are Borel since thanks to \cref{prop:palleseparate}, the function $\rho$ is upper semicontinuous and, since by assumption $\mathcal{S}^h\llcorner \Gamma$ is locally asymptotically doubling, we also know that $\mathcal{S}^h\llcorner \Gamma(A(k)\setminus\mathscr{D}(k))=0 $. Finally, from \cref{prop:palleseparate} we infer that $\mathcal{S}^h(\Gamma\setminus\cup_{k=1}^{+\infty} A(k))=0$. 
Let us apply \cite[Lemma 4.5]{antonelli2020rectifiable} to $N$ and $\mathscr{F}_\delta(k)$ and we obtain the disjoint subfamily $\mathscr{G}_\delta(k)$ of $\mathscr{F}_\delta(k)$
such that
\begin{itemize}
    \item[($\alpha$)] for any $B,B^\prime\in\mathscr{G}_\delta(k)$ we have that $5^NB\cap 5^NB^\prime=\emptyset$,
    \item[($\beta$)] $\bigcup_{B\in\mathscr{F}_\delta(k)} B\subseteq \bigcup_{B\in\mathscr{G}_\delta(k)} 5^{N+1}B$.
\end{itemize}
Throughout the rest of the proof we fix a $w\in \mathscr{D}(k)$ such that there exists a sequence $\{\ell_i(w)\}_{i\in\N}$ satisfying \eqref{eq:bdbddens}, $\ell_i(w)\leq \min\{k^{-1},\delta\}/8$, and
\begin{equation}\label{eqn:RequirementsOnR}
\frac{\mathcal{S}^{h}\llcorner \mathscr{D}(k)(B(w,\ell_i(w)))}{\mathcal{S}^{h}\llcorner \Gamma(B(w,\ell_i(w)))}\geq \frac{1}{2}\qquad\text{for any }i\in\N,
\end{equation}
where the inequality follows from the fact that $\mathcal{S}^h\llcorner \Gamma$-almost every point of $\mathscr{D}(k)$ has density one with respect to the locally asymptotically doubling measure $\mathcal{S}^h\llcorner \Gamma$. Notice that, according to the previous discussion, the previous choice on $w$ is made in a set of full $\mathcal{S}^h\llcorner\Gamma$-measure, so that if we prove the estimate \eqref{eq:bigproj} with such a $w$ we are done.
For the ease of notation we continue the proof fixing the radius $\ell_i(w)=R$. We stress that the forthcoming estimates are verified also for any $\ell_i(w)$. 
As in \cite[Proposition 4.6]{antonelli2020rectifiable}, one can prove that for any couple of closed balls $B(x,r),B(y,s)\in \mathscr{G}_\delta(k)$ such that $B(w,R)$ intersects both $B(x,5^{N+1}r)$ and $B(y,5^{N+1}s)$, we have
\begin{equation}
    P_\mathbb{V}(B(x,r))\cap P_{\mathbb{V}}(B(y,s))=\emptyset.
    \label{eq:palle}
\end{equation}

In order to proceed with the conclusion of the proof, let us define
\begin{equation}
    \begin{split}
    \mathscr{F}_\delta(w,R):=&\{B\in\mathscr{F}_\delta(k): 5^{N+1}B\cap B(w, R)\cap\mathscr{D}(k)\neq \emptyset\},\\
    \mathscr{G}_\delta(w,R):=&\{B\in\mathscr{G}_\delta(k): 5^{N+1}B\cap B(w, R)\cap\mathscr{D}(k)\neq \emptyset\},
        \nonumber
    \end{split}
\end{equation}
Thanks to our choice of $R$, see \eqref{eqn:RequirementsOnR}, and the definition of $\mathscr{G}_{\delta}(w,R)$ we have
$$\frac{R^h}{2^{h+1}}\leq \mathcal{S}^h\llcorner \Gamma(B(w,R))\leq 2\mathcal{S}^h\llcorner \mathscr{D}(k)(B(w,R))\leq2\mathcal{S}^h\llcorner \mathscr{D}(k)\bigg(\bigcup_{B\in \mathscr{G}_{\delta}(w,R)} 5^{N+1}B\bigg).$$
Let $\mathscr{G}_{\delta}(w,R):=\{B(x_i,r_i)\}_{i\in\N}$ and recall that $x_i\in \mathscr{D}(k)$. This implies, thanks to \cref{cor:2.2.19}, that
\begin{equation}
\begin{split}
    \mathcal{S}^h\llcorner \mathscr{D}(k)\bigg(\bigcup_{B\in \mathscr{G}_{\delta}(w,R)} 5^{N+1}B\bigg)\leq 2\cdot 5^{(N+1)h}\sum_{i\in\N}r_i^h
    &=  2\cdot 5^{(N+1)h}\oldC{BallC}(\mathbb{V},\mathbb L)^{-1}\sum_{i\in\N} \mathcal{S}^h(P_\mathbb{V}(B(x_i,r_i)))\\
    = 2\cdot 5^{(N+1)h}\oldC{BallC}(\mathbb{V},\mathbb L)^{-1}\mathcal{S}^h\bigg(P_\mathbb{V}\bigg(\bigcup_{i\in\N} B(x_i,r_i)\bigg)\bigg)&\leq 2\cdot 5^{(N+1)h}\oldC{BallC}(\mathbb{V},\mathbb L)^{-1}\mathcal{S}^h\bigg(P_\mathbb{V}\bigg(\bigcup_{B\in \mathscr{F}_\delta(w,R)} B\bigg)\bigg),
    \nonumber
    \end{split}
\end{equation}
where the first inequality comes from the subadditivity of the measure and the upper estimate that we have in the definition of $\mathscr{F}_{\delta}(k)$; while the first identity of the second line above comes from \eqref{eq:palle}. Summing up, for any $\delta>0$ we have
$$
\frac{\oldC{BallC}(\mathbb{V},\mathbb L) R^h}{ 5^{(N+1)h}2^{h+3}}\leq \mathcal{S}^h\bigg(P_\mathbb{V}\bigg(\bigcup_{B\in \mathscr{F}_\delta(w,R)} B\bigg)\bigg).
$$
Arguing as in the end of the proof of \cite[Proposition 4.6]{antonelli2020rectifiable}, we get the Hausdorff convergence
$$
P_{\mathbb{V}}\bigg(\bigcup_{B\in \mathscr{F}_\delta(w,R)} B\bigg)\underset{H,\delta\to 0}{\longrightarrow} P_\mathbb{V}\bigg(\overline{\mathscr{D}(k)}\cap B(w,R)\bigg).
$$
Thanks to the upper semicontinuity of the Lebesgue measure with respect to the Hausdorff convergence we eventually infer that
$$\frac{\oldC{BallC}(\mathbb{V},\mathbb L) R^h}{5^{(N+1)h}2^{h+3}}\leq\limsup_{\delta\to0}\mathcal{S}^h\bigg(P_\mathbb{V}\bigg(\bigcup_{B\in \mathscr{F}_\delta(w,R)} B\bigg)\bigg)\leq \mathcal{S}^h(P_\mathbb{V}(\overline{\mathscr{D}(k)}\cap B(w,R)))\leq \mathcal{S}^h(P_\mathbb{V}(\Gamma\cap B(w,R))),$$
where the last inequality above comes from the compactness of $\Gamma$ and the fact that $\mathscr{D}(k)\subseteq \Gamma$.
\end{proof}

\begin{proposizione}\label{prop:mutuallyabs}
Let us fix $\alpha\leq\oldep{ep:Cool}(\mathbb V,\mathbb L)$ and suppose $\Gamma$ is a compact $C_\mathbb{V}(\alpha)$-set of finite $\mathcal{S}^h$-measure such that
$$
\limsup_{r\to 0}\frac{\mathcal{S}^h\llcorner \Gamma(B(x,2r))}{\mathcal{S}^h\llcorner \Gamma(B(x,r))}<\infty,
$$
for $\mathcal{S}^h$-almost every $x\in \Gamma$. Let us set $\varphi:P_{\mathbb V}(\Gamma)\to \mathbb{L}$ the map whose graph is $\Gamma$, see \cref{prop:ConeAndGraph}, and set $\Phi:P_{\mathbb V}(\Gamma)\to\mathbb G$ to be the graph map of $\varphi$. Let us define $\Phi_*\mathcal S^h\llcorner\mathbb V$ to be the measure on $\Gamma$ such that for every measurable $A\subseteq \Gamma$ we have $\Phi_*\mathcal S^h\llcorner\mathbb V(A):=\mathcal{S}^h\llcorner\mathbb V(\Phi^{-1}(A))=\mathcal{S}^h\llcorner \mathbb V(P_{\mathbb V}(A))$. Then $\Phi_*\mathcal{S}^h\llcorner \mathbb{V}$ is mutually absolutely continuous with respect to $\mathcal{S}^h\llcorner \Gamma$.
\end{proposizione}

\begin{proof}

The fact that $\Phi_*\mathcal{S}^h\llcorner \mathbb{V}$ is absolutely continuous with respect to $\mathcal{S}^h\llcorner \Gamma$ is an immediate consequence of the second part of \cref{cor:2.2.19}. 
Viceversa, suppose by contradiction that there exists a compact subset $C$ of $\Gamma$ of positive $\mathcal{S}^h$-measure such that
\begin{equation}
    0=\Phi_*\mathcal{S}^h\llcorner \mathbb{V}(C)=\mathcal{S}^h(P_\mathbb{V}(C)).
    \label{eq:contradiction}
\end{equation}
Since $\mathcal{S}^h\llcorner C$ is locally asymptotically doubling by \cref{prop:Lebesuge} and $C$ has positive and finite $\mathcal{S}^h$-measure, we infer thanks to \cref{prop:proj} that the set $C$ must have a projection of positive $\mathcal{S}^h$-measure. This however comes in contradiction with \eqref{eq:contradiction}.
\end{proof}

In order to prove the forthcoming \cref{prop:DensitaInfPositiva} we need the following result, which is precisely \cite[Proposition 4.10]{antonelli2020rectifiable}.

\begin{proposizione}
\label{prop:vitali2PD}
Suppose $d$ is a homogeneous left-invariant distance on $\mathbb{G}$, let $\mathbb V,\mathbb L$ be complementary subgroups of $\mathbb G$ such that $\mathbb V\in\G_c(h)$, and let us fix $\alpha\leq\oldep{ep:Cool}(\mathbb V,\mathbb L)$. Suppose that $\Gamma$ is a compact $C_{\mathbb V,d}(\alpha)$-set of finite $\mathcal{S}^h$-measure. As in \cref{prop:mutuallyabs}, let us denote with $\Phi:P_{\mathbb V}(\Gamma)\to\mathbb G$ the graph map of $\varphi:P_{\mathbb V}(\Gamma)\to \mathbb L$ whose intrinsic graph is $\Gamma$.
Then, for $\mathcal{S}^h_d$-almost every $w\in P_\mathbb{V}(\Gamma)$ we have
\begin{equation}
    \lim_{r\to 0}\frac{\mathcal{S}^h_d\big(P_\mathbb{V}\big(B_{d}(\Phi(w),r)\cap \Phi(w)C_{\mathbb V,d}(\alpha)\big)\cap P_\mathbb{V}(\Gamma)\big)}{\mathcal{S}_{d}^h\big(P_\mathbb{V}\big(B_d(\Phi(w),r)\cap \Phi(w)C_{\mathbb V,d}(\alpha)\big)\big)}=1.
    \label{eq:limi}
\end{equation}
\end{proposizione}

\begin{proposizione}\label{prop:DensitaInfPositiva}
Let $\alpha\leq\oldep{ep:Cool}(\mathbb V,\mathbb L)$ and suppose $\Gamma$ is a $C_\mathbb{V}(\alpha)$-set such that $\mathcal{S}^h\llcorner \Gamma$ is locally asymptotically doubling. Then, $\Theta^h_*(\mathcal{S}^h\llcorner\Gamma,x)>0$ for $\mathcal{S}^h$-almost every $x\in\Gamma$.
\end{proposizione}

\begin{proof}
Assume by contradiction that there exists a compact set $C\subseteq \Gamma$ of positive $\mathcal{S}^h$-measure such that $\Theta^h_*(\mathcal{S}^h\llcorner\Gamma,x)=0$ for every $x\in C$. Since by \cref{prop:mutuallyabs} the measures $\mathcal{S}^h\llcorner \Gamma$ and $\Phi_*\mathcal{S}^h\llcorner \mathbb{V}$ are mutually absolutely continuous, the set $P_\mathbb{V}(C)$ must have positive $\mathcal{S}^h$-measure as well. In particular we have thanks to \cref{prop:vitali2PD}, \cref{lemma:lowerbd}, \cref{cor:2.2.19}, and \cref{prop:InvarianceOfProj} that for $\mathcal{S}^h$-almost every $x\in C$ we have
\begin{equation}
    \begin{split}
        &\mathcal{S}^h\big(P_\mathbb{V}\big(B(0,1)\cap C_\mathbb{V}(\alpha)\big)\big)\\=&\liminf_{r\to 0}\frac{\mathcal{S}^h\big(P_\mathbb{V}\big(B(x,\mathfrak{C}(\alpha)r)\cap xC_\mathbb{V}(\alpha)\big)\cap P_\mathbb{V}(\Gamma)\big)}{\mathcal{S}^h\big(P_\mathbb{V}\big(B(\Phi(w),\mathfrak{C}(\alpha)r)\cap \Phi(w)C_\mathbb{V}(\alpha)\big)\big)}\frac{\mathcal{S}^h\big(P_\mathbb{V}\big(B(\Phi(w),\mathfrak{C}(\alpha)r)\cap \Phi(w)C_\mathbb{V}(\alpha)\big)\big)}{(\mathfrak{C}(\alpha)r)^h}\\
        =&\liminf_{r\to 0}\frac{\mathcal{S}^h\big(P_\mathbb{V}\big(B(x,\mathfrak{C}(\alpha)r)\cap xC_\mathbb{V}(\alpha)\big)\cap P_\mathbb{V}(\Gamma)\big)}{(\mathfrak{C}(\alpha)r)^h}\\
        \leq& \liminf_{r\to 0}\frac{\mathcal{S}^h\llcorner \mathbb{V}(P_\mathbb{V}(B(x,r)\cap \Gamma))}{(\mathfrak{C}(\alpha)r)^h}\leq 2\oldC{ProjC}(\mathbb{V},\mathbb L)\liminf_{r\to 0}\frac{\mathcal{S}^h\llcorner \Gamma(B(x,r))}{(\mathfrak{C}(\alpha)r)^h}=0,
        \nonumber
    \end{split}
\end{equation}
where $\mathfrak{C}(\alpha)$ is the constant introduced in \cref{lemma:lowerbd}. The above computation is in contradiction with the fact that $\mathcal{S}^h\big(P_\mathbb{V}\big(B(0,1)\cap C_\mathbb{V}(\alpha)\big)\big)$ is positive thus concluding the proof of the proposition.
\end{proof}

We are now in a position to state the main result of this subsection.
\begin{teorema}\label{thm:ExistenceOfDensityPlus}
    Let $\Gamma\subseteq \mathbb G$ be compact such that $\mathcal{S}^h(\Gamma)<+\infty$. Assume that for $\mathcal{S}^h\llcorner\Gamma$-almost every $x\in\mathbb G$ we have $\mathrm{Tan}(\mathcal{S}^h\llcorner\Gamma,x)=\{\lambda\mathcal{S}^h\llcorner\mathbb V(x):\lambda>0,\}$, where $\mathbb V(x)\in\G_c(h)$. Then,  $\mathcal{S}^h\llcorner\Gamma$ is $\mathscr{P}_h^c$-rectifiable.
\end{teorema}
\begin{proof}
We have that $\mathcal{S}^h\llcorner\Gamma$ is locally asymptotically doubling, see \cref{prop:PhiAsDoubling}. Moreover, from \cref{prop:TangentCompemented}, there exist $\{\mathbb V_i\}_{i\in\mathbb N}\subseteq \G_c(h)$, and $\{\mathbb L_i\}_{i\in\mathbb N}$, such that $\mathbb L_i$ and $\mathbb V_i$ are homogeneous complementary subgroups, with the property that for every $\alpha>0$ there exists a family of compact sets $\{\Gamma_i\}$ such that $\Gamma_i$ is a $C_{\mathbb V_i}(\min\{\alpha,\oldep{ep:Cool}(\mathbb V_i,\mathbb L_i)\})$-set, and
\begin{equation}\label{eqn:ShAlmostEverywhere}
\mathcal{S}^h(\Gamma \setminus \cup_{i\in\mathbb N}\Gamma_i)=0.
\end{equation}
Since $\mathcal{S}^h\llcorner\Gamma$ is locally asymptotically doubling, then $\mathcal{S}^h\llcorner\Gamma_i$ is locally asymptotically doubling for every $i\in\mathbb N$, see \cref{prop:Lebesuge}. Hence, we can apply \cref{prop:DensitaInfPositiva} to conclude that  $\Theta^h_*(\mathcal{S}^h\llcorner\Gamma_i,x)>0$  for every $i\in\mathbb N$ and for $\mathcal{S}^h$-almost every $x\in\Gamma_i$. In addition, from the previous inequality and \cite[2.10.19(5)]{Federer1996GeometricTheory} for every $i\in\mathbb N$, we get that
\begin{equation}\label{eqn:DensityBoundYEAH}
0<\Theta^h_*(\mathcal{S}^h\llcorner\Gamma_i,x)\leq \Theta^{h,*}(\mathcal{S}^h\llcorner\Gamma_i,x)<+\infty, \quad \text{for $\mathcal{S}^h$-almost every $x\in\Gamma_i$.}
\end{equation}
Moreover, since for $\mathcal{S}^h$-almost every $x\in\Gamma$ we have $\Tan(\mathcal{S}^h\llcorner\Gamma,x)=\{\lambda\mathcal{S}^h\llcorner\mathbb V(x):\lambda>0\}$ with $\mathbb V(x)\in \G_c(h)$, we deduce that, for every $i\in\mathbb N$, the locality of tangents in \cref{prop:Lebesuge} ensures that for $\mathcal{S}^h$-almost every $x\in\Gamma_i$ we have $\Tan(\mathcal{S}^h\llcorner\Gamma_i,x)=\{\lambda\mathcal{S}^h\llcorner\mathbb V(x):\lambda>0\}$. From the previous equality, we conclude that for every $i\in\mathbb N$ we have $\Tan_h(\phi,x)\subseteq\{\lambda\mathcal{S}^h\llcorner\mathbb V(x):\lambda>0\}$. Hence, from the latter conclusion and \eqref{eqn:DensityBoundYEAH} we get that $\mathcal{S}^h\llcorner\Gamma_i$ is a $\mathscr{P}_h^c$-rectifiable measure for every $i\in\mathbb N$. Finally, from \eqref{eqn:ShAlmostEverywhere} and \cref{prop:Lebesuge} we conclude that $\mathcal{S}^h\llcorner\Gamma$ is a $\mathscr{P}^c_h$-rectifiable measure. 
\end{proof}

\subsection{From approximate tangent planes to $\mathscr{P}$-rectifiability}

In this section we aim at proving that whenever an approximate (complemented) $h$-dimensional tangent plane to a set $\Gamma$ exists almost everywhere in the sense of \cref{prop:rett.1}, then the measure $\mathcal{S}^h\llcorner\Gamma$ is $\mathscr{P}^c_h$-rectifiable. First, we need a crucial estimate on projections that will be useful also later on. Since none of the main results of this subsection depends on the choice of the metric, in the following $d$ will be an arbitrary homogeneous left-invariant metric and $\lVert\rVert_{d}$ its associated homogeneous norm.

\begin{proposizione}\label{prop:EstimateOnProjection}
Suppose $d$ is a homogeneous left-invariant distance on $\mathbb{G}$ and let $\mathbb{V},\mathbb{W}\in \G_c(h)$ be complemented by the same subgroup $\mathbb{L}$. Then, there exists an increasing function $\Delta:(0,\oldC{ProjC}(d,\mathbb W,\mathbb L)]\to(0,+\infty)$, depending only on $\mathbb{V}$, $\mathbb{W}$, $\mathbb{L}$ and $d$, such that $\lim_{\beta\to 0}\Delta(\beta)=0$ and satisfying the following condition.

For any $\alpha\leq\oldep{ep:Cool}(\mathbb V,\mathbb L)$ and any  $C_{\mathbb{V},d}(\alpha)$-set $\Gamma$ of finite $\mathcal{S}^h$-measure if there are an $x\in \Gamma$, a $\beta\leq\oldC{ProjC}(d,\mathbb W,\mathbb L)$ and a $\rho>0$ such that 
\begin{equation}\label{eqn:HypothesisOnW}
    \Gamma\cap B_d(x,r) \subseteq xC_{\mathbb{W},d}(\beta),\qquad\text{for all $0<r<\rho$.}
\end{equation}
then
\begin{equation}
    \Big\lvert\frac{\mathcal{S}_d^h(P_{\mathbb V}(B_d(x,r)\cap xC_{\mathbb W,d}(\beta))\cap P_{\mathbb V}(\Gamma))}{r^h}-\frac{\mathcal{S}^h(P_{\mathbb V}(B_d(x,r)\cap xC_{\mathbb W,d}(\beta)\cap \Gamma))}{r^h}\Big\rvert\leq \Delta(\beta),
    \nonumber
\end{equation}
for any $0<r<(1+\alpha(\oldC{ProjC}(d,\mathbb V,\mathbb L)-\alpha)^{-1})^{-1}\oldC{ProjC}(d,\mathbb V,\mathbb L)\rho=:\varrho(\rho,\alpha)$.
\end{proposizione}

\begin{proof}
In order to simplify notation throughout the proof of the proposition, we will drop everywhere the dependence on $d$.

Let us fix an $x\in \Gamma$, a $0<\beta\leq \oldC{ProjC}(\mathbb W,\mathbb L)$ and a $\rho$ where \eqref{eqn:HypothesisOnW} holds. We denote with $ P_\mathbb V,P_\mathbb L^{\mathbb V}$, respectively, the projections associated to the splitting $\mathbb G=\mathbb V\cdot\mathbb L$, and analogously for the splitting $\mathbb G=\mathbb W\cdot\mathbb L$. For the sake of notation, for any fixed $0<r<(1+\alpha(\oldC{ProjC}(\mathbb V,\mathbb L)-\alpha)^{-1})^{-1}\oldC{ProjC}(\mathbb V,\mathbb L)\rho$ we let
\begin{equation}
    \begin{split}
        A_r:=P_{\mathbb V}(B(x,r)\cap xC_{\mathbb W}(\beta))\cap P_{\mathbb V}(\Gamma)\qquad\text{and}\qquad
        B_r:=P_{\mathbb V}(B(x,r)\cap xC_{\mathbb W}(\beta)\cap \Gamma).
        \nonumber
    \end{split}
\end{equation}
Since the inclusion $B_r\subseteq A_r$ is always verified, we want to estimate the measure of those $w$ contained in $A_r\setminus B_r$. If $y\in A_r$, there are $w\in \Gamma$ such that $P_{\mathbb{V}}(w)=y$, and an $\ell\in\mathbb{L}$ such that $y\ell\in B(x,r)\cap xC_\mathbb{W}(\beta)$. Let us notice that \cref{cor:2.2.19} implies that $\|P_{\mathbb V}(x^{-1}y)\|=\|P_{\mathbb V}(x^{-1}y\ell)\|\leq \oldC{ProjC}(\mathbb V,\mathbb L)^{-1}r$. Moreover, since $\Gamma$ is a $C_{\mathbb V}(\alpha)$-set, we even get that, by exploiting \cref{oss:ConiEquivalenti},
$$\|P_{\mathbb L}^{\mathbb V}(x^{-1}w)\|\leq \alpha(\oldC{ProjC}(\mathbb V,\mathbb L)-\alpha)^{-1}\|P_{\mathbb V}(x^{-1}w)\|=\alpha(\oldC{ProjC}(\mathbb V,\mathbb L)-\alpha)^{-1}\|P_{\mathbb V}(x^{-1}y)\| \leq \alpha(\oldC{ProjC}(\mathbb V,\mathbb L)-\alpha)^{-1}\oldC{ProjC}(\mathbb V,\mathbb L)^{-1}r.$$
This implies in particular that
$$\|x^{-1}w\|\leq \|P_{\mathbb V}(x^{-1}w)\|+\|P_{\mathbb L}^{\mathbb V}(x^{-1}w)\|=\|P_{\mathbb V}(x^{-1}y)\|+\|P_{\mathbb L}^{\mathbb V}(x^{-1}w)\|\leq (1+\alpha(\oldC{ProjC}(\mathbb V,\mathbb L)-\alpha)^{-1})\oldC{ProjC}^{-1}(\mathbb V,\mathbb L)r.$$
Hence, from the choice of $r$, we infer that $(1+\alpha(\oldC{ProjC}(\mathbb V,\mathbb L)-\alpha)^{-1})\oldC{ProjC}(\mathbb V,\mathbb L)^{-1}r<\rho$ and thus  we can use the hypothesis in \eqref{eqn:HypothesisOnW} applied to $w$ to obtain that $x^{-1}w\in C_\mathbb W(\beta)$.
Thus, by also exploiting \cref{oss:ConiEquivalenti} and the fact that $x^{-1}y\ell\in C_\mathbb W(\beta)$ we get that
 \begin{align}
      \lVert P_\mathbb{L}^\mathbb{W}(x^{-1}y)\ell\rVert\leq \beta(\oldC{ProjC}(\mathbb W,\mathbb L)-\beta)^{-1}\lVert P_\mathbb{W}(x^{-1}y)\rVert\quad&\text{and}\quad\lVert P_\mathbb{L}^\mathbb{W}(x^{-1}y)P_\mathbb{L}^\mathbb{V}(w)\rVert\leq \beta(\oldC{ProjC}(\mathbb W,\mathbb L)-\beta)^{-1}\lVert P_\mathbb{W}(x^{-1}y)\rVert,
       \label{numberio2}
 \end{align}
 where the last inequality comes from the fact that $P_\mathbb{L}^\mathbb{W}(x^{-1}w)=P_\mathbb{L}^\mathbb{W}(x^{-1}yP_\mathbb{L}^\mathbb{V}(w))=P_\mathbb{L}^\mathbb{W}(x^{-1}y)P_\mathbb{L}^\mathbb{V}(w)$.
 Thanks to \eqref{numberio2} we deduce that
 \begin{equation}
 \begin{split}
     \lVert \ell^{-1}P_\mathbb{L}^{\mathbb V}(w)\rVert&\leq \lVert P_\mathbb{L}^\mathbb{W}(x^{-1}y)\ell\rVert+\lVert P_\mathbb{L}^\mathbb{W}(x^{-1}y)P_\mathbb{L}^{\mathbb V}(w)\rVert\leq 2\beta(\oldC{ProjC}(\mathbb W,\mathbb L)-\beta)^{-1}\lVert P_\mathbb{W}(x^{-1}y)\rVert \\
     &= 2\beta(\oldC{ProjC}(\mathbb W,\mathbb L)-\beta)^{-1}\lVert P_\mathbb{W}(x^{-1}y\ell)\rVert\\
     &\leq 2\beta(\oldC{ProjC}(\mathbb W,\mathbb L)-\beta)^{-1} \oldC{ProjC}(\mathbb{W},\mathbb{L})^{-1}\lVert x^{-1}y\ell\rVert \\
     &\leq 2\beta(\oldC{ProjC}(\mathbb W,\mathbb L)-\beta)^{-1} \oldC{ProjC}(\mathbb{W},\mathbb{L})^{-1}r.
     \end{split}
     \end{equation}
      This in particular implies that
     $$
     \lVert x^{-1}w\rVert=\lVert x^{-1}yP_\mathbb{L}^{\mathbb V}(w)\rVert\leq r+\lVert \ell^{-1}P_\mathbb{L}^{\mathbb V}(w)\rVert\leq(1+2\beta(\oldC{ProjC}(\mathbb W,\mathbb L)-\beta)^{-1} \oldC{ProjC}(\mathbb{W},\mathbb{L})^{-1})r=:f_2(\beta)r.
     $$
The above chain of inequalities, together with the hypothesis \eqref{eqn:HypothesisOnW}, allows us to conclude that 
   $$A_r\subseteq P_\mathbb{V}(B(x,f_2(\beta)r)\cap xC_\mathbb{W}(\beta)\cap \Gamma).$$
Finally this allows us to infer
   \begin{equation}
       \begin{split}
           \mathcal{S}^h(A_r)-\mathcal{S}^h(B_r)
           \leq\mathcal{S}^h(P_\mathbb{V}(B(x,f_2(\beta)r)\cap xC_\mathbb{W}(\beta)\cap \Gamma)\setminus P_\mathbb{V}(B(x,r)\cap xC_\mathbb{W}(\beta)\cap \Gamma))\\
           =\mathcal{S}^h(P_\mathbb{V}(B(x,f_2(\beta)r)\setminus B(x,r)\cap xC_\mathbb{W}(\beta)\cap \Gamma)),
       \end{split}
   \end{equation}
   where the last identity comes from the injectivity of $P_\mathbb{V}$ when restricted to $\Gamma$, since $\alpha\leq\oldep{ep:Cool}(\mathbb V,\mathbb L)$, see \cref{prop:ConeAndGraph}. Finally, \cref{prop:InvarianceOfProj} implies
   \begin{equation}
       \begin{split}
            \mathcal{S}^h(A_r)-\mathcal{S}^h(B_r)
           &\leq\mathcal{S}^h(P_\mathbb{V}(B(x,f_2(\beta)r)\setminus B(x,r)\cap xC_\mathbb{W}(\beta)))\\
           &=\mathcal{S}^h(P_\mathbb{V}(B(0,f_2(\beta))\setminus B(0,1)\cap C_\mathbb{W}(\beta)))r^h=:\Delta(\beta)r^h.
       \end{split}
   \end{equation}
   The function $\Delta$ is easily seen to be increasing and thanks to the continuity from above of the measure, the fact that
   $\lim_{\beta\to 0}\Delta(\beta)=0$ immediately follows too since $f_2(\beta)\to 1$ as $\beta\to 0$.
\end{proof}


\begin{proposizione}\label{prop:rett.1}
Suppose $d$ is a homogeneous left-invariant distance on $\mathbb{G}$ and that $\Gamma$ is a Borel set of $\mathcal{C}^h_d$-finite measure such that at $\mathcal{C}^h_d$-almost every $x\in \Gamma$ there exists $\mathbb V(x)\in\G_c(h)$ for which for any $0<\beta<1$ there exists a $\rho(x,\beta)>0$ such that
\begin{equation}
\Gamma\cap B_d(x,\rho(x,\beta))\subseteq xC_{\mathbb{V}(x),d}(\beta). 
\label{eq:num1}
\end{equation}
Then, the measure $\mathcal{C}^h_d\llcorner \Gamma$ is $\mathscr{P}_h^c$-rectifiable. In addition we have that $\Theta^h_d(\mathcal{C}^h_d\llcorner \Gamma,x)=1$  and $\Tan_h(\mathcal{C}^h_d\llcorner \Gamma,x)=\{\mathcal{C}^h_d\llcorner \mathbb{V}(x)\}$ for $\mathcal{C}^h_d\llcorner \Gamma$-almost every $x$.
\end{proposizione}

\begin{proof} 
In order to simplify notation throughout the proof of the proposition, we will drop everywhere the dependence on $d$.
First of all we define the family of sets
$$
\mathscr{F}:=\{\Gamma\subseteq \mathbb{G}:\Gamma\text{ is Borel, }\mathcal{S}^h\llcorner \Gamma\text{ is $\mathscr{P}_h^c$-rectifiable and $\Theta^h(\mathcal{C}^h\llcorner \Gamma,x)=1$ for $\mathcal{S}^h\llcorner \Gamma$-almost every $x$}\}.
$$
Thanks to \cite[Proposition 1.22]{MarstrandMattila20} we can write $\Gamma$ as $\Gamma=\Gamma^r\cup \Gamma^u$ where $\Gamma^r$ is a Borel set for which there are countable many $\Sigma_k\in\mathscr{F}$ such that $\Gamma^r\subseteq \cup_k\Sigma_k$ and $\Gamma^u$ is a Borel set such that $\mathcal{S}^h(\Gamma\cap \Sigma)=0$ for any $\Sigma\in\mathscr{F}$.

Let us prove that $\Gamma^r\in\mathscr{F}$. For any $k\in\N$ we define $\tilde \Sigma_k:=\Sigma_k\setminus \cup_{1\leq i\leq k-1}\Sigma_i$. Thanks to \cref{prop:Lebesuge} the measures $\mathcal{S}^h\llcorner \tilde\Sigma_k$ is still $\mathscr{P}_h^c$-rectifiable, the $\tilde\Sigma_k$ are pairwise disjoint and their union still contains $\Gamma^r$. Again by \cref{prop:Lebesuge} we infer that for any $k\in \N$ the measure $\mathcal{S}^h\llcorner (\tilde\Sigma_k\cap \Gamma^r)$ is $\mathscr{P}_h^c$-rectifiable and $\Theta^h(\mathcal{C}^h\llcorner  (\tilde\Sigma_k\cap \Gamma^r),x)=1$ for $\mathcal{S}^h\llcorner  (\tilde\Sigma_k\cap \Gamma^r)$-almost every $x$ and this finally implies that
$$\Theta^h(\mathcal{C}^h\llcorner  \Gamma^r,x)=1,$$
for $\mathcal{S}^h\llcorner  \Gamma^r$-almost every $x$. Applying \cref{prop:Lebesuge} to the measure $\mathcal{S}^h\llcorner  \Gamma^r$ and to the Borel set $\tilde \Sigma_k$ we infer that $\Tan_h(\mathcal{S}^h\llcorner  \Gamma^r,x)$ is unique and flat $\mathcal{S}^h\llcorner  \Gamma^r$-almost everywhere on $\tilde \Sigma_k$. Since the $\tilde\Sigma_k$ countably cover $\Gamma^r$ this concludes the proof that $\mathcal{S}^h\llcorner \Gamma^r$ is $\mathscr{P}_h^c$-rectifiable.
 \medskip
 
The above argument shows that we can assume by contradiction that $\Gamma$ is compact set of positive and finite $\mathcal{S}^h$-measure and that 
\begin{equation}
    \mathcal{S}^h(\Gamma\cap \Sigma)=0\text{ for any }\Sigma\in\mathscr{F}.
    \label{eq:gammaunrect}
\end{equation}
For any $\eta>0$ we let
$$
\G_c^\eta(h):=\Big\{\mathbb{V}\in\G_c(h):\inf_{\mathbb{W}\in\G(h)\setminus\G_c(h)}d_\mathbb{G}(\mathbb{V},\mathbb{W})\geq \eta\Big\}\subseteq \G_c(h).
$$
Thanks to \cref{prop:CompGrassmannian} it follows that $\G_c^\eta(h)$ is a closed, thus compact, subset of $\G(h)$.
Thanks to \cref{prop:approxtangentismeasurable}, for any $\eta>0$ the set $\Gamma\eta:=\{x\in\Gamma:\eqref{eq:num1}\text{ holds at $x$ and } \mathbb{V}(x)\in\G_c^\eta(h)\}$ is $\mathcal{S}^h$-measurable. In addition to this, since $\mathbb{V}(x)$ belongs $\mathcal{S}^h\llcorner \Gamma$-almost everywhere to $\G_c(h)$, that is an open set in $\G(h)$, see \cite[Proposition 2.17]{antonelli2020rectifiable}, we have
$$\mathcal{S}^h(\Gamma\setminus \bigcup_{\eta\in\Q^+\setminus\{0\}}\Gamma^\eta)=0.$$
In particular there exists an $\eta_0>0$ such that $\mathcal{S}^h(\Gamma^{\eta_0})>0$.
In the following we let $E$ be a compact subset of $\Gamma^{\eta_0}$ such that 
$$
\mathcal{S}^h(\Gamma^{\eta_0}\setminus E)<\mathcal{S}^h(\Gamma^{\eta_0})/2.
$$
Note further that thanks to \cref{prop:grasscompiffC3} we have that
$$m(\eta_0):=\min_{\mathbb{W}\in\G_c^{\eta_0}(h)}\mathfrak{e}(\mathbb{V})>0.$$
Let
$\mathscr{D}:=\{\mathbb V_{j}\}_{j\in\N}$
 be a countable dense subset of $\G_c^{\eta_0}(h)$ and
 \begin{equation}\nonumber
     \text{for all $j\in\N$ we choose a complement $\mathbb L_{j}$ of $\mathbb V_{j}$ s.t. $\oldep{ep:Cool}(\mathbb V_{j},\mathbb L_{j})>\mathfrak{e}(\mathbb{V}_j)/2\geq m(\eta_0)/2$}.
 \end{equation}
From now on we let $\varepsilon$ be a fixed positive number in $(0,m(\eta_0)/10)$ such that 
\begin{equation}\label{eqn:BoundStr}
    1-3m(\eta_0)^{-1}\varepsilon(1+3 m(\eta_0)^{-2}\varepsilon)/(m(\eta_0)-\varepsilon)>0,
\end{equation}
which we can do taking $\varepsilon$ small enough. The previous estimate will play a role later on. For any $p,q\in\N$ we define the set
\begin{equation}
    F(p,q):=\{x\in E: B(x,1/q)\cap \Gamma\subseteq xC_{\mathbb{V}_p}(\varepsilon/6)
\},
\end{equation}
and we claim that
\begin{equation}
    \mathcal{S}^h\big(E\setminus \bigcup_{p,q\in\N} F(p,q)\big)=0.
    \label{eq:num2}
\end{equation}
By density of the family $\mathscr{D}$ in $\G_c^{\eta_0}(h)$ and since by construction for any $x\in\Gamma^{\eta_0}$ we have $\mathbb V(x)\in \G_c^{\eta_0}(h)$, we deduce that there must exist a plane $\mathbb V_p\in\mathscr{D}$ such that $d_{\mathbb G}(\mathbb{V}_p,\mathbb{V}(x))<30^{-1}\varepsilon$.
This, jointly with \cite[Lemma 2.15]{antonelli2020rectifiable}, 
implies that
\begin{equation}\label{eq:split1}
\begin{split}
    C_{\mathbb{V}(x)}(30^{-1}\varepsilon) 
    \subseteq C_{\mathbb V_p}(6^{-1}\varepsilon).
    \end{split}
\end{equation}
Since by definition of $\Gamma^{\eta_0}$, \eqref{eq:num1} holds at every point $x\in E$, we can find a $\rho(x)>0$ such that for any $0<r<\rho(x)$ we have
\begin{equation}
     B(x,r)\cap \Gamma\subseteq  xC_{\mathbb{V}(x)}(30^{-1}\varepsilon).
     \label{eq:split2}
\end{equation}
In particular, putting together \eqref{eq:split1} and \eqref{eq:split2} we infer that for $\mathcal{S}^h\llcorner \Gamma$-almost every $x\in E$ there are a $p=p(x)>0$ and a $\rho(x)>0$ such that whenever $0<r<\rho(x)$ we have
$$
B(x,r)\cap \Gamma\subseteq xC_{\mathbb{V}_{p}}(6^{-1}\varepsilon),
$$
and this concludes the proof of \eqref{eq:num2}. Thanks to \cite[Proposition 3.3]{antonelli2020rectifiable} and \cref{prop:ConeAndGraph}, we get that there are countably many $\mathbb{V}_j\in \G_c^{\eta_0}(h)$ complemented by some $\mathbb{L}_j$, compact subsets $K_j$ of $\mathbb{V}_j$ and intrinsic Lipschitz functions $\varphi_j:K_j\subset\mathbb{V}_j\to\mathbb{L}_j$ such that
\begin{enumerate}
    \item for any $z\in K_j$ we have  $\Gamma_j=\{w\varphi_j(w):w\in K_j\}\subseteq z\varphi_j(z)C_\mathbb{V_j}(\varepsilon)$, and $\Gamma_j\subseteq E$ 
    \item $\mathcal{S}^h(E\setminus \cup_j\Gamma_j)=0$.
\end{enumerate}
Thanks to  \cite[Corollary 4.17]{EdgarCentered} we know that $\Theta^{*,h}(\mathcal{C}^h\llcorner E,x)\leq 1$ for $\mathcal{C}^h\llcorner E$-almost every $x$ and now we wish to prove that $\Theta^{h}_*(\mathcal{C}^h\llcorner E,x)\geq 1$ for $\mathcal{C}^h\llcorner E$-almost every $x$.

Fix a $j\in\N$, and an $x\in \Gamma_j$ such that the conclusion in \cref{prop:vitali2PD} holds. Notice that such a choice of $x$ can be made in a set of $\mathcal{C}^h\llcorner\Gamma_j$-full measure in $\Gamma_j$. Suppose that $r_k$ is an infinitesimal sequence such that 
$$
\Theta^h_*(\mathcal{C}^h\llcorner\Gamma_j,x)=\lim_{k\to \infty} r_k^{-h}\mathcal{C}^h\llcorner \Gamma_j(B(x,r_k)).
$$
Thanks to item 1. above and to \cref{prop:EstimateOnProjection} one infers that we have that for any $k\in \N$ we get
\begin{equation}
    \Big\lvert\frac{\mathcal{C}^h(P_{\mathbb V_j}(B(x,r_k)\cap xC_{\mathbb V_j}(\varepsilon))\cap P_{\mathbb V_j}(\Gamma_j))}{r_k^h}-\frac{\mathcal{C}^h(P_{\mathbb V_j}(B(x,r_k)\cap xC_{\mathbb V_j}(\varepsilon)\cap \Gamma_j))}{r_k^h}\Big\rvert\leq \Delta_j(\varepsilon),
    \label{es:num2}
\end{equation}
where $
\Delta_j$
was introduced in the statement of \cref{prop:EstimateOnProjection} and depends only on the split $\mathbb{V}_j\cdot\mathbb L_j=\mathbb{G}$. In addition to this, for any $j\in\N$ the definitions of $\oldep{ep:Cool}(\cdot,\cdot)$ and of $\mathbb{L}_j$ imply that
\begin{equation}
   \oldC{ProjC}(\mathbb{V}_j,\mathbb{L}_j)=2\oldep{ep:Cool}(\mathbb{V}_j,\mathbb{L}_j)>\mathfrak{e}(\mathbb{V}_j)\geq m(\eta_0), \label{estimatedelta}
\end{equation}
and in turn this means that $\Delta_j(\varepsilon)$ can be estimated with
\begin{equation}
    \begin{split}
        \Delta_j(\varepsilon)&=\mathcal{C}^h(P_{\mathbb{V}_j}(B(0,f_2(\varepsilon))\setminus B(0,1)\cap C_{\mathbb{V}_j}(\varepsilon)))\\
        &=\mathcal{C}^h(P_{\mathbb{V}_j}(B(0,1+2\varepsilon(\oldC{ProjC}(\mathbb V_j,\mathbb L_j)-\varepsilon)^{-1} \oldC{ProjC}(\mathbb{V}_j,\mathbb{L}_j)^{-1})\setminus B(0,1)\cap C_{\mathbb{V}_j}(\varepsilon)))\\
        &\leq \mathcal{C}^h(P_{\mathbb{V}_j}(B(0,1+3 m(\eta_0)^{-2}\varepsilon)\setminus B(0,1)\cap C_{\mathbb{V}_j}(\varepsilon))),
        \label{es:num3}
    \end{split}
\end{equation}
where the last inequality above comes from \eqref{estimatedelta} and the fact that $\varepsilon\in(0,m(\eta_0)/10)$.
From  \eqref{es:num2}, the invariance properties in \cref{prop:InvarianceOfProj}, the fact that $x\in\Gamma_j$ was chosen in such a way that \cref{prop:vitali2PD} holds, and the homogeneity of $\mathcal{C}^h\llcorner\mathbb V$, we infer that
\begin{equation}
    \begin{split}
       \mathcal{C}^h(P_{\mathbb V_j}(B(0,1)\cap C_{\mathbb V_j}(\varepsilon)))& \limsup_{k\to+\infty}\Big(\frac{\mathcal{C}^h(P_{\mathbb V_j}(B(x,r_k)\cap xC_{\mathbb V_j}(\varepsilon))\cap P_{\mathbb V_j}(\Gamma_j))}{\mathcal{C}^h(P_{\mathbb V_j}(B(x,r_k)\cap xC_{\mathbb V_j}(\varepsilon)))}-\frac{\mathcal{C}^h(P_{\mathbb V_j}(B(x,r_k)\cap xC_{\mathbb V_j}(\varepsilon)\cap \Gamma_j)}{\mathcal{C}^h(P_{\mathbb V_j}(B(x,r_k)\cap xC_{\mathbb V_j}(\varepsilon)))})\Big)\\
       &= \mathcal{C}^h(P_{\mathbb V_j}(B(0,1)\cap C_{\mathbb V_j}(\varepsilon)))\Big(1-\liminf_{k\to+\infty}\frac{\mathcal{C}^h(P_{\mathbb V_j}(B(x,r_k)\cap xC_{\mathbb V_j}(\varepsilon)\cap \Gamma_j))}{\mathcal{C}^h(P_{\mathbb V_j}(B(x,r_k)\cap xC_{\mathbb V_j}(\varepsilon)))}\Big)\leq \Delta_j(\varepsilon).
       \nonumber
    \end{split}
\end{equation}
This implies that, for every $0\leq\delta\leq 1/100$, up to passing to a non-relabelled subsequence in $k$, we can assume without loss of generality that for any $k\in\N$ we have
$$
1-\frac{\mathcal{C}^h(P_{\mathbb V_j}(B(x,r_k)\cap \Gamma_j))}{r_k^h\mathcal{C}^h(P_{\mathbb V_j}(B(0,1)\cap C_{\mathbb V_j}(\varepsilon)))}= 1-\frac{\mathcal{C}^h(P_{\mathbb V_j}(B(x,r_k)\cap xC_{\mathbb V_j}(\varepsilon)\cap \Gamma_j))}{\mathcal{C}^h(P_{\mathbb V_j}(B(x,r_k)\cap xC_{\mathbb V_j}(\varepsilon)))}\leq \frac{\delta+\Delta_j(\varepsilon)}{\mathcal{C}^h(P_{\mathbb V_j}(B(0,1)\cap C_{\mathbb V_j}(\varepsilon)))}.
$$

Now, let us fix a $k\in\N$ sufficiently large such that
$\lvert \mathcal{C}^h\llcorner \Gamma_j(B(x,r_k))/r_k^h-\Theta^h_*(\mathcal{C}^h\llcorner\Gamma_j,x)\rvert\leq \delta$, and let $\Gamma^\prime_j \subseteq \Gamma_j$ be a Borel set such that $\lvert\mathcal{C}^h(B(x,r_k)\cap \Gamma_j)-\mathcal{C}_0^h(B(x,r_k)\cap \Gamma_j')\rvert\leq \delta r_k^h$. Finally, we choose a covering with balls $\{B(y_\ell,s_\ell)\}_{\ell\in \N}$ of $\Gamma_j' \cap B(x,r_k)$, with $y_\ell\in\Gamma_j'$,  such that $\lvert \sum_{\ell\in \N}s_\ell^h-\mathcal{C}^h_0 (B(x,r_k)\cap \Gamma_j^\prime)\rvert\leq \delta r_k^h$. This implies in particular that
\begin{equation}
    \lvert \mathcal{C}^h(B(x,r_k)\cap \Gamma_j)-\sum_{\ell\in\N}s_\ell^h\rvert\leq\lvert\mathcal{C}^h(B(x,r_k)\cap \Gamma_j)-\mathcal{C}_0^h(B(x,r_k)\cap \Gamma_j')\rvert+\lvert \sum_{\ell\in \N}s_\ell^h-\mathcal{C}^h_0 (B(x,r_k)\cap \Gamma_j^\prime)\rvert\leq 2\delta r_k^h.
    \label{eq:num5}
\end{equation}

The above inequalities together imply in particular that for such a $k\in\mathbb N$ we have
\begin{equation}
\begin{split}
      &1-\frac{\delta+\Delta_j(\varepsilon)}{\mathcal{C}^h(P_{\mathbb V_j}(B(0,1)\cap C_{\mathbb V_j}(\varepsilon)))}\leq \frac{\mathcal{C}^h(P_{\mathbb V_j}(B(x,r_k)\cap \Gamma_j))}{r_k^h\mathcal{C}^h(P_{\mathbb V_j}(B(0,1)\cap C_{\mathbb V_j}(\varepsilon)))}\leq\frac{\mathcal{C}^h(P_{\mathbb V_j}(B(x,r_k)\cap \Gamma_j^\prime))+\oldC{BallC}(\mathbb V_j,\mathbb L_j)\delta r_k^h}{r_k^h\mathcal{C}^h(P_{\mathbb V_j}(B(0,1)\cap C_{\mathbb V_j}(\varepsilon)))}\\
     &\leq\frac{\mathcal{C}^h(P_{\mathbb V_j}(\bigcup_{\ell\in\N}B(y_\ell,s_\ell)\cap y_\ell C_{\mathbb{V}_j}(\varepsilon)))+\oldC{BallC}\delta r_k^h}{r_k^h\mathcal{C}^h(P_{\mathbb V_j}(B(0,1)\cap C_{\mathbb V_j}(\varepsilon)))}
      \leq r_k^{-h}\sum_{    \ell\in\N} s_\ell^h+\frac{\oldC{BallC}\delta}{\mathcal{C}^h(P_{\mathbb V_j}(B(0,1)\cap C_{\mathbb V_j}(\varepsilon)))},\end{split}
  \label{eq:num6}
\end{equation}
where in the second inequality we used 
\begin{equation}
\begin{split}
\mathcal{C}^h(P_{\mathbb V_j}(B(x,r_k)\cap \Gamma_j))&-\mathcal{C}^h(P_{\mathbb V_j}(B(x,r_k)\cap \Gamma_j'))=\mathcal{C}^h(P_{\mathbb V_j}(B(x,r_k)\cap \Gamma_j\setminus B(x,r_k)\cap \Gamma_j'))\\
&\leq \oldC{BallC}\mathcal{C}^h(B(x,r_k)\cap \Gamma_j\setminus B(x,r_k)\cap \Gamma_j')=\oldC{BallC}\left(\mathcal{C}^h(B(x,r_k)\cap \Gamma_j)-\mathcal{C}^h(B(x,r_k)\cap \Gamma_j')\right) \\
&\leq \oldC{BallC}\left(\mathcal{C}^h(B(x,r_k)\cap \Gamma_j)-\mathcal{C}^h_0(B(x,r_k)\cap \Gamma_j')\right)\leq \oldC{BallC}\delta r_k^h,
\end{split}
\end{equation}
that is true taking into account the second part of \cref{cor:2.2.19}, the fact that $P_{\mathbb V_j}$ is injective on $\Gamma_j$, see \cref{prop:ConeAndGraph}, and the fact that $\mathcal{S}^h\leq \mathcal{C}^h$ by definition.
Hence putting together \eqref{eq:num5} and \eqref{eq:num6} we deduce that, for $k$ large enough,
$$
1-\frac{(1+\oldC{BallC})\delta+\Delta_j(\varepsilon)}{\mathcal{C}^h(P_{\mathbb V_j}(B(0,1)\cap C_{\mathbb V_j}(\varepsilon)))}\leq r_k^{-h}\sum_{\ell\in\N}s_\ell^h\leq\mathcal{C}^h(B(x,r_k)\cap \Gamma_j)/r_k^h+2\delta\leq\Theta^h_*(\mathcal{C}^h\llcorner \Gamma_j,x)+3\delta.
$$
Thanks to the arbitrariness of $\delta$, this implies that for $\mathcal{C}^h\llcorner \Gamma_j$-almost every $x$ we have, by making use of \eqref{es:num3}
\begin{equation}
    \begin{split}
        1-\daleth(\varepsilon,j):=1-\frac{\mathcal{C}^h(P_{\mathbb{V}_j}(B(0,1+3 m(\eta_0)^{-2}\varepsilon)\setminus B(0,1)\cap C_{\mathbb{V}_j}(\varepsilon)))}{\mathcal{C}^h(P_{\mathbb V_j}(B(0,1)\cap C_{\mathbb V_j}(\varepsilon)))}&\leq 1-\frac{\Delta_j(\varepsilon)}{\mathcal{C}^h(P_{\mathbb V_j}(B(0,1)\cap C_{\mathbb V_j}(\varepsilon)))}\\
        &\leq \Theta^h_*(\mathcal{C}^h\llcorner \Gamma_j,x)\leq \Theta^h_*(\mathcal{C}^h\llcorner E,x).
        \label{eqn:EquazioneIncriminata}
    \end{split}
\end{equation}
We now wish to get a bound from above of $\daleth(\varepsilon,j)$ that does not depend on $j$. In order to do this, we first of all let $\rho_1(\varepsilon):=1+3 m(\eta_0)^{-2}\varepsilon$ and $\rho_2(\varepsilon):=1-3m(\eta_0)^{-1}\varepsilon\rho_1(\varepsilon)/(m(\eta_0)-\varepsilon)>0$, thanks to \eqref{eqn:BoundStr}.
We claim that the following inclusion holds
\begin{equation}
    \begin{split}
        P_{\mathbb{V}_j}(B(0,\rho_1(\varepsilon))\setminus B(0,1)\cap C_{\mathbb{V}_j}(\varepsilon))
        \subseteq P_{\mathbb{V}_j}(B(0,\rho_1(\varepsilon))\cap C_{\mathbb{V}_j}(\varepsilon))\setminus P_{\mathbb{V}_j}(B(0,\rho_2(\varepsilon))\cap C_{\mathbb{V}_j}(\varepsilon)).
        \label{eq:num4}
    \end{split}
\end{equation}
By definition, if $y\in P_{\mathbb{V}_j}(B(0,\rho_1(\varepsilon))\setminus B(0,1)\cap C_{\mathbb{V}_j}(\varepsilon))$, there exists an $\ell_1\in \mathbb{L}_j$ such that $y\ell_1\in B(0,\rho_1(\varepsilon))\setminus B(0,1)\cap C_{\mathbb{V}_j}(\varepsilon)$. Notice that if $\ell\in \mathbb{L}_j$ is such that $y\ell\in C_{\mathbb{V}_j}(\varepsilon)$, by \cref{cor:2.2.19} we have
$$
m(\eta_0)\lVert \ell\rVert\leq\oldC{ProjC}(\mathbb{V}_j,\mathbb{L}_j)\lVert \ell\rVert\leq\dist(y\ell,\mathbb{V}_j)\leq\varepsilon\lVert y\ell\rVert\leq \varepsilon\lVert y\rVert+\varepsilon\lVert \ell\rVert,
$$
and then $\|\ell\|\leq \varepsilon\|y\|/(m(\eta_0)-\varepsilon)$.
This implies that in order to prove that $y$ does not belong to $P_{\mathbb{V}_j}(B(0,\rho_2(\varepsilon))\cap C_{\mathbb{V}_j}(\varepsilon))$ we just need to show that $y\ell\not \in B(0,\rho_2(\varepsilon))\cap C_{\mathbb{V}_j}(\varepsilon)$ for any $\ell\in B(0,\varepsilon\lVert y\rVert/(m(\eta_0)-\varepsilon))\cap \mathbb{L}_j$. This however, follows from the following computation
$$\lVert y\ell \rVert\geq \lVert y\ell_1 \rVert-\lVert \ell_1^{-1}\ell \rVert\geq 1-\frac{2\varepsilon\lVert y\rVert}{m(\eta_0)-\varepsilon}>\rho_2(\varepsilon),$$
where the last inequality comes from the fact that $\|\ell_1^{-1}\ell\|\leq \|\ell_1^{-1}\|+\|\ell\|\leq 2\varepsilon\|y\|/(m(\eta_0)-\varepsilon)$, and $\lVert y\rVert\leq \oldC{ProjC}(\mathbb{V}_j,\mathbb{L}_j)^{-1}\lVert y\ell_1\rVert\leq m(\eta_0)^{-1}\rho_1(\varepsilon)$.
This proves \eqref{eq:num4} and in turn the inequality
$\daleth(\varepsilon,j)\leq\rho_1(\varepsilon)^h-\rho_2(\varepsilon)^h$,  by homogeneity of $\mathcal{C}^h$. Furthermore, since on the right-hand side of the previous inequality we have an expression independent on $j$ we conclude that, by exploiting \eqref{eqn:EquazioneIncriminata}, for $\mathcal{C}^h$-almost every $x\in \Gamma_j$ we have
$$1-(\rho_1(\varepsilon)^h-\rho_2(\varepsilon)^h)\leq \Theta^h_*(\mathcal{C}^h\llcorner E,x).$$
Thanks to the arbitrariness of $j$ and to the fact that $\mathcal{C}^h(E\setminus \cup_{j}\Gamma_j)=0$, we deduce that the previous inequality holds for $\mathcal{C}^h$-almost every $x\in E$. Since $\varepsilon$ can be chosen arbitrarily small, we conclude that $\Theta^h_*(\mathcal{C}^h\llcorner E,x)\geq 1$, and then $\Theta^h(\mathcal{C}^h\llcorner E,x)=1$ for $\mathcal{C}^h$-almost every $x\in E$.
\medskip

Eventually, \cref{prop:convergence} together with \eqref{eq:num1} concludes that for $\mathcal{C}^h$-almost every $x\in E$ and for any $\nu\in\Tan_h(\mathcal{C}^h\llcorner E,x)$ the support of $\nu$ is contained in $\mathbb{V}(x)$. In addition to this, from \cref{regularity:density} and \cref{regularity:density2} we have that for $\mathcal{C}^h\llcorner E$-almost every $x\in\mathbb G$ we have $\Tan_h(\mathcal{C}^h\llcorner E,x)=\{\mathcal{C}^h\llcorner \mathbb{V}(x)\}$. This concludes the proof of the fact that $\mathcal{C}^h\llcorner E$ is $\mathscr{P}_c^h$-rectifiable and this comes in contradiction with the fact that $E\subseteq \Gamma$ has positive $\mathcal{S}^h\llcorner\Gamma$-measure by construction and \eqref{eq:gammaunrect}.
\end{proof}

%
%

Let us now verify that an intrinsically differentiable graph satisfies the hypothesis of \cref{prop:rett.1}. First, we recall the definition of intriniscally differentiable graph.
\begin{definizione}[Intrinsically differentiable graph]\label{defiintrinsicdiffgraph}
   Let $\mathbb V$ and $\mathbb L$ be two complementary subgroups of a Carnot group $\mathbb G$, with $h:=\dim_H\mathbb V$. Let ${\varphi}:A\subseteq \mathbb V \to\mathbb L$ be a continuous function with $A$ Borel in $\mathbb V$. Let $a_0\in A$. We say that $\Gamma:=\mathrm{graph}(\varphi):=\{a\cdot\varphi(a):a\in A\}$ is an {\em $h$-dimensional intrinsically differentiable graph} at $w\in\Gamma$ if there exists a homogeneous subgroup $\mathbb V(w)$ such that for all $k>0$
    \begin{equation}\label{eqn:TANGENTE}
\lim_{\lambda \to \infty } d_{H,\mathbb G}\left(\delta _\lambda (w^{-1}\cdot\Gamma )\cap B(0,k), \mathbb V(w)\cap B(0,k)\right)=0,
    \end{equation}
    where $d_{H,\mathbb G}$ is the Hausdorff distance between closed subsets of $\mathbb G$. We will call $\mathbb V(w)$ the {\em Hausdorff tangent} of $\Gamma$ at $w$.
    \end{definizione}
    
    \begin{lemma}\label{prop:idiffapproximatetangent}
    Let $\varphi:A\subseteq \mathbb V\to\mathbb L$ be a map such that $\Gamma:=\mathrm{graph}(\varphi)$ is an intrinsically differentiable graph at $w\in\Gamma$ with tangent $\mathbb V(w)$. Then, for every $\beta$ there exists $\rho=\rho(\beta)$ such that 
    $$
    \Gamma\cap B(w,\rho)\subseteq w C_{\mathbb V(w)}(\beta).
    $$
    \end{lemma}
    \begin{proof}
    We first claim that for every $\varepsilon>0$ there exists $r_0:=r_0(\varepsilon)$ such that 
    \begin{equation}\label{eqn:HAUSHAUSHAUS}
    \sup_{p\in \Gamma\cap B(w,r)}\mathrm{dist}(p,w\mathbb V(w))\leq \varepsilon r, \qquad \text{for all $0<r\leq r_0$}.
    \end{equation}
    Indeed, this follows just by taking $k=1$ in the definition \eqref{eqn:TANGENTE} and by exploiting the very definition of Hausdorff distance.
    
    Now let us take $\varepsilon\leq \beta/2$. We claim that $\Gamma\cap B(w,r_0(\varepsilon))\subseteq wC_{\mathbb V(w)}(\beta)$. Indeed, let $p\in \Gamma\cap B(w,r_0(\varepsilon))$, and $k\geq 1$ be such that $r_02^{-k}< \|w^{-1}\cdot p\|\leq r_02^{-k+1}$. Since $p\in \Gamma\cap B(w,r_02^{-k+1})$, from \eqref{eqn:HAUSHAUSHAUS} we get
    $$
    \mathrm{dist}(p,w\mathbb V(w))\leq \varepsilon r_02^{-k+1} \leq 2\varepsilon\|w^{-1}\cdot p\|\leq \beta \|w^{-1}\cdot p\|,
    $$
    thus showing the claim.
    \end{proof}
    
We are now ready to give the proof of \cref{thm:INTRO1Equivalence}.    
\begin{proof}[Proof of {\cref{thm:INTRO1Equivalence}}]
We prove different implications in separate points.

\begin{enumerate}
    \item[1.$\Rightarrow$2.] If $\mathcal{S}^h\llcorner\Gamma$ is $\mathscr{P}_h^c$-rectifiable, then $\mathcal{S}^h\llcorner\Gamma$ is asymptoticall doubling. Hence, by a routine argument (cf. \cite[Remark 14.4(3)]{Mattila1995GeometrySpaces}) we get that, for $\mathcal{S}^h\llcorner\Gamma$-almost every $x\in\mathbb G$, every element in $\mathrm{Tan}(\mathcal{S}^h\llcorner\Gamma,x)$ is a constant multiple of an element of $\mathrm{Tan}_h(\mathcal{S}^h\llcorner\Gamma,x)$, which is by hypothesis of the form $\lambda\mathcal{S}^h\llcorner\mathbb V(x)$ with $\mathbb V(x)\in\G_c(h)$, whence the conclusion.
    \item[2.$\Rightarrow$1.] It follows from \cref{thm:ExistenceOfDensityPlus} by approximating the Borel set $\Gamma$ from within by compact sets.
    \item[1.$\Rightarrow$3.] It is a consequence of \cite[Theorem 1.8]{antonelli2020rectifiable},  \cref{prop:TangentCompemented}, and the fact that the Hausdorff tangent at $\mathcal{S}^h\llcorner\Gamma_i$-almost ever $x$ of $\Gamma_i$ is complemented since it coincides almost everywhere with the subgroup on which it is supported the tangent measure.
    \item[3.$\Rightarrow$1.] Since, for every $i\in\mathbb N$, $\Gamma_i$ is an intriniscally differentiable graph at $\mathcal{S}^h\llcorner\Gamma_i$-almost every point of it, by \cref{prop:idiffapproximatetangent} we conclude that the hypothesis of \cref{prop:rett.1} is verified. Hence, for every $i\in\mathbb N$, $\mathcal{S}^h\llcorner\Gamma_i$ is $\mathscr{P}_h^c$-rectifiable. Hence, since $\mathcal{S}^h(\Gamma\setminus\cup_{i=1}^{+\infty}\Gamma_i)=0$, by a routine argument involving the locality of tangents and the Lebesgue differentiation theorem, see \cref{prop:Lebesuge}, we conclude that $\mathcal{S}^h\llcorner\Gamma$ is $\mathscr{P}_h^c$-rectifiable as well.
\end{enumerate}
\medskip
Let us show that the item 3. implies the last part of the statement. Since, for every $i\in\mathbb N$, $\Gamma_i$ is intrinsically differentiable, arguing as above we can apply \cref{prop:rett.1} and then conclude that $\Theta^h(\mathcal{C}^h\llcorner\Gamma_i,x)=1$ for $\mathcal{C}^h\llcorner\Gamma_i$-almost every $x\in\mathbb G$. Hence, from the Lebesgue differentiation theorem in \cref{prop:Lebesuge}, we conclude that, for every $i\in\mathbb N$, $\Theta^h(\mathcal{C}^h\llcorner\Gamma,x)=1$ for $\mathcal{C}^h\llcorner\Gamma_i$-almost every $x\in\mathbb G$, and hence the same conclusion holds for $\mathcal{C}^h\llcorner\Gamma$-almost every $x\in\mathbb G$ since $\mathcal{C}^h(\Gamma\setminus\cup_{i=1}^{+\infty}\Gamma_i)=0$. The last convergence result is a direct consequence of the fact that the density is 1 and \cite[Proposition 2.26]{antonelli2020rectifiable}.
\end{proof}

\section{Area formula}\label{sec:AreaFormula}
In this section $\mathbb G$ is an arbitrary Carnot group, and $\mathbb V$ and $\mathbb L$ are two homogeneous complementary subgroups, i.e., such that $\mathbb G=\mathbb V\cdot\mathbb L$ and $\mathbb V\cap\mathbb L=\{0\}$. Moreover, let $h$ be the Hausdorff dimension of $\mathbb V$. We equip $\mathbb G$ with an arbitrary fixed left-invariant homogeneous distance $d$ that sometimes will be understood.




\begin{lemma}[{\cite[Proposition 3.1.5]{FMS14}}]\label{lem:ComplGraph}
Let $\mathbb V,\mathbb L$ be two complementary subgroups in $\mathbb G$. Let $\mathbb P$ be a homogeneous subgroup that is a complementary subgroup of $\mathbb L$. Then there exists a map $\varphi_{\mathbb P}:\mathbb V\to\mathbb L$ such that $\mathbb P=\Phi_{\mathbb P}(\mathbb V):=\mathbb V\cdot\varphi_P(\mathbb V)$.
\end{lemma}


\begin{definizione}[Area factor, {\cite[Lemma 3.2]{JNGV20}}]\label{def:areafactorCENTR}
Let $\mathbb V,\mathbb L$ be two complementary subgroups in $\mathbb G$. Let $\mathbb P$ be a homogeneous subgroup that is a complementary subgroup of $\mathbb L$. Take $\varphi_{\mathbb P}:\mathbb V\to\mathbb L$ as in \cref{lem:ComplGraph} and let $\Phi_{\mathbb P}:v\mapsto v\cdot\varphi_{\mathbb P}(v)$ be its graph map. Then the {\em centered area factor} of $\mathbb P$ with respect to the splitting $\mathbb V\cdot\mathbb L$ is the unique $0<\mathcal{A}(\mathbb P)<+\infty$ such that
\begin{equation}\label{eqn:EqualityOfMeasures}
\mathcal{C}^h\llcorner \mathbb P=\mathcal{A}(\mathbb P)(\Phi_{\mathbb P})_*(\mathcal{C}^h\llcorner\mathbb V).
\end{equation}
\end{definizione}

Let $\mathbb P$, $\varphi_{\mathbb P}$, and $\Phi_{\mathbb P}:\mathbb V\to\mathbb P$ be as in \cref{def:areafactorCENTR}. It is readily seen that $(\Phi_{\mathbb P})_*(\mathcal{C}^h\llcorner\mathbb V)$ is a Haar measure on $\mathbb P$, compare with the beginning of the proof of \cite[Lemma 3.2]{JNGV20}. From \cite[Theorem 3.1]{FSSC15} we conclude
\begin{equation}\label{eqn:WhoIsTheAreaFactor?}
    \mathcal{A}(\mathbb P)^{-1}=\limsup_{r\to 0} \frac{((\Phi_{\mathbb P})_*(\mathcal{C}^h\llcorner\mathbb V))(B(0,r))}{r^h}=\limsup_{r\to 0} \frac{\mathcal{C}^h(P_{\mathbb V}(B(0,r)\cap\mathbb P))}{r^h}=\mathcal{C}^h(P_{\mathbb V}(B(0,1)\cap\mathbb P)),
\end{equation}
where in the last equality we used the homogeneity of $\mathcal{C}^h\llcorner\mathbb V$.

\begin{lemma}\label{lem:ContinuousAreaFactorCENTR}
Given the splitting $\mathbb G=\mathbb V\cdot\mathbb L$, the area factor $\mathcal{A}(\cdot)$ is continuous on the set of homogeneous subgroups that have $\mathbb L$ as a complementary subgroup.
\end{lemma}
\begin{proof}
It directly follows from the explicit expression in \eqref{eqn:WhoIsTheAreaFactor?} together with a simple argument that can be found, e.g., at the end of \cite[Proof of Lemma 3.2]{JNGV20}.
\end{proof}

\begin{definizione}[Elementary $\mathscr{P}_h$-rectifiable graph]\label{def:ElementaryPh}
Let $\mathbb V,\mathbb L$ be two homogeneous complementary subgroups of a Carnot group $\mathbb G$, and $\alpha\leq\oldep{ep:Cool}(\mathbb V,\mathbb L)$. We say that a compact set $\Gamma$ is an $\alpha$-{\em elementary $\mathscr{P}_h^c$-rectifiable graph with respect to $\mathbb V$ and $\mathbb L$} if the following four conditions hold 
\begin{enumerate}
    \item[(i)] $\Gamma$ is a compact $C_{\mathbb V}(\alpha)$-set of $\mathcal{S}^h$-finite measure and thus it is  the intrinsic graph of a continuous map $\varphi:A\subseteq\mathbb V\to\mathbb L$, with $A:=P_{\mathbb V}(\Gamma)$, see \cref{prop:ConeAndGraph},
    \item[(ii)] $\mathcal{S}^h\llcorner\Gamma$ is a $\mathscr{P}_h^c$-rectifiable measure,
    \item[(iii)] for $\mathcal{S}^h\llcorner\Gamma$-almost every $x\in\mathbb G$, the subgroup $\tau(\Gamma,x):=\tau(\mathcal{S}^h\llcorner\Gamma,x)$ defined in \cref{lem:BorelTangents} is complementary to $\mathbb L$,
   \item[(iv)] The value of $\mathcal{A}(\tau(\Gamma,x))$ is uniformly bounded above for $\mathcal{S}^h\llcorner\Gamma$-almost every $x\in\mathbb G$, where $\mathcal{A}$ is the centered area factor defined in \cref{def:areafactorCENTR}.
\end{enumerate}
\end{definizione}

For the crucial limit result in \cref{prop:CrucialLimit}, we need the following adaptation of \cite[Proposition 4.10]{antonelli2020rectifiable}.
\begin{proposizione}
\label{prop:vitali2PDAdapted}
Let $\mathbb V,\mathbb L$ be complementary subgroups of a Carnot group $\mathbb G$. Let us fix $\alpha\leq\oldep{ep:Cool}(\mathbb V,\mathbb L)$ and suppose that $\Gamma$ is a compact $C_\mathbb{V}(\alpha)$-set of finite $\mathcal{S}^h$-measure. For $\mathcal{S}^h\llcorner\Gamma$-almost every $x\in\mathbb G$, let $\mathcal{C}_x:=C_{\mathbb V(x)}(\beta_x)$, for some $\mathbb V(x)\in \G_c(h)$ that is complemented by $\mathbb L$, and some $\beta_x>0$. Let us further assume that $\mathcal{A}(\mathbb V(x))$, defined with respect to the splitting $\mathbb G=\mathbb V\cdot\mathbb L$ (see \cref{def:areafactorCENTR}) is uniformly bounded above by a constant $C$ for $\mathcal{S}^h\llcorner\Gamma$-almost every $x\in\mathbb G$. As in \cref{prop:mutuallyabs}, let us denote with $\Phi:P_{\mathbb V}(\Gamma)\to\mathbb G$ the graph map of $\varphi:P_{\mathbb V}(\Gamma)\to \mathbb L$ whose intrinsic graph is $\Gamma$.

Then for $\mathcal{S}^h$-almost every $w\in P_\mathbb{V}(\Gamma)$ we have
\begin{equation}
    \lim_{r\to 0}\frac{\mathcal{S}^h\big(P_\mathbb{V}\big(B(\Phi(w),r)\cap \Phi(w)\mathcal{C}_{\Phi(w)}\big)\cap P_\mathbb{V}(\Gamma)\big)}{\mathcal{S}^h\big(P_\mathbb{V}\big(B(\Phi(w),r)\cap \Phi(w)\mathcal{C}_{\Phi(w)}\big)\big)}=1.
    \label{eq:limiAdapted}
\end{equation}
\end{proposizione}
\begin{proof}
The proof is almost identical to the one of \cite[Proposition 4.10]{antonelli2020rectifiable}, and we outline just the main changes. First, the fine covering $S$ is exactly the same of \cite[Proposition 4.10]{antonelli2020rectifiable}, except that in $(\beta)$ one defines $G(w,r):=P_\mathbb{V}(B(\Phi(w),r)\cap \Phi(w)C_{\mathbb V(\Phi(w))}(\beta_{\Phi(w)}))$ for $w\in P_{\mathbb V}(\Gamma)$. Thus we need to check \cite[Equation (84)]{antonelli2020rectifiable} with the newly defined covering $S$. The verification of \cite[Equation (84)]{antonelli2020rectifiable} for the part of the covering in $(\alpha)$ is precisely the same as in \cite{antonelli2020rectifiable}. Moreover, as it is readily seen by how the estimates are made, arguing verbatim as in \cite{antonelli2020rectifiable}, and by using the same notation therein, we get whenever $w\in P_{\mathbb V}(\Gamma)$ and $0<r<r(w)$ is sufficiently small we have
$$
\hat{G}(w,r)\subseteq P_\mathbb{V}(B(\Phi(w),50(A+1)r)) \cup P_\mathbb{V}(B(\Phi(w), C(\Gamma)r)),
$$
where $C(\Gamma)$ is a suitable constant depending only on $\Gamma$.
Hence whenever $w\in P_{\mathbb V}(\Gamma)$ and $0<r<r(w)$ is sufficiently small we have, by exploiting the homogeneity of $\mathcal{C}^h$ and the invariance properties in \cref{prop:InvarianceOfProj},
$$
\frac{\mathcal{C}^h(\hat{G}(w,r))}{\mathcal{C}^h(G(w,r))}\leq \frac{\mathcal{C}^h(P_{\mathbb V}(B(0,\max{50(A+1),C(\Gamma)})))}{\mathcal{C}^h(P_\mathbb V(B(0,1)\cap C_{\mathbb V(\Phi(w))}(\beta_{\Phi(w)})))}\leq C\mathcal{C}^h(P_{\mathbb V}(B(0,\max{50(A+1),C(\Gamma)})))
$$
where the last inequality is true since $\mathcal{C}^h(P_\mathbb V(B(0,1)\cap C_{\mathbb V(\Phi(w))}(\beta_{\Phi(w)})))\geq \mathcal{C}^h(P_{\mathbb V}(B(0,1)\cap \mathbb V(\Phi(w))))=\mathcal{A}(\mathbb V(\Phi(w)))^{-1}\geq C^{-1}$. Since the ratio of centered Hausdorff measures on $\mathbb V$ is the same as the ratio of spherical Hausdorff measures, the previous estimate allows to conclude that \cite[Equation (84)]{antonelli2020rectifiable} holds true also for $w\in P_{\mathbb V}(\Gamma)$ for this newly defined covering $S$ described above. Hence applications of standard differentiation results allow to conclude the proof of \eqref{eq:limiAdapted} as in \cite{antonelli2020rectifiable}.
\end{proof}

\begin{proposizione}\label{prop:CrucialLimit}
Let $\mathbb V,\mathbb L$ be two complementary subgroups of $\mathbb G$ and let $\alpha\leq\oldep{ep:Cool}(\mathbb V,\mathbb L)$ and. Let $\Gamma$ be an $\alpha$-elementary $\mathscr{P}_h^c$-rectifiable graph with respect to $\mathbb V,\mathbb L$, see \cref{def:ElementaryPh}. Then for $\mathcal{S}^h\llcorner\Gamma$ almost every $x\in \mathbb G$, we have the following equality 
\begin{equation}\label{eqn:LHS=RHSCENTR}
\begin{split}
    \lim_{r\to 0}r^{-h}\mathcal{C}^h(P_{\mathbb V}(B(x,r)\cap\Gamma)) = \mathcal{C}^h(P_{\mathbb V}(B(0,1)\cap\tau(\Gamma,x))),
    \end{split}
\end{equation}
where $\tau(\Gamma,x)$ is the tangent plane at $x$ introduced in \cref{lem:BorelTangents}.
\end{proposizione}

\begin{proof}
Let us notice first that by means of \cref{prop:InvarianceOfProj} and from the homogeneity of the measure we get that 
\begin{equation}\label{eqn:Bastacosi}
\mathcal{C}^h(P_{\mathbb V}(B(0,1)\cap\delta_{\lambda}(x^{-1}\Gamma)))=\lambda^{h}\mathcal{C}^h(P_{\mathbb V}(B(x,\lambda^{-1})\cap \Gamma)),
\end{equation}
for every $x\in\Gamma$, $\lambda>0$. Let us call $\Phi:A\subseteq\mathbb V\to\mathbb L$ the graph map of $\varphi$ as in item (i) of \cref{def:ElementaryPh}. From \cref{prop:mutuallyabs} we get that the measure $\Phi_*(\mathcal{C}^h\llcorner\mathbb V)$ is mutually absolutely continuous with respect to $\mathcal{C}^h\llcorner\Gamma$. As a consequence, if we fix $\vartheta,\gamma\in\mathbb N$, we have that $\mathcal{C}^h\llcorner\Gamma$-almost every point $x$ in $E(\vartheta,\gamma)$, see \cref{def:EThetaGamma}, is a point of density one for the measure $\Phi_*(\mathcal{C}^h\llcorner\mathbb V)$, that is to say for every $\vartheta,\gamma\in\mathbb N$ we have that 
$$
\lim_{r\to 0} \frac{(\Phi_*(\mathcal{C}^h\llcorner\mathbb V))(B(x,r))}{(\Phi_*(\mathcal{C}^h\llcorner\mathbb V))(B(x,r)\cap E(\vartheta,\gamma))}=\lim_{r\to 0} \frac{\mathcal{C}^h(P_{\mathbb V}(B(x,r)\cap\Gamma))}{\mathcal{C}^h(P_{\mathbb V}(B(x,r)\cap E(\vartheta,\gamma)))}=1, \qquad \text{for $\mathcal{C}^h\llcorner\Gamma$-almost every $x\in E(\vartheta,\gamma)$.}
$$
From the previous equality, \cref{prop::E}, identity \eqref{eqn:Bastacosi} and the invariance properties of \cref{prop:InvarianceOfProj}, we conclude that it is sufficient to prove that for every $\vartheta,\gamma\in\mathbb N$ we have that 
\begin{equation}\label{eqn:LHS=RHSCENTRNUOVO}
\begin{split}
    \lim_{r\to 0}\mathcal{C}^h(P_{\mathbb V}(B(0,1)\cap\delta_{1/r}(x^{-1}E(\vartheta,\gamma)))) = \mathcal{C}^h(P_{\mathbb V}(B(0,1)\cap\tau(\Gamma,x))),
    \end{split}
\end{equation}
for $\mathcal{C}^h\llcorner\Gamma$-almost every $x\in E(\vartheta,\gamma)$.


From now on we assume $\vartheta,\gamma\in\mathbb N$ to be fixed. Thanks to \cite[Proposition 3.2]{antonelli2020rectifiable} for $\mathcal{C}^h\llcorner\Gamma$-almost every $x\in E(\vartheta,\gamma)$ and any $\beta\leq\oldC{ProjC}(\tau(\Gamma,x),\mathbb L)$, there exists a $\tilde\varrho(x,\beta)$ such that $E(\vartheta,\gamma)\cap B(x,r)\subseteq xC_{\tau(\Gamma,x)}(\beta)$ for very $0<r<\widetilde\varrho(x,\beta)$. For such an $x$ and $\beta\leq\oldC{ProjC}(\tau(\Gamma,x),\mathbb L)$, note that \cref{prop:EstimateOnProjection} with the choices $\Gamma=E(\vartheta,\gamma)$ and $\rho=\tilde\varrho(x,\beta)$ allows us to infer that
\begin{equation}\label{eqn:CrucialLimit3}
\left|\frac{\mathcal{C}^h(P_{\mathbb V}(B(x,r)\cap xC_{\tau(\Gamma,x)}(\beta))\cap P_{\mathbb V}(E(\vartheta,\gamma)))}{r^h}-\frac{\mathcal{C}^h(P_{\mathbb V}(B(x,r)\cap xC_{\tau(\Gamma,x)}(\beta)\cap E(\vartheta,\gamma)))}{r^h}\right|\leq \Delta_x (\beta),
\end{equation}
for any $0<r<\varrho(\tilde\varrho(x,\beta),\alpha)$. In addition to this, \cref{prop:vitali2PDAdapted}, the homogeneity of $\mathcal{C}^h\llcorner\mathbb V$, and the invariance properties of \cref{prop:InvarianceOfProj} imply that, for any $\beta>0$, we get that for $\mathcal{C}^h\llcorner\Gamma$-almost every $x\in E(\vartheta,\gamma)$ we have
\begin{equation}\label{eq:limite}
    \begin{split}
        &     \qquad\qquad\qquad\qquad\lim_{r\to 0}\frac{\mathcal{C}^h(P_{\mathbb V}(B(x,r)\cap xC_{\tau(\Gamma,x)}(\beta))\cap P_{\mathbb V}(E(\vartheta,\gamma)))}{r^h}\\
        =&\lim_{r\to 0}\frac{\mathcal{C}^h(P_{\mathbb V}(B(x,r)\cap xC_{\tau(\Gamma,x)}(\beta))\cap P_{\mathbb V}(E(\vartheta,\gamma)))}{\mathcal{C}^h(P_{\mathbb V}(B(x,r)\cap xC_{\tau(\Gamma,x)}(\beta)))} 
        \cdot\frac{\mathcal C^h(P_{\mathbb V}(B(x,r)\cap xC_{\tau(\Gamma,x)}(\beta)))}{r^h}\\ &\qquad\qquad\qquad\qquad\qquad=\mathcal{C}^h(P_{\mathbb V}(B(0,1)\cap C_{\tau(\Gamma,x)}(\beta))).
    \end{split}
\end{equation}
Finally, from the continuity of measures, for $\mathcal{C}^h\llcorner\Gamma$-almost every $x\in E(\vartheta,\gamma)$ there exists a function $\Delta^\prime_x(\beta)$ with $\Delta^\prime_x(\beta)\to 0$ as $\beta\to 0$ (pointwise in $x$), and
\begin{equation}\label{eqn:CrucialLimit2}
    |\mathcal{C}^h(P_{\mathbb V}(B(0,1)\cap C_{\tau(\Gamma,x)}(\beta)))-\mathcal{C}^h(P_{\mathbb V}(B(0,1)\cap\tau(\Gamma,x)))|\leq \Delta^\prime_x(\beta), \qquad \text{for all $\beta>0$}.
\end{equation}
Let us define $E_n$ as the set of points $x$ in $E(\vartheta,\gamma)$ such that $\tau(\Gamma,x)$ exists and $\oldC{ProjC}(\tau(\Gamma,x),\mathbb L)>1/n$. Obviously $\mathcal{C}^h(E(\vartheta,\gamma)\setminus \cup_{n=1}^{+\infty} E_n)=0$. Thus it is sufficient to prove the claim \eqref{eqn:LHS=RHSCENTRNUOVO} for $\mathcal{C}^h\llcorner\Gamma$-almost every $x\in E_n$. Let us fix $n\in\mathbb N$.
The above discussion shows that, if we fix $\beta\leq 1/n$, then for $\mathcal{C}^h\llcorner\Gamma$-almost every $x\in E_n$ we have that \eqref{eqn:CrucialLimit3}, \eqref{eq:limite} and \eqref{eqn:CrucialLimit2} imply
\begin{equation}\label{eqn:INTRIGO}
    \begin{split}
       &\limsup_{r\to 0}\lvert \mathcal{C}^h(P_{\mathbb V}(B(0,1)\cap\delta_{1/r}(x^{-1}E(\vartheta,\gamma))))-\mathcal{C}^h(P_{\mathbb V}(B(0,1)\cap\tau(\Gamma,x)))\rvert\\
        &\leq \Delta^\prime_x(\beta)+\limsup_{r\to 0}\lvert \mathcal{C}^h(P_{\mathbb V}(B(0,1)\cap\delta_{1/r}(x^{-1}E(\vartheta,\gamma))))-\mathcal{C}^h(P_{\mathbb V}(B(0,1)\cap C_{\tau(\Gamma,x)}(\beta)))\rvert\\
        \leq\Delta^\prime_x(\beta)&+\limsup_{r\to 0}\Big\lvert \frac{\mathcal{C}^h(P_{\mathbb V}(B(x,r)\cap xC_{\tau(\Gamma,x)}(\beta)\cap  E(\vartheta,\gamma)))}{r^h}-\frac{\mathcal{C}^h(P_{\mathbb V}(B(x,r)\cap xC_{\tau(\Gamma,x)}(\beta))\cap P_{\mathbb V}(E(\vartheta,\gamma)))}{r^h}\Big\rvert\\
        &+\limsup_{r\to0}\Big\lvert \frac{\mathcal{C}^h(P_{\mathbb V}(B(x,r)\cap xC_{\tau(\Gamma,x)}(\beta))\cap P_{\mathbb V}(E(\vartheta,\gamma)))}{r^h}-\frac{\mathcal{C}^h(P_{\mathbb V}(B(x,r)\cap xC_{\tau(\Gamma,x)}(\beta)))}{r^h}\Big\rvert\\
        &\qquad\qquad\qquad\qquad\qquad\qquad\leq \Delta^\prime_x(\beta)+\Delta_x(\beta).
    \end{split}
\end{equation}
Thus by taking the intersection of the $\mathcal{C}^h\llcorner\Gamma$-full measure sets in $E_n$ on which the previous inequality holds for $\beta=1/m$, with $m\geq n$, we get that for $\mathcal{C}^h\llcorner\Gamma$-almost every $x\in E_n$, the previous inequality holds for every $\beta=1/m$, with $m\geq n$. By fixing an $x$ in such a set of full $\mathcal{C}^h\llcorner\Gamma$-measure in $E_n$ and taking $\beta=1/m$ in \eqref{eqn:INTRIGO} and $m\to +\infty$, we get the claim \eqref{eqn:LHS=RHSCENTRNUOVO} for $\mathcal{C}^h\llcorner\Gamma$-almost every $x\in E_n$, and thus the proof is concluded.

\end{proof}

\begin{proposizione}\label{prop:SplitInElementaryGraph}
There exist a family $\mathscr F:=\{\mathbb V_k\}_{k\in\mathbb N}\subseteq \G_c(h)$ and $\mathbb L_k$ complementary subgroups of $\mathbb V_k$ such that the following holds. If $\phi$ is a $\mathscr{P}_h^c$-rectifiable measure, there are continuous maps $\varphi_k:A_k\subseteq\mathbb V_k\to\mathbb L_k$, with $A_k$ compact, such that
\begin{itemize}
\item[(i)] for every $k\in\mathbb N$ we have $\Gamma_k:=\mathrm{graph}(\varphi_k)=A_k\cdot\varphi_k(A_k)$ is an $\alpha_k$-elementary  $\mathscr{P}_h^c$-rectifiable graph with respect to $\mathbb V_k$ and $\mathbb L_k$ for some $\alpha_k$, see \cref{def:ElementaryPh},
\item[(ii)] $\phi(\mathbb G\setminus \cup_{k\in\mathbb N}\Gamma_k)=0$,
    
\end{itemize}
\end{proposizione}

\begin{proof}
The result in \cite[Theorem 3.4]{antonelli2020rectifiable} implies that we can find countably many $\mathbb{V}_k\in\G_c(h)$ complemented by some $\mathbb L_k$ such that the following holds. If $\phi$ is a $\mathscr{P}_h^c$-rectifiable measure, then there exist compact sets $\Gamma_k$ such that
\begin{enumerate}
\item $\phi(\mathbb G\setminus \cup_{k\in\mathbb N}\Gamma_k)=0$,
    \item for any $k\in\N$ the set $\Gamma_k$ is a $C_\mathbb{V_k}(\min\{\oldep{ep:Cool}(\mathbb{V}_k,\mathbb{L}_k),\hbar_\mathbb{G}\})$-set, where $\hbar_\mathbb{G}>0$ is the constant in \cite[Proposition 2.8]{antonelli2020rectifiable}.
\end{enumerate}
It is immediate to see that the measures $\phi$ and $\mathcal{C}^h$ are mutually absolutely continuous (see, e.g., \cite[Proposition 2.6]{antonelli2020rectifiable}) and hence by the Lebesgue differentiation theorem and the locality of tangents, cf. \cref{prop:Lebesuge}, the measure $\mathcal{C}^h\llcorner\Gamma_k$ is still a $\mathscr{P}_h^c$-rectifiable measure. This proves that each $\Gamma_k$ verifies the hypothesis (i) and (ii) of \cref{def:ElementaryPh}. In order to check (iii) we note that that, from item 2. above, for $\mathcal{C}^h\llcorner \Gamma_k$-almost every $x$ we have that the tangent $\mathbb{V}(x)$ is contained in $ C_{\mathbb{V}_k}(\min\{\oldep{ep:Cool}(\mathbb{V}_k,\mathbb{L}_k),\hbar_\mathbb{G}\})=:C_k$. This implies thanks to \cite[Proof of Proposition 2.17]{antonelli2020rectifiable} that $\mathbb{V}(x)$ is a complementary subgroup of $\mathbb{L}_k$. In order to conclude the proof of item (iv) of \cref{def:ElementaryPh}, we must prove that for any $k\in\N$ there exists a constant $C>0$ such that
$$\mathcal{C}^h(P_{\mathbb V_k}(B(0,1)\cap\mathbb V(x)))^{-1}=:\mathcal{A}(\mathbb{V}(x))\leq C\qquad \text{for }\mathcal{C}^h\llcorner \Gamma_k\text{-almost any }x\in \mathbb{G}.$$
Since $\mathbb V(x)\subseteq C_k$ for $\mathcal{C}^h\llcorner\Gamma_k$-almost every $x\in\mathbb G$, it is sufficient to show that there exists a constant $c>0$ such that for any $\mathbb{W}\in \G_c(h)$ contained in $C_k$ we have
$$\mathcal{C}^h(P_{\mathbb V_k}(B(0,1)\cap\mathbb W))\geq c.$$
Suppose by contradiction that there exists a sequence of planes $\mathbb{W}_i\in \G_c(h)$ contained in $C_k$ such that
$$\mathcal{C}^h(P_{\mathbb V_k}(B(0,1)\cap\mathbb W_i))\leq i^{-1}.$$
The compactness result in \cref{prop:CompGrassmannian} implies that there exists a $\mathbb{W}\in\G(h)$ such that $\lim_{i\to 0}d_{\mathbb{G}}(\mathbb{W},\mathbb{W}_i)=0$. Since $\mathbb W_i\in C_k$ for every $i$ we also get $\mathbb W\in C_k$ and then, since the aperture of the cone $C_k$ is smaller than $\min\{\oldep{ep:Cool}(\mathbb{V}_k,\mathbb{L}_k),\hbar_\mathbb{G}\}$, \cite[Proof of Proposition 2.17]{antonelli2020rectifiable} we have that $\mathbb{L}_k$ and $\mathbb{W}$ are complementary subgroups, and thus $\mathbb W\in \G_c(h)$. Finally \cref{lem:ContinuousAreaFactorCENTR} implies that $\mathcal{C}^h(P_{\mathbb V_k}(B(0,1)\cap\mathbb W))=0$. This is not possible since the area factor $\mathcal{A}(\mathbb W)$ relative to the splitting $\mathbb G=\mathbb V_k\cdot\mathbb L_k$ should be finite, see \cref{def:areafactorCENTR}.
\end{proof}

\begin{teorema}[Area formula for the centered measure]\label{thm:AreaFormulaCENTR}
Let $\mathbb V,\mathbb L$ be two complementary subgroups of $\mathbb G$. Suppose further $\Gamma$ is an $\alpha$-elementary $\mathscr{P}_h^c$-rectifiable graph with respect to $\mathbb{V}$ and $\mathbb{L}$, see \cref{def:ElementaryPh}.
Then, for every Borel function $\psi:\Gamma\to[0,+\infty)$ we have 
\begin{equation}\label{eqn:AreaFormulaCENTR1}
\int_\Gamma \psi(w)d\mathcal{C}^h = \int_A \psi(a\cdot\varphi(a))\mathcal{A}( \tau(\Gamma,a\cdot\varphi(a)))d\mathcal{C}^h\llcorner \mathbb{V}.
\end{equation}
where $\mathcal{A}(\cdot)$ denotes the centered area factor introduced in \cref{def:areafactorCENTR}.
\end{teorema}

\begin{osservazione}
Note that the above expression is well defined since thanks to \cref{prop:mutuallyabs} the map $a\to \tau(\Gamma,a\cdot\varphi(a))$ is defined up to $\mathcal{C}^h$-null sets on $\mathbb{V}$.
\end{osservazione}

\begin{proof}
As a first step, let us show that the map $a\mapsto \mathcal{A}(\tau(\Gamma,a\cdot\varphi(a)))=:f(a)$ is $\mathcal{C}^h\llcorner \mathbb{V}$-measurable. To do so let us first recall that
\begin{enumerate}
    \item the map $a\mapsto a\cdot\varphi(a)$ is continuous from $A$ to $\Gamma$,
    \item the map $x\mapsto \tau(\Gamma,x)$ sending points of $\Gamma$ into elements of $\G_c(h)$ is $\mathcal{C}^h\llcorner \Gamma$-measurable thanks to \cref{lem:BorelTangents} and for $\mathcal{S}^h\llcorner \Gamma$-almost every $x\in \Gamma$ the plane $\tau(\Gamma,x)$ is a complementary subgroup of $\mathbb{L}$ thanks to \cref{def:ElementaryPh}(iii),
    \item thanks to \cref{lem:ContinuousAreaFactorCENTR}, the function $\mathbb W\mapsto\mathcal{A}(\mathbb W)$ is continuous when restricted to those $\mathbb{W}\in\G_c(h)$ that are complements of $\mathbb{L}$.
\end{enumerate}
Finally, items 1., 2. and 3. conclude the proof of the $\mathcal{C}^h\llcorner \mathbb{V}$-measurability of the function $f$.
In addition to this, thanks to \cref{def:ElementaryPh}(iv) we know that $f$ is uniformly bounded on $A$ and thus it is an element of $L^1_{\mathrm{loc}}(A)$.

We now introduce a measure $\mu$ supported on $\Gamma$ such that for any Borel set $E$ we have
$$
\mu(E)=\int_{P_{\mathbb V}(\Gamma\cap E)}f(a)d\mathcal{C}^h\llcorner\mathbb V(a).
$$
Since $\mathcal{C}^h\llcorner \Gamma$ is a $\mathscr{P}^c_h$-rectifiable measure, it is locally asymptotically doubling and thus \cref{prop:mutuallyabs} implies that $\mu\ll \mathcal{C}^h\llcorner\Gamma$. Therefore, if we are able to prove that $\Theta^{h,*}(\mu,x)=1$ for $\mu$-almost every $x\in \Gamma$, then \cite[Theorem 3.1]{FSSC15} concludes the proof.

Let us now proceed and prove that $\Theta^{h,*}(\mu,x)=1$ for $\mu$-almost every $x\in \Gamma$. As a first step, we note that
\begin{equation}\label{eqn:Conc4CENTR}
    \Big|r^{-h}\int_{P_{\mathbb V}( B(z\cdot\varphi(z),r)\cap \Gamma)}(f(a)-f(z))d\mathcal{C}^h(a)\Big|\leq \frac{\mathcal{C}^h(P_{\mathbb V}(B(0,1)))}{\mathcal{C}^h(P_{\mathbb V}(B(z\cdot\varphi(z),r)\cap\Gamma))}\int_{P_{\mathbb V}(B(z\cdot\varphi(z),r)\cap\Gamma)}|f(a)-f(z)|d\mathcal{C}^h(a),
\end{equation}
for any $z\in A$ and where, in order to obtain the above inequality, we used the fact that
$$
\mathcal{C}^h(P_{\mathbb V}(B(z\cdot\varphi(z),r)\cap\Gamma))\leq \mathcal{C}^h(P_{\mathbb V}(B(z\cdot\varphi(z),r)))=r^h\mathcal{C}^h(P_{\mathbb V}(B(0,1)).
$$
In addition to this, since $\mathcal{C}^h\llcorner\Gamma$ is supposed to be a $\mathscr{P}^c_h$-rectifiable measure, we infer by \cite[Proposition 4.9]{antonelli2020rectifiable} that for any $z\in P_\mathbb{V}(\Gamma)$ there exists a $0<\rho(z)<1$ such that the covering relation of $P_\mathbb{V}(\Gamma)$
$$\{(z,P_\mathbb{V}(B(\Phi(z),r)\cap \Gamma)):z\in P_\mathbb{V}(\Gamma)\text{ and }0<r<\rho(z)\}\},$$
is a $\mathcal{C}^h\llcorner P_{\mathbb V}(\Gamma)$-Vitali relation. Hence, since $f\in L^1_{\mathrm{loc}}(A)$, \cite[Corollary 2.9.9]{Federer1996GeometricTheory} allows to conclude that 
\begin{equation}\label{eqn:Conc3CENTR}
\lim_{\varepsilon\to 0^+}\sup\left\{\frac{\int_{P_{\mathbb V}(B(z\cdot\varphi(z),r)\cap\Gamma)}|f-f(z)|d\mathcal{C}^h}{\mathcal{C}^h(P_{\mathbb V}(B(z\cdot\varphi(z),r)\cap\Gamma))}: r<\rho(z),\diam (P_{\mathbb V}(B(z\cdot\varphi(z),r)\cap\Gamma))<\varepsilon \right\}=0,
\end{equation}
for $\mathcal{C}^h\llcorner P_{\mathbb V}(\Gamma)$-almost every $z\in  P_{\mathbb V}(\Gamma)$.
As a consequence, thanks to \eqref{eqn:Conc4CENTR} and \eqref{eqn:Conc3CENTR} we get
\begin{equation}\label{eqn:FINEFINE}
    \begin{split}
        &\limsup_{\mathfrak{r}\to 0}\Big\lvert \mathfrak{r}^{-h}\mu(B(z\cdot\varphi(z),\mathfrak{r}))-f(z)\mathfrak{r}^{-h}\mathcal{C}^h(P_\mathbb{V}(B(z\cdot\varphi(z),\mathfrak{r})\cap \Gamma))\Big\rvert\\
        &\leq \limsup_{\mathfrak{r}\to 0} \frac{\mathcal{C}^h(P_{\mathbb V}(B(0,1)))}{\mathcal{C}^h(P_{\mathbb V}(B(z\cdot\varphi(z),\mathfrak{r})\cap\Gamma))}\int_{P_{\mathbb V}(B(z\cdot\varphi(z),\mathfrak{r})\cap\Gamma)}|f(a)-f(z)|d\mathcal{C}^h(a)=0,
    \end{split}
\end{equation}
for $\mathcal{C}^h\llcorner P_{\mathbb V}(\Gamma)$-almost every $z\in P_{\mathbb V}(\Gamma)$. Thanks to the absolute continuity of $\mu$ with respect to $\mathcal{C}^h\llcorner \Gamma$, and to \eqref{eqn:FINEFINE} we infer that for $\mu$-almost every $x\in \Gamma$ we have
$$
\Theta^{*,h}(\mu,x)=f(x)\limsup_{\mathfrak{r}\to 0}\mathfrak{r}^{-h}\mathcal{C}^h\llcorner \mathbb{V}(P_\mathbb{V}(B(x,\mathfrak{r})\cap \Gamma))=f(x)\mathcal{C}^h(P_{\mathbb V}(B(0,1)\cap\tau(\Gamma,x)))=1,
$$
where the last identity follows from the definition of $f$, \eqref{eqn:WhoIsTheAreaFactor?} and \cref{prop:CrucialLimit}.
\end{proof}

\begin{corollario}\label{cor:AreaFormulaMeasure}
For any $\mathscr{P}^c_h$-rectifiable measure $\phi$ there are countably many $\mathbb{V}_k\in \G_c(h)$ respectively complemented by some $\mathbb{L}_k$, and countably many pairwise disjoint elementary $\mathscr{P}_h^c$-rectifiable graphs $\Gamma_k$ with respect to $\mathbb V_k$ and $\mathbb L_k$ such that for every Borel function $\psi:\mathbb{G}\to[0,+\infty)$ we have 
\begin{equation}\label{eqn:AreaFormulaCENTR2}
\int \psi(w)d\phi(w) = \sum_{k\in\N}\int \psi(a\cdot\varphi(a))\Theta^h(\phi,x)\mathcal{A}_k( \tau(\Gamma,a\cdot\varphi(a)))d\mathcal{C}^h\llcorner \mathbb{V}_k,
\end{equation}
where $\mathcal{A}_k(\cdot)$ denotes the centered area factor with respect to the splitting $\mathbb G=\mathbb V_k\cdot\mathbb L_k$ introduced in \cref{def:areafactorCENTR}.
\end{corollario}

\begin{proof}
First of all, thanks to \cite[Theorem 3.4, Proposition 2.5 and Proposition 2.6]{antonelli2020rectifiable} there exists a $\mathcal{C}^h$-$\sigma$-finite Borel set $\Sigma$ on which $\phi$ is supported, and moreover $\phi$ is mutually absolutely continuous with respect to $\mathcal{C}^h\llcorner\Sigma$. Moreover, from \cite[Theorem 4.13]{antonelli2020rectifiable} and \cite[Theorem 3.1]{FSSC15} we conclude that
$$\phi=\Theta^h(\phi,x)\mathcal{C}^h\llcorner \Sigma.$$
Since clearly $\mathcal{C}^h\llcorner \Sigma$ is $\mathscr{P}^h_c$-rectifiable thanks to \cref{prop:Lebesuge}, we infer by \cref{prop:SplitInElementaryGraph} 
that we can find the claimed disjoint $\mathscr{P}_h^c$-elementary graphs $\Gamma_k$ covering $\mathcal{C}^h$-almost all $\Sigma$. Applying \cref{thm:AreaFormulaCENTR} to each one of the $\Gamma_k$s concludes the proof.
\end{proof}

Let us now conclude this section by giving the proof of \cref{thm:AREAINTRO1} and \cref{thm:AREAINTRO2}.

\begin{osservazione}\label{rem:AncheConIlip}
We remark that \cref{thm:AreaFormulaCENTR} holds as well if we substitute the item (i) in the definition of $\alpha$-elementary $\mathscr{P}_h^c$-rectifiable graph, see \cref{def:ElementaryPh}, with the following item
\begin{equation}
    \text{(i)*\qquad $\Gamma$ is the compact graph of an intrinsic Lipschitz function $\varphi:A\subseteq \mathbb V\to\mathbb L$, see \cref{def:iLipfunctions}.}
\end{equation}
Indeed, the proof of \cref{thm:AreaFormulaCENTR} is ultimately based on \cref{prop:mutuallyabs}, the differentiation result in \cref{prop:vitali2PD}, and \cref{prop:CrucialLimit}, which are themselves based on \cref{prop:mutuallyabs}, the differentiation results \cref{prop:vitali2PD}, \cref{prop:vitali2PDAdapted} and \cref{prop:EstimateOnProjection}. All these latter results do not specifically use the fact that $\Gamma$ is a compact $C_{\mathbb V}(\alpha)$-set for $\alpha\leq\oldep{ep:Cool}(\mathbb V,\mathbb L)$ but just two basic consequences of this: namely, the fact that $\mathbb L\cap C_{\mathbb V}(\alpha)=\{0\}$, see \cref{lemma:LCapCw=e}, and the fact that $P_\mathbb V$ is injective on $\Gamma$, see \cref{prop:ConeAndGraph}. Since we obviously have, by the very definition of the cone $C_{\mathbb V,\mathbb L}(\alpha)$, that $\mathbb L\cap C_{\mathbb V,\mathbb L}(\alpha)=\{0\}$ for every $\alpha>0$, and since we also readily have that if $\Gamma$ is an intrinisc Lipschitz graph then $P_\mathbb V$ is injective on $\Gamma$, we conclude that the same strategy of the proof can be adapted to prove \cref{thm:AreaFormulaCENTR} with the above modification of the definition of $\alpha$-elementary $\mathscr{P}_h^c$-rectifiable graph.
\end{osservazione}

\begin{proof}[Proof of \cref{thm:AREAINTRO1}]
First of all, let us notice that we can assume $A$, and hence $\Gamma$, to be compact. Indeed, since $A$ is Borel, $\Gamma$ is Borel, because it is the image of $A$ under the graph map of $\varphi$, which is a continuous injective map. Since $\mathcal{S}^h\llcorner\Gamma$ is $\mathscr{P}_h^c$-rectifiable we hence deduce that $\Gamma$ is $\mathcal{S}^h$-$\sigma$-finite, cf. \cite[Proposition 2.4 and Proposition 2.5]{antonelli2020rectifiable}. Hence we have that there exists an increasing sequence $\{E_i\}_{i\in\mathbb N}$ of compact sets such that $\chi_{E_i}\to\chi_{\Gamma}$ holds  $\mathcal{S}^h\llcorner\Gamma$-almost everywhere. From \eqref{eq:n520} we also deduce that $\chi_{P_{\mathbb V}(E_i)}\to\chi_{P_{\mathbb V}(\Gamma)}$ holds $\mathcal{S}^h\llcorner A$-almost everywhere. Hence if we know the \cref{thm:AREAINTRO1} to be true for each $E_i$ we are done by monotone convergence.
 
Hence we now prove \cref{thm:AREAINTRO1} assuming $A$, and then $\Gamma$, to be compact. Taking into account \cref{rem:AncheConIlip}, we observe that the proof of \cref{thm:AREAINTRO1} is concluded if we have in addition that $\mathcal{A}(\mathbb V(x))$ is uniformly bounded above for $\mathcal{S}^h\llcorner\Gamma$-almost every $x\in\mathbb G$, since it is the only reamining hypothesis to verify in order to apply the modified version of \cref{thm:AreaFormulaCENTR} discussed in \cref{rem:AncheConIlip}. But since $a\to\mathcal{A}(\mathbb V(a\cdot\varphi(a)))$ is $\mathcal{C}^h\llcorner\mathbb V$-measurable, see the first part of the proof of \cref{thm:AreaFormulaCENTR}, we can divide the set $A$ into countably many disjoint measurable subsets where $\mathcal{A}(\cdot)$ is uniformly bounded above. Hence, by approximating each of these countably many pieces from the inside with compact sets as explained at the beginning of this proof, and applying \cref{thm:AreaFormulaCENTR} as discussed in \cref{rem:AncheConIlip}, we conclude by approximation that \cref{thm:AreaFormulaCENTR} holds on each piece of the latter disjoint union. Then summing together finishes the proof of \cref{thm:AREAINTRO1}.
\end{proof}
\begin{proof}[Proof of \cref{thm:AREAINTRO2}]
It is a consequence of \cref{thm:AREAINTRO1} and \cref{prop:rett.1}, taking into account that the intrinsic differentiability implies the hypothesis of \cref{prop:rett.1} (cf. \cref{prop:idiffapproximatetangent}) from which we get that $\mathcal{S}^h\llcorner\Gamma$ is $\mathscr{P}_h^c$-rectifiable.
\end{proof}

\section{Applications}

In this section we provide some applications of the rectifiability criterion proved in \cref{prop:rett.1}, which was at the core of the proof of 3. $\Rightarrow$ 1. of \cref{thm:INTRO1Equivalence}. Let us first recall the definition of $C^1_{\mathrm H}$-function.
\begin{definizione}[$C^1_{\mathrm H}$-function]\label{def:C1h}
Let $\mathbb G$ and $\mathbb G'$ be two Carnot groups endowed with left-invariant homogeneous distances $d$ and $d'$, respectively. Let $\Omega\subseteq \mathbb G$ be open and let $f:\Omega\to\mathbb G'$ be a function. We say that $f$ is {\em Pansu differentiable at $x\in\Omega$} if there exists a homogeneous homomorphism $df_x:\mathbb G\to\mathbb G'$ such that
$$
\lim_{y\to x}\frac{d'(f(x)^{-1}\cdot f(y),df_x(x^{-1}\cdot y))}{d(x,y)}=0.
$$
Moreover we say that $f$ is {\em of class $C^1_{\mathrm H}$ in $\Omega$} if the map $x\mapsto df_x$ is continuous from $\Omega$ to the space of homogeneous homomorphisms from $\mathbb G$ to $\mathbb G'$.
\end{definizione}

\begin{proposizione}\label{structure:liplevelsets}
Let $B$ be a Borel set in $\mathbb{G}$ and suppose $\mathbb{H}$ is a Carnot group of homogeneous dimension $\mathcal{Q}^\prime$ with $\mathcal{Q}\geq\mathcal{Q}^\prime$. Let $f:B\subseteq \mathbb{G}\to \mathbb{H}$ be a Lipschitz map such that
\begin{equation}
    \text{$\mathrm{Ker}(df(x))\in \G_c(\mathbb{G})$ for $\mathcal{S}^{\mathcal{Q}}$-almost every $x\in \{z\in B:df(z) \text{ exists and is surjective}\}$,}
    \label{condizionesurj}
\end{equation}
where $\G_c(\mathbb G)$ denotes the set of complemented homogeneous subgroups in $\mathbb G$.
Then, for $\mathcal{S}^{\mathcal{Q}^\prime}$-almost every $y\in f(B)$, the following holds. For $\mathcal{S}^{\mathcal{Q}-\mathcal{Q}^\prime}$-almost every $x\in f^{-1}(y)$ and any $0<\beta<1$ there exists a $\rho(x,\beta)>0$ such that 
$$f^{-1}(y)\cap B(x,\rho(x,\beta))\subseteq xC_{\mathrm{Ker}(df(x))}(\beta).$$
In particular the measure $\mathcal{S}^{\mathcal{Q}-\mathcal{Q}^\prime}\llcorner f^{-1}(y)$ is $\mathscr{P}^c_{\mathcal{Q}-\mathcal{Q}^\prime}$-rectifiable in $\mathbb{G}$ and $$\Tan_{\mathcal{Q}-\mathcal{Q}^\prime}(\mathcal{S}^{\mathcal{Q}-\mathcal{Q}^\prime}\llcorner f^{-1}(y),x)\subseteq\{\lambda\mathcal{C}^{\mathcal{Q}-\mathcal{Q}^\prime}\llcorner \mathrm{Ker}(df(x)):\lambda> 0\}\quad\text{for $\mathcal{S}^{\mathcal{Q}-\mathcal{Q}^\prime}$-almost every $x\in f^{-1}(y)$.}$$
\end{proposizione}

\begin{proof}
Without loss of generality we can assume that $B$ is a bounded Borel set. Thanks to \cite[Equation (9) and Proposition 1.12]{Magnanicoarea}, since $B$ is bounded and $f$ is Lipschitz on $B$, one infers that for $\mathcal{S}^{\mathcal{Q}^\prime}$-almost every $y\in f(B)$ we have $\mathcal{S}^{\mathcal{Q}-\mathcal{Q}^\prime}(f^{-1}(y))<\infty$. Moreover, from \cite[Theorem 2.6 and Theorem 2.7]{Magnanicoarea} we have
\begin{equation}
 \mathcal{S}^{\mathcal{Q}-\mathcal{Q}^\prime}(f^{-1}(y)\cap N)=0, \qquad\text{and}\qquad\mathcal{S}^{\mathcal{Q}-\mathcal{Q}^\prime}(f^{-1}(y)\cap N^\prime)=0,
    \label{cond:num1}
\end{equation}
where $N$ is the $\mathcal{S}^{\mathcal{Q}}$-null set on which the Pansu differential $df(x)$ does not exists and $N^\prime$ is the $\mathcal{S}^\mathcal{Q}$-measurable set where $df(x)$ is not surjective. In addition to this, thanks to \eqref{condizionesurj}, we know that
$$\mathcal{S}^{\mathcal{Q}}(\{z\in B\setminus (N\cup N^{\prime}):\mathrm{Ker}(df(x))\not\in \G_c(\mathbb{G})\})=0,$$
and thus, \cite[Theorem 2.6]{Magnanicoarea} implies that for $\mathcal{S}^{\mathcal{Q}^\prime}$-almost every $y\in f(B)$
we have
\begin{equation}
\mathcal{S}^{\mathcal{Q}-\mathcal{Q}^\prime}(\{z\in f^{-1}(y): df(z) \text{ exists, is surjective and }\mathrm{Ker}(df(z))\not\in \G_c(\mathbb{G})\})=0.
    \label{cond:num2}
\end{equation}
We can further refine the above condition thanks to the following observation. Since for $\mathcal{S}^{\mathcal{Q}^\prime}$-almost every $y\in f(B)$ and for $\mathcal{S}^{\mathcal{Q}-\mathcal{Q}^\prime}$-almost every $z\in f^{-1}(y)$ the Pansu differential $df(z)$ exists and is surjective, then by the first homomorphism theorem $\mathbb{G}/\mathrm{Ker}(df(z))\cong \mathbb{H}$ and in particular
the subgroup $\mathrm{Ker}(df(z))$ must have homogeneous dimension $\mathcal{Q}-\mathcal{Q}^\prime$. 

Therefore, throughout the rest of the proof, we fix a $y\in f(B)$ such that  $\mathcal{S}^{\mathcal{Q}-\mathcal{Q}^\prime}(f^{-1}(y))<\infty$, and recall that for $\mathcal{S}^{\mathcal{Q}-\mathcal{Q}^\prime}$-almost every $z\in f^{-1}(y)$ the Pansu differential $df(z)$ exists, is surjective and $\mathrm{Ker}(df(z))\in \G_c(\mathcal{Q}-\mathcal{Q}^\prime)$.

In order to conclude the proof we show that the hypothesis of \cref{prop:rett.1} is satisfied. Fix a point $x\in f^{-1}(y)$ such that $df(x)$ exists, is surjective and $\mathrm{Ker}(df(x))\in \G_c(\mathcal{Q}-\mathcal{Q}^\prime)$. Notice that thanks to what we proved above, such $x$ can be chosen in a set of $\mathcal{S}^{\mathcal{Q}-\mathcal{Q}'}$-full measure in $f^{-1}(y)$. Let us note that for any $\varepsilon>0$ there exists an $\eta:=\eta(x,\varepsilon)>0$ such that for any $w\in B(x,\eta)$ we have  
$$
d_{\mathbb{H}}(f(x)df(x)[x^{-1}w],f(w))\leq \varepsilon d_{\mathbb{G}}(w,x).
$$
This in particular implies that for any $w\in B(x,\eta)\cap f^{-1}(y)$ we have
\begin{equation}
    \begin{split}
        0&=d_{\mathbb{H}}(f(w),f(x))\geq d_{\mathbb{H}}(f(x)df(x)[x^{-1}w],f(x))-d_{\mathbb{H}}(f(x)df(x)[x^{-1}w],f(w))\\
        &\geq d_{\mathbb{H}}(f(x)df(x)[x^{-1}w],f(x))-\varepsilon d_{\mathbb{G}}(w,x),
        \nonumber
    \end{split}
\end{equation}
implying that
$$
\lVert df(x)[P_{\mathbb{L}(x)}(x^{-1}w)]\rVert_{\mathbb{H}}=\lVert df(x)[x^{-1}w]\rVert_{\mathbb{H}}\leq \varepsilon\lVert x^{-1}w\rVert_\mathbb{G},
$$
where $\mathbb{L}(x)$ is a complementary subgroup of $\mathrm{Ker}(df(x))$ and $P_{\mathbb{L}(x)}$ is the splitting projection on $\mathbb{L}(x)$ associated to the split $\mathrm{Ker}(df(x))\cdot \mathbb{L}(x)$. Thanks to a standard compactness argument, it is not hard to see that there exists a constant $C(x)>0$ such that $\lVert df(x)[P_{\mathbb{L}(x)}(x^{-1}w)]\rVert_{\mathbb H}\geq C(x)\lVert P_{\mathbb{L}(x)}(x^{-1}w)\rVert_{\mathbb G}$ for every $w\in B(x,\eta)$, and thus 
$$
\mathrm{dist}_\mathbb{G}(x^{-1}w, \mathrm{Ker}(df(x)))\leq \lVert P_{\mathbb{L}(x)}(x^{-1}w)\rVert_\mathbb{G}\leq C(x)^{-1}\varepsilon \lVert x^{-1}w\rVert_\mathbb{G},
$$
for every $w\in B(x,\eta)\cap f^{-1}(y)$, proving that $f^{-1}(y)\cap B(x,\eta)\subseteq xC_{\mathrm{Ker}(df(x))}(C(x)^{-1}\varepsilon)\cap B(x,\eta)$. Since we fall in the hypothesis of \cref{prop:rett.1}, the proof of the proposition is achieved.
\end{proof}

Before stating the following corollaries of \cref{structure:liplevelsets}, let us recall the definition of $(\mathbb G,\mathbb G')$-rectifiable sets.
\begin{definizione}[$C^1_{\mathrm H}$-submanifold]\label{def:C1Hmanifold}
Given an arbitrary Carnot group $\mathbb G$, we say that $\Sigma\subseteq \mathbb G$ is a {\em $C^1_{\mathrm H}$-submanifold} of $\mathbb G$ if there exists a Carnot group $\mathbb G'$ such that for every $p\in \Sigma$ there exists an open neighborhood $\Omega$ of $p$ and a function $f\in C^1_{\mathrm H}(\Omega;\mathbb G')$ such that 
\begin{equation}\label{eqn:RepresentationOfSigma}
\Sigma\cap \Omega =\{g\in\Omega:f(g)=0\},
\end{equation}
and $df_p:\mathbb G\to\mathbb G'$ is surjective with $\mathrm{Ker}(df_p)$ complemented. In this case we say that $\Sigma$ is a {\em $C^1_{\mathrm H}(\mathbb G,\mathbb G')$-submanifold}.
\end{definizione}
\begin{definizione}[($\mathbb G,\mathbb G')$-rectifiable set]\label{def:GG'Rect}
Given two arbitrary Carnot groups $\mathbb G$ and $\mathbb G'$ of homogeneous dimensions $Q$ and $Q'$, respectively, we say that $\Sigma\subseteq \mathbb G$ is a {\em $(\mathbb G,\mathbb G')$-rectifiable set} if there exist countably many subsets $\Sigma_i$ of $\mathbb G$ that are $C^1_{\mathrm H}(\mathbb G,\mathbb G')$-submanifolds, such that 
$$
\mathcal{H}^{Q-Q'}\left(\Sigma\setminus\bigcup_{i=1}^{+\infty}\Sigma_i\right)=0.
$$
\end{definizione}

\begin{corollario}\label{structure:liplevelsets2}
Suppose $\mathcal{Q}\geq m$, let $B\subseteq \mathbb G$ a Borel subset and $f:B\to \R^m$ be a Lipschitz map such that
\begin{equation}
    \text{$\mathrm{Ker}(df(x))\in \G_c(\mathbb{G})$ for $\mathcal{S}^{\mathcal{Q}}$-almost every $x\in \{z\in B:df(z) \text{ exists and is surjective}\}$.}
    \label{condizionesurjeu}
\end{equation}
Then, $m\leq \mathrm{dim}(V_1)$ and for $\mathcal{S}^m$-almost every $y\in f(B)$ the set $f^{-1}(y)$ is $C^1_H(\mathbb{G},\R^m)$-rectifiable.
\end{corollario}

\begin{proof}
First of all, let us note that \cref{structure:liplevelsets} immediately implies that for $\mathcal{S}^m$-almost every $y\in f(B)$, the measure $\mathcal{S}^{\mathcal{Q}-m}\llcorner f^{-1}(y)$ is $\mathscr{P}^c_{\mathcal{Q}-m}$-rectifiable. 

A necessary step to conclude the proof is to investigate further the structure of $\mathrm{Ker}(df(x))$ whenever it exists. In order to do so, we fix a point where $df(x)$ exists and note that for any $v\in\mathbb{G}$ we have $df(x)[\delta_\lambda(v)]=\lambda df(x)[v]$ for any $\lambda>0$. Since, thanks to the identification through $\exp$ of $\mathbb{G}$ with its Lie algebra, $df(x)$ can be expressed as a matrix, thus we conclude that for every $v\in V_j$ with $2\leq j\leq \kappa$, where $\kappa$ is the step of the group, we have
$$ 
df(x)[v]=\lim_{\lambda\to 0}\frac{df(x)[\delta_\lambda(v)]}{\lambda}=\lim_{\lambda\to 0}\frac{\lambda^j df(x)[v]}{\lambda}=0.
$$
So, $\mathrm{Ker}(df(x))$ is a normal subgroup containing $V_2\oplus\ldots\oplus V_\kappa$ and on the points where $df(x)$ is surjective we must have 
$$
\dim_{\mathrm{hom}}(\mathrm{Ker}(df(x)))=\mathcal{Q}-m\geq\dim_{\mathrm{hom}}(V_2\oplus\ldots\oplus V_k)=Q-\mathrm{dim}(V_1),
$$
proving that we have $m\leq \mathrm{dim}(V_1)$.

Throughout the rest of the proof we fix an $y\in f(B)$ such that $f^{-1}(y)$ is a $\mathscr{P}^c_{\mathcal{Q}-m}$-rectifiable measure. In addition to this, since $f^{-1}(y)$ is closed we may as well assume without loss of generality that $f^{-1}(y)$ is compact with finite $\mathcal{S}^{\mathcal{Q}-m}$-finite measure since the class of $\mathscr{P}^c_{\mathcal{Q}-m}$-rectifiable measures is closed under restriction to a Borel subset by splitting $f^{-1}(y)$ as in \cite[Proposition 2.4, Proposition 2.5 and Proposition 2.6]{antonelli2020rectifiable}.

Thanks to \cref{lem:BorelTangents} applied to the measure $\mathcal{S}^{\mathcal{Q}-m}\llcorner f^{-1}(y)$, and thanks to \cref{structure:liplevelsets}, we know that the map $x\mapsto \mathrm{Ker}(df(x))\in \G_c(\mathcal{Q}-m)$ is $\mathcal{S}^{\mathcal{Q}-m}\llcorner f^{-1}(y)$-measurable. Calling $\{e_i\}_{i=1}^n$ a basis of $\mathfrak g$, the latter observiation in concjunction with \cref{prop:projmap} yields that for any $i=1,\ldots,n$ the vector fields $x\mapsto\Pi_{\mathrm{Ker}(df(x))^\perp}[e_i]=:v_i(x)$ are $\mathcal{S}^{\mathcal{Q}-m}\llcorner f^{-1}(y)$-measurable. Moreover, thanks to the above discussion, according to which $\mathrm{Ker}(df(x))\supseteq V_2\oplus\dots\oplus V_\kappa$, each $v_i(x)$ belong $\mathcal{S}^{\mathcal{Q}-m}\llcorner f^{-1}(y)$-almost everywhere to $GV_1$.

For any $\varepsilon>0$ thanks to Lusin's theorem and the Borel regularity of the measure $\mathcal{S}^{\mathcal{Q}-m}\llcorner f^{-1}(y)$, we can find a compact set $K_{\varepsilon}\subseteq\mathbb{G}$ such that $\mathcal{S}^{\mathcal{Q}-m}( f^{-1}(y)\setminus K_{\varepsilon})\leq \varepsilon \mathcal{S}^{\mathcal{Q}-m} (f^{-1}(y))$ and the vector fields $v_i(x)$ are continuous on $K_\varepsilon$ for any $i=1,\ldots,n$. In particular we can split $K_\varepsilon$ into a finite partition $\{K_{\varepsilon}^I:I=(i_1,\ldots,i_m)\in\{1,\ldots,n\}^m\}$ of Borel subsets on which the vector fields
$v_{i_1}(x),\ldots,v_{i_m}(x)$ are a basis for $\mathrm{Ker}(df(x))^\perp$.

In the following we will show that for any choice 
of $I \in\{1,\ldots,n\}^m$ such that $\mathcal{S}^{\mathcal{Q}-m}\llcorner f^{-1}(y)(K_{\varepsilon}^I)>0$ and of  $\vartheta,\gamma\in \N$, the set $E(\vartheta,\gamma)$ relative to the measure $\mathcal{S}^{\mathcal{Q}-m}\llcorner K_\varepsilon^I$, introduced in \cref{def:EThetaGamma}, can be covered $\mathcal{S}^{\mathcal{Q}-m}$-almost all with countably many $C^1_H(\mathbb G,\mathbb R^m)$-rectifiable sets. This would imply that $K_\varepsilon^I$ can be covered $\mathcal{S}^{\mathcal{Q}-m}$-almost all with $C^1_H(\mathbb G,\mathbb R^m)$-rectifiable sets thanks to \cref{prop::E} and thus so can $K_\varepsilon$ thanks to the finiteness of the family $I$. Finally, the arbitrariness of $\varepsilon$ would conclude the proof of the proposition. 

For any $j=1,\ldots,m$ we let
$$\rho_{j,\delta}(x):=\sup\bigg\{ \frac{\lvert\langle v_{i_j}(x),\pi_1(x^{-1}z)\rangle\rvert}{\lVert x^{-1}z\rVert}:z\in E(\vartheta,\gamma)\text{ and }\lVert x^{-1}z\rVert\leq \delta\bigg\},$$
and we claim that for any $j=1,\ldots,m$ we have $\lim_{\delta\to 0}\rho_{j,\delta}(x)=0$ for any $x\in E(\vartheta,\gamma)$ at which for any $\beta>0$ there exists a $\rho=\rho(x,\beta)>0$ such that 
\begin{equation}
    f^{-1}(y)\cap B(x,\rho)\subseteq xC_{\mathrm{Ker}(df(x))}(\beta).
    \label{eq:num:1}
\end{equation}
Note that thanks to \cref{structure:liplevelsets}, the above condition \eqref{eq:num:1} is satisfied for $\mathcal{S}^{\mathcal{Q}-m}$-almost every $x\in E(\vartheta,\gamma)$. We remark that the functions $\rho_{j,\delta}$ are measurable for any $i\in\N$ and $\delta>0$. Indeed, on the one hand the function $(x,z)\mapsto\lvert \langle v_{i_j}(x),\pi_1(x^{-1}z)\rangle\rvert/d(x,z)$ is $\mathcal{S}^{\mathcal{Q}-m}\llcorner f^{-1}(y)$-measurable since it is the quotient of two $\mathcal{S}^{\mathcal{Q}-m}\llcorner f^{-1}(y)$-measurable functions. On the other, since $\mathbb{G}$ is separable, it is immediate to see that $\rho_{j,\delta}$ can be rewritten as the supremum on $z$ over a countable subset of $E(\vartheta,\gamma)\cap B(x,\delta)$, showing that $\rho_{j,\delta}$ is indeed measurable. Thanks to \cite[Proposition 1.5]{MarstrandMattila20}, we know that at any $x\in E(\vartheta,\gamma)$ where \eqref{eq:num:1} is satisfied for some $\beta$ and some $\rho>0$, we have
\begin{equation}
\begin{split}
        \frac{\lvert\langle v_{i_j}(x),\pi_1(x^{-1}z)\rangle\rvert}{\lVert x^{-1}z\rVert}=\frac{\dist(x^{-1}z,\mathbb{V}(v_{i_j}(x)))}{\lVert x^{-1}z\rVert}\leq \frac{\dist(x^{-1}z,\mathrm{Ker}(df(x)))}{\lVert x^{-1}z\rVert}\leq \beta,
        \label{eq:num:2}
\end{split}
\end{equation}
for any $z\in B(x,\rho)\cap f^{-1}(y)$, where $\mathbb{V}(v_{i_j}(x))$ is the $1$-codimensional homogeneous subgroup orthogonal (in the Euclidean sense) to the vector $v_{i_j}(x)$, and where the second last inequality above comes from the fact that $\mathrm{Ker}(df(x))$ is contained in $\mathbb{V}(v_{i_j}(x))$.
The bound \eqref{eq:num:2} together with \cref{structure:liplevelsets} conclude that $\lim_{\delta\to 0}\rho_{j,\delta}(x)=0$ for $\mathcal{S}^{Q-m}\llcorner E(\vartheta,\gamma)$-almost every $x\in\mathbb{G}$. Thanks to Severini-Egoroff's theorem for any $\tilde \varepsilon>0$ we can find a compact set $\tilde K_{\tilde \varepsilon}$ inside $E(\vartheta,\gamma)$ such that
\begin{enumerate}
    \item $\mathcal{S}^{\mathcal{Q}-m}( E(\vartheta,\gamma)\setminus \tilde K_{\tilde \varepsilon})\leq \tilde \varepsilon\mathcal{S}^{\mathcal{Q}-m}( E(\vartheta,\gamma))$,
    \item $v_{i_j}(x)$ is continuous on $\tilde K_{\tilde \varepsilon}$ for any $j=1,\ldots,m$,
    \item $\rho_{j,\delta}$ converges uniformly to $0$ on $\tilde K_{\tilde \varepsilon}$ for any $j=1,\ldots,m$.
\end{enumerate}
Thanks to Whitney extension theorem, see for instance \cite[Theorem 5.2]{step2}, we infer that we can find $m$ $C^1_{\mathrm{H}}$-functions defined on all of $\mathbb G$ such that $f_{j,\tilde \varepsilon}\vert_{\tilde K_{\tilde \varepsilon}}=0$ and $\nabla_\HH f_{j,\tilde \varepsilon}(x)=v_{i_j}(x)$ for any $x\in \widetilde K_{\tilde \varepsilon}$. This shows that, thanks to arbitrariness of $\tilde \varepsilon$ and to the fact that the $v_{i_j}$'s are independent everywhere on $E(\vartheta,\gamma)$, the set $E(\vartheta,\gamma)$ can be covered $\mathcal{S}^{\mathcal{Q}-m}$-almost all by the $0$-level set of countably many $C^1_H(\mathbb{G},\R^m)$-maps. Thus the proof is concluded.
\end{proof}

We end this section with some consequences of the previous \cref{structure:liplevelsets2}. The first part of the forthcoming corollary follows verbatim from the second part of the proof above; while the second part of the forthcoming corollary is a byproduct of the first part in conjunction with \cite[Proposition 6.2]{antonelli2020rectifiable}.

\begin{corollario}\label{cor:PhCCoorizzontali}
Let $\mathbb G$ be a Carnot group of homogeneous dimension $Q$, and let $1\leq h\leq Q$ be a natural number. Let $\Gamma\subseteq \mathbb G$ be a compact set such that $\mathcal{S}^h(\Gamma)<+\infty$. The following are equivalent 
\begin{enumerate}
    \item[1.] $\mathcal{S}^h\llcorner\Gamma$ is a $\mathscr{P}_h^c$-rectifiable measure, and at $\mathcal{S}^h\llcorner\Gamma$-almost every $x\in\mathbb G$ the tangent plane is complemented by a \textbf{horizontal} subgroup.
    \item[2.] $\Gamma$ is $C^1_{\mathrm{H}}(\mathbb G,\mathbb R^{Q-h})$-rectifiable.
\end{enumerate}
\end{corollario}
If any of the previous holds we have $Q-h\leq \mathrm{dim}(V_1)$.
\begin{proof}
1. $\Rightarrow$ 2. is obtained arguing precisely as in the second part of the proof of \cref{structure:liplevelsets2}. 2. $\Rightarrow$ 1. is an immediate consequence of \cite[Proposition 6.2]{antonelli2020rectifiable}
\end{proof}

\begin{corollario}\label{structure:liplevelsets3}
Let $B\subseteq \mathbb G$ be a Borel subset. Suppose $f:B\to \R$ is a Lipschitz map. Then, for $\mathcal{S}^1$-almost every $y\in f(B)$ the set $f^{-1}(y)$ is $C^1_{\mathrm{H}}$-rectifiable.
\end{corollario}

\begin{proof}
Let us assume that $x$ is point where the Pansu's differential $df(x)$ exists and is surjective. Thanks to the first homomorphism theorem we know that $\mathbb{G}/\mathrm{Ker}(df(z))\cong \R$ and in particular $\mathrm{Ker}(df(z))$ is a $1$-codimensional homogeneous subgroup of $\mathbb{G}$. These subgroups are always complemented, and thus \cref{structure:liplevelsets2} concludes the proof.
\end{proof}

\begin{corollario}\label{structurelip4}
Let $B\subseteq \mathbb G$ be a Borel subset. Suppose $f:B\to\R^m$ is a Lipschtiz function with $\mathcal{Q}\geq m$. Then for $\mathcal{S}^m$-almost every $y\in f(B)$ there are $m$ $C^1_H$-rectifiable sets $\Gamma_i(y)$ such that
$$f^{-1}(y)=\bigcap_{i=1}^m \Gamma_i(y).$$
\end{corollario}

\begin{proof}
Since $f(z)=y$ if and only if for any $i=1,\ldots,m$ we have $f_i(z)=y_i$, the claim immediately follows thanks to \cref{structure:liplevelsets2} and \cref{structure:liplevelsets3}.
\end{proof}

\appendix


\printbibliography

\section*{Acknowledgments} The first author is partially supported by the European Research Council (ERC Starting Grant 713998 GeoMeG `\emph{Geometry of Metric Groups}'). The second author is supported by the Simons Foundation Wave Project.

\end{document}

%% file: structure.tex
\usepackage[
nochapters, 
pdfspacing, 
dottedtoc 
]{classicthesis} 
\usepackage{amsopn}
\usepackage[T1]{fontenc} 
\usepackage{multirow}
\usepackage[utf8]{inputenc} 
\usepackage{hhline}
\usepackage{graphicx} 
\graphicspath{{Figures/}} 
\usepackage{indentfirst}
\usepackage{enumitem} 
\usepackage{savesym}
\usepackage{mathrsfs}
\usepackage{geometry}
\usepackage{esint}
\usepackage{longtable}

\usepackage[
    backend=biber,
    style=numeric,
    natbib=false,
    url=false, 
    doi=false,
    eprint=false,
    maxnames=50
]{biblatex}

\usepackage{subfig} 

\usepackage{amsmath, amssymb,amsthm,amsfonts} 

\usepackage[english,capitalize]{cleveref}

\usepackage{thmtools}
\usepackage{thm-restate}

\usepackage{hyperref}
\newcommand{\vertiii}[1]{{\left\vert\kern-0.25ex\left\vert\kern-0.25ex\left\vert #1 
    \right\vert\kern-0.25ex\right\vert\kern-0.25ex\right\vert}}

\theoremstyle{plain}
\newtheorem{teorema}{Theorem}[section]
\newtheorem{proposizione}[teorema]{Proposition}
\newtheorem{lemma}[teorema]{Lemma}
\newtheorem{corollario}[teorema]{Corollary}
\newtheorem*{theorem*}{Theorem}

\theoremstyle{definition}
\newtheorem{definizione}{Definition}[section]

\theoremstyle{remark}
\newtheorem{osservazione}{Remark}[section]

\newcommand{\Tan}{\mathrm{Tan}}
\newcommand{\N}{\mathbb{N}}

\newcommand{\Q}{\mathbb{Q}}
\newcommand{\R}{\mathbb{R}}

\newcommand{\HH}{\mathbb{H}}

\newcommand{\G}{Gr}
\newcommand{\res}

\newcounter{const}
\newcommand{\newC}{\refstepcounter{const}\ensuremath{C_{\theconst}}}
\newcommand{\oldC}[1]{\ensuremath{C_{\ref{#1}}}}

\newcounter{eps}
\newcommand{\newep}{\refstepcounter{eps}\ensuremath{\varepsilon_{\theeps}}}
\newcommand{\oldep}[1]{\ensuremath{\varepsilon_{\ref{#1}}}}

\DeclareMathOperator*{\lip}{Lip_1^+}

\DeclareMathOperator*{\supp}{supp}

\DeclareMathOperator*{\diam}{diam}

\DeclareMathOperator*{\dist}{dist}

\hypersetup{
colorlinks=true, breaklinks=true,bookmarksnumbered,
urlcolor=webbrown, linkcolor=RoyalBlue, citecolor=webgreen, 
pdftitle={}, 
pdfauthor={\textcopyright}, 
pdfsubject={}, 
pdfkeywords={}, 
pdfcreator={pdfLaTeX}, 
pdfproducer={LaTeX with hyperref and ClassicThesis} 
}